\documentclass[10pt]{amsart}
\pdfoutput=1
\usepackage{ mathrsfs }
\usepackage{amsmath, amsthm, amssymb,slashed,stmaryrd}
\usepackage{ifpdf}
\usepackage[pdftex]{graphicx}
\usepackage{tikz}
\usetikzlibrary{matrix,arrows,calc}

\usepackage[pdftex,plainpages=false,hypertexnames=false,pdfpagelabels]{hyperref}
\long\def\todo#1{{\color{red} {#1}}}
\long\def\stodo#1{{\color{blue} {#1}}}

 \theoremstyle{plain}
 \newtheorem{thm}{Theorem}[section]
    
 \newtheorem{cor}[thm]{Corollary}
 \newtheorem{lem}[thm]{Lemma}
 \newtheorem{prop}[thm]{Proposition}
 
 \theoremstyle{definition}
 \newtheorem{defn}[thm]{Definition}
 \newtheorem{notation}[thm]{Notation}
 \newtheorem{ex}[thm]{Example}
 \newtheorem{constr}[thm]{Construction} 
 \newtheorem*{thm*}{Theorem}
 \theoremstyle{remark}
 \newtheorem{rmk}[thm]{Remark}

\def\beq{\begin{eqnarray}}
\def\eeq{\end{eqnarray}}
 \newcommand{\bp}{\begin{proof}[Proof]}
 \newcommand{\ep}{\end{proof}}

\DeclareSymbolFont{bbold}{U}{bbold}{m}{n}
\DeclareSymbolFontAlphabet{\mathbbold}{bbold}

\def\Ell{{\rm Ell}}
\def\bS{{\mathbb{S}}}
\def\dEll{\widehat{\rm Ell}{}}
\def\TMF{{\rm TMF}}

\def\Eu{{\rm Eu}}
\def\Th{{\rm Th}}

\def\MO{{\rm MO}}
\def\MString{{\rm MString}}
\def\O{{\mathcal{O}}}
\def\H{{\rm H}}
\def\HH{{\mathbb H}}
\def\Spin{{\rm Spin}}
\def\U{{\rm U}}
\def\SU{{\rm SU}}
\def\SO{{\rm SO}}
\def\BSO{{\rm BSO}}
\def\BU{{\rm BU}}
\def\BO{{\rm BO}}
\def\BSU{{\rm BSU}}

\def\Ad{{\rm Ad}}
\def\Map{{\sf Map}}

\def\MU{{\rm MU}}
\def\Mell{\mathcal{M}_{\rm ell}}
\def\che{{\sf c}}

\def\EE{{\mathcal{E}}}
\def\Euni{\mathcal{E}}

\def\pt{{\rm pt}}

\def\Sym{{\rm Sym}}
\def\Poly{{\rm Poly}}

\def\MF{{\rm MF}}
\def\JF{{\rm JF}}
\def\vol{{\rm vol}}

\def\Pic{{\sf Pic}}
\def\ev{{\rm ev}}
\def\odd{{\rm odd}}

\def\ch{{\rm ch}}
\def\Wit{{\rm Wit}}

\def\R{{\mathbb{R}}}
\def\fg{{\mathfrak{g}}}
\def\M{{\mathbb{M}}}
\def\E{{\mathbb{E}}}

\def\CP{{\mathbb{CP}}}

\def\id{{{\rm id}}}
\def\K{{\rm {K}}}
\def\C{{\mathbb{C}}}

\def\Z{{\mathbb{Z}}}
\def\X{{\mathcal{X}}}
\def\F{{\mathcal{F}}}

\def\Bun{{\sf Bun}}

\def\SL{{\rm SL}}

\def\Hom{{\sf Hom}}
\def\SM{{\sf Map}}
\def\Rep{{\rm Rep}}
\def\Spec{{\rm Spec}}

\def\Lat{{\sf Lat}}
\def\cL{{\mathcal{F}}}

\newcommand{\op}{{\sf{op}}}   
\newcommand{\sq}{/}
\newcommand{\cq}{/^c}
\newcommand{\nsq}{{\sq^{\!\nabla}}\!}

\def\downin{\ensuremath{\rotatebox[origin=c]{90}{$\in$}}}

\newcommand\nc{\newcommand}
\nc\mf\mathfrak
\nc\mc\mathcal
\nc\mb\mathbb

\begin{document}

\title[A de Rham model for equivariant elliptic cohomology]{A de Rham model for complex analytic equivariant elliptic cohomology}

\author{Daniel Berwick-Evans and Arnav Tripathy}

\date{\today}

\begin{abstract}
We construct a cocycle model for complex analytic equivariant elliptic cohomology that refines Grojnowski's theory when the group is connected and Devoto's when the group is finite. We then construct Mathai--Quillen type cocycles for equivariant elliptic Euler and Thom classes, explaining how these are related to positive energy representations of loop groups. Finally, we show that these classes give a unique complex analytic equivariant refinement of Hopkins' ``theorem of the cube" construction of the $\MString$-orientation of elliptic cohomology. 
\end{abstract}

\maketitle 
\setcounter{tocdepth}{1}
\tableofcontents

\section{Introduction}\label{sec:conv}


Equivariant K-theory facilitates a rich interplay between representation theory and topology. For example, universal Thom classes come from representations of spin groups; power operations are controlled by the representation theory of symmetric groups; and
the equivariant index theorem permits geometric constructions of representations of Lie groups. 

\emph{Equivariant elliptic cohomology} is expected to lead to an even deeper symbiosis between representation theory and topology. First evidence appears in the visionary work of Grojnowski~\cite{Grojnowski} and Devoto~\cite{DevotoII}. 
Grojnowski's complex analytic equivariant elliptic cohomology (defined for connected Lie groups) makes contact with positive energy representations of loop groups~\cite{Ando,GanterEllipticWCF}. 
Devoto's construction (defined for finite groups) interacts with moonshine phenomena~\cite{BakerThomas,GanterHecke,Morava}. 

Equivariant elliptic cohomology over the complex numbers is already a deep object. By analogy, equivariant K-theory with complex coefficients subsumes the character theory of compact Lie groups, which in turn faithfully encodes their representation theory. Analogously, equivariant elliptic cohomology over the complex numbers should be viewed as a home for ``elliptic character theory," although the complete picture of what elliptic representation theory really \emph{is} remains an open question~\cite{Segal_Elliptic,GKV,HKR,BZN1}. 


This paper gives a cocycle model for complex analytic equivariant elliptic cohomology as a sheaf of commutative differential graded algebras on the moduli space of $G$-bundles over elliptic curves. The approach is uniform in the group~$G$. When $G$ is connected, we recover a cocycle model for Grojnowski's equivariant elliptic cohomology, and when $G$ is finite we recover a cocycle model for Devoto's. One great utility of cocycle models is that they bring new computational tools for applications. 
The elliptic cocycles presented below are concrete and explicit, namely compatible equivariant differential forms on certain fixed point sets. This makes them well-suited for applications. 

One source of such applications has been long in the making. Indeed, Grojnowski's original motivation for studying equivariant elliptic cohomology was to construct certain elliptic algebras, e.g., an elliptic analog of the affine Hecke algebra. Crucially, he recognized that such algebras should arise geometrically by applying equivariant elliptic cohomology to certain varieties, such as the Steinberg variety. This is the third step in the program that produces increasingly sophisticated representation-theoretic objects by applying first ordinary equivariant cohomology, then equivariant K-theory, and next equivariant elliptic cohomology to varieties built out of algebraic groups. The cohomological and K-theoretic variants of this paradigm have already met great success, notably in Kazhdan--Lusztig's K-theoretic construction of the affine Hecke algebra~\cite{KL}. The program has seen further development in recent years with the expectation of new examples from supersymmetric gauge theory~\cite{BDGH, BDGHK}. In the corresponding mathematical theory of symplectic resolutions, the closely related work of Maulik--Okounkov~\cite{MO} constructs representations of generalized quantum groups by applying equivariant cohomology theories to Nakajima quiver varieties. Equivariant elliptic cohomology is starting to play an increasingly important role at this nexus of representation theory, geometry and physics, e.g., in the work of Zhao--Zhong~\cite{ZhaoZhong} and Yang--Zhao~\cite{YangZhao}. The construction by Aganagic--Okounkov of \emph{elliptic stable envelopes}~\cite{AganagicOkounkov} in the (extended) equivariant elliptic cohomology of symplectic resolutions has far-reaching consequences in enumerative geometry and integrable systems. In particular, it interweaves with the recent elliptic Schubert calculus of Rimanyi and Weber~\cite{Rimanyi}. We emphasize that these applications are already quite interesting for complex analytic equivariant elliptic cohomology; refinements to objects over~$\Z$ will further deepen the story. 

Such refinements are the subject of Lurie's ongoing work as surveyed in~\cite{Lurie} with the state of the art being finite group equivariant elliptic cohomology~\cite{LurieIII}. The setup is inherently derived: Lurie's equivariant elliptic cohomology arises as a certain sheaf of $E_\infty$-ring spectra. The cocycle model below begins to bridge the gap between Grojnowski's approach and Lurie's. Indeed, over the complex numbers $E_\infty$-ring spectra can be modeled by commutative differential graded algebras (cdgas). Our model for equivariant elliptic cohomology is a sheaf of cdgas on a moduli space of $G$-bundles over elliptic curves. The higher derived sections of this sheaf are previously unexplored and further intertwine representation theoretic data with the rich geometry of elliptic curves, e.g., see Remark~\ref{rmk:derived} and Example~\ref{ex:derived} below. 


\subsection*{Motivation for the definition of elliptic cocycles} The precise form of our definition of elliptic cocycles (Definition~\ref{defn:ellcocycle}) takes motivation from three sources. 

The first is the preexisting de~Rham model for complexified equivariant K-theory. Cocycles in this case are ``bouquets" of equivariant differential forms that assemble into sections of a sheaf over the moduli space of $G$-bundles on the circle~$S^1$, or equivalently, the quotient of a Lie group~$G$ acting on itself by conjugation; see Block--Getzler~\cite[\S1]{BlockGetzler}, Duflo--Vergne~\cite{DufloVergne} and Vergne~\cite[Definition~23]{Vergne}. These bouquets appear naturally in equivariant index theory, as discussed in~\cite[Chapter~7 Additional remarks]{BGV}. 
With the correct perspective, Definition~\ref{defn:ellcocycle} is a natural generalization to \emph{elliptic bouquets} as a sheaf on the moduli space of $G$-bundles on elliptic curves. Indeed, Vergne has recently (and independently) produced a de~Rham model similar to Definition~\ref{defn:ellcocycle} in the special case of $\U(1)$-equivariant elliptic cohomology~\cite{VergneEll}.

The second motivation comes from ``delocalizing" Borel equivariant elliptic cohomology, as emphasized by Grojnowski~\cite[\S1]{Grojnowski}. As reviewed in~\S\ref{sec:complete} below, the Atiyah--Segal completion theorem compares equivariant K-theory with Borel equivariant K-theory. For $G=\U(1)$, the Atiyah--Segal completion map~\eqref{eq:AtiyahSegal} restricts functions on the multiplicative group $\mathbb{G}_m$ to functions on the \emph{formal} multiplicative group, $\widehat{\mathbb{G}}_m$. Demanding naturality in the group $G$ and using techniques of reduction to maximal tori, much of equivariant K-theory can then be constructed out of the multiplicative group~\cite[\S2]{Lurie}. By the definition of elliptic cohomology, the Borel equivariant elliptic cohomology group~$\Ell(B\U(1))$ can be interpreted as functions on the (moduli of) formal elliptic groups. In the spirit of delocalization, one might expect $\U(1)$-equivariant elliptic cohomology to be (the sheaf of) functions on the universal elliptic curve. If one then starts with the de~Rham model for Borel equivariant elliptic cohomology and then follows Grojnowski's delocalization procedure, this gives another road to Definition~\ref{defn:ellcocycle} for $G=\U(1)$. Further exploiting naturality in the group and reducing to maximal tori then leads to the general definition, much in the same spirit of~\cite[\S2.6]{Grojnowski} and~\cite[\S3.4-3.5]{Lurie}.

The third motivation for Definition~\ref{defn:ellcocycle} is an anticipated relationship between elliptic cohomology and 2-dimensional supersymmetric quantum field theory through a conjectured isomorphism~\cite{Witten_Dirac,Segal_Elliptic,ST04,ST11}
\beq
&&\left\{\begin{array}{c} 2{\rm -dimensional \ quantum\ field\ theories} \\ {\rm with} \ \mathcal{N}=(0,1)\ {\rm supersymmetry\ over}\ M\end{array}\right\}/{\rm deformation} \stackrel{\sim}{\dashrightarrow} \TMF(M)\label{eq:TMFconj}
\eeq
that realizes deformation classes of field theories as classes in the universal elliptic cohomology theory of topological modular forms (TMF). This cohomology theory is constructed as the global sections of a sheaf of $E_\infty$-ring spectra over the moduli stack of elliptic curves. One of the great challenges is to relate this sophisticated homotopical object to quantum field theory: at a superficial level, the candidate objects from physics have absolutely nothing to do with the objects in homotopy theory. In confronting this challenge, Lurie suggests~\cite[\S5.5]{Lurie} that an equivariant refinement would go a long way to constructing the map~\eqref{eq:TMFconj}. 

Pondering such an equivariant refinement of~\eqref{eq:TMFconj} is what originally lead us to Definition~\ref{defn:ellcocycle} of equivariant elliptic cocycles, albeit by a fairly circuitous route. To summarize, a model for (non-equivariant) complex analytic elliptic cohomology of a manifold~$M$ comes from considering functions on the moduli space of classical fields for the 2-dimensional $\mathcal{N}=(0,1)$ supersymmetric sigma model with target $M$~\cite{DBE_MQ}. Turning on background gauge fields for a gauge group $G$ that acts on $M$ results in a moduli space of fields whose functions are a model for complex analytic equivariant elliptic cohomology of the $G$-manifold~$M$~\cite{Ell1}. This equivariant moduli space is a union of \emph{twist fields} parameterized by pairs of commuting elements in~$G$, e.g., see~\cite{stringorbifolds}. Functions on twist fields for a fixed pair of commuting elements is precisely the data~\eqref{eq:cocycledef} for an elliptic cocycle, and compatibility between twist fields begets properties (1) and (2) in Definition~\ref{defn:ellcocycle}. We explain this connection to physics more fully in the companion paper~\cite{Ell1}. Together with the de~Rham model of this paper, we obtain an equivariant refinement of the isomorphism~\eqref{eq:TMFconj} over~$\C$. The easier case of a finite group is treated in~\cite{DBE_Equiv}. We view these results as a first step towards Lurie's proposed construction of the isomorphism~\eqref{eq:TMFconj}.


\subsection*{Outline and overview of results} 
Let~$G$ be a compact Lie group.
We construct a cocycle model for $G$-equivariant complex analytic elliptic cohomology as a functor~$\dEll_G^\bullet$ from $G$-manifolds to sheaves of commutative differential graded algebras (cdgas) on a stack~$\Bun_G(\EE)$. 

In~\S\ref{sec:1} we define the stack~$\Bun_G(\EE)$ and describe some of its basic geometric features. Roughly, $\Bun_G(\EE)$ classifies isomorphism classes of flat $G$-bundles over complex analytic elliptic curves. When $G=T$ is a torus, we identify
\beq
\Bun_T(\EE)\simeq \underbrace{\mathcal{E}^\vee\times_{\Mell} \cdots \times_{\Mell} \mathcal{E}^\vee}_{{\rm rk}(T){\rm \ times}}\label{eq:fibprodEll}
\eeq
with the iterated fiber product of the (dual) universal elliptic curve over the moduli stack~$\Mell$ of elliptic curves. This gives $\Bun_T(\EE)$ a holomorphic structure, and $\Bun_G(\EE)$ has a similar holomorphic structure for general $G$. Supposing that~$G$ is connected, $\Bun_G(\EE)$ supports holomorphic line bundles called \emph{Looijenga line bundles}. When $G$ is simple and simply connected, sections are spanned by (super) characters of positive energy representations of the loop group~$LG$, where the level of the representation determines the isomorphism class of the Looijenga line bundle. When~$G$ is finite, $\Bun_G(\EE)$ supports line bundles constructed by Freed and Quinn~\cite{FreedQuinn} in their study of Chern--Simons theory. Such line bundles are central to generalized moonshine when $G$ is the monster group; see Remark~\ref{rmk:moonshine}. 

In~\S\ref{sec:users} we define the sheaf~$\dEll_G^\bullet(M)$ for a $G$-manifold~$M$ and derive some of its basic properties. For example, there is a canonical identification $\dEll_G^0(\pt)\simeq \mathcal{O}_{\Bun_G(\EE)}$ with the sheaf of holomorphic functions on $\Bun_G(\EE)$. This gives $\dEll_G^\bullet(M)$ the canonical structure of a sheaf of $\mathcal{O}_{\Bun_G(\EE)}$-modules (Proposition~\ref{prop:analytic}). We show that restricting~$\dEll_G^\bullet(M)$ along the section $0\colon \Mell\to \Bun_G(\EE)$ associated to the trivial $G$-bundle gives a map to Borel equivariant elliptic cohomology. This constructs an elliptic Atiyah--Segal completion map (Theorem~\ref{thm:complete}) in the same spirit as advocated in~\cite[\S2]{Lurie}. 

In~\S\ref{sec:Groj} we prove that $\dEll_G^\bullet(M)$ is a cocycle refinement of Grojnowski's complex analytic equivariant elliptic cohomology when~$G$ is connected (Theorem~\ref{thm:Groj}). In~\S\ref{sec:Devoto} we prove that $\dEll_G^\bullet(M)$ is a cocycle refinement of Devoto's equivariant elliptic cohomology over~$\C$ when $G$ is finite (Theorem~\ref{thm:Devoto}). 
In these sections we also briefly review the preexisting definitions. 

In~\S\ref{sec:Thom} we construct cocycle representatives of equivariant elliptic Euler and Thom classes for the groups $G=\U(n)$ and $\Spin(2n)$ (Propositions~\ref{thm:Euler1} and~\ref{thm:Thom}, respectively). These cocycles come from products of certain theta functions, interpreted as sections of the sheaf~$\dEll_G$ twisted by a (Looijenga) line bundle. In this way, the Euler and Thom cocycles are determined by characters of level~1 vacuum representations of loop groups, see Proposition~\ref{thm:Euler2}. 

Thom classes determine orientations for equivariant elliptic cohomology, leading to elliptic Chern classes of vector bundles, wrong-way maps, and elliptic fundamental classes of appropriately oriented submanifolds. The $\U(n)$-equivariant Thom cocycle therefore leads to an equivariant and cocycle refinement of the $\MU\langle 6\rangle$-orientation, and the $\Spin(2n)$-equivariant Thom cocycle leads to an equivariant and cocycle refinement of the ${\rm MString}=\MO\langle 8\rangle$-orientation. We verify compatibility with the corresponding nonequivariant classes in complex analytic elliptic cohomology in Theorem~\ref{thm:compat}. 


In~\S\ref{sec:FGL}, we compare the equivariant characteristic classes from~\S\ref{sec:Thom} with the ones studied by Hopkins~\cite{HopkinsICM94} and Ando--Hopkins--Strickland~\cite{AHSI} in their construction of $\MU\langle 6\rangle$- and $\MString$-orientations of elliptic cohomology theories and~TMF. First, we make a basic observation: complex orientations of elliptic cohomology theories over~$\C$ do not admit equivariant refinements (Proposition~\ref{eq:nogo}). However, they do admit \emph{unique} twisted equivariant refinements (Proposition~\ref{prop:twistedequiv}). Next we turn our attention to $\MU\langle 6\rangle$-orientations. Recall there is a version of the splitting principle for bundles with $\U\langle 6 \rangle$-structure so that characteristic classes are determined by the class associated to the virtual vector bundle~$(L_1 - 1) \otimes (L_2 - 1) \otimes (L_3 - 1)$ over $[\pt/\U(1)]^{\times 3}$ for $L_i$ the tautological bundles on their respective factors. There is a (non-equivariant) characteristic class of this virtual vector bundle determined by the $\MU\langle6\rangle$- and $\MString$-orientations of elliptic cohomology. We show this class has a unique complex analytic equivariant refinement (Theorem~\ref{thm:cube}), whose existence is a consequence of the theorem of the cube. Finally, we show that this equivariant class agrees with the Euler class associated from the unique twisted equivariant complex orientation.  

In Appendix A, we give some essential background for the paper, starting in~\S\ref{sec:technicalities} with some useful results in Lie theory. We briefly review the Cartan model for equivariant cohomology in~\S\ref{sec:equivdeRham}. Finally, we use the language of smooth stacks (i.e., stacks on the site of smooth manifolds) throughout the paper; we review some key aspects in~\S\ref{appen:stacks}.

\subsection*{Notation and conventions}

For simplicity we assume that a $G$-manifold $M$ embeds $G$-equivariantly into a finite-dimensional $G$-representation. This is automatically satisfied when $M$ is compact by results of Mostow~\cite{Mostow} and Palais~\cite{Palais}. 

For a $G$-manifold $M$, we use the notation~$M/G$ to denote the Lie groupoid quotient with underlying stack $[M/G]$. Let $M\cq G$ denote the coarse quotient, taken in sheaves on the site of smooth manifolds. When the $G$-action is free, the stack $[M/G]$ is representable and we often identify it with the (coarse) quotient in manifolds. We refer to~\S\ref{appen:stacks} for more detail. We remark that in topology sometimes $M/\!\!/G$ is used to denote the stacky quotient, but we avoid this notation because it conflicts with the standard notation for the GIT quotient. 

Sheaf conditions are always imposed in the strict (i.e., non-homotopical) sense. For example, a sheaf of commutative differential graded algebras is a chain complex of sheaves with a commutative graded multiplication. 

Tensor products of algebras of functions or spaces of sections will always be taken as the projective tensor product of Fr\'echet spaces. This is a completion of the algebraic tensor product having the key property that $C^\infty(M)\otimes C^\infty(N)\simeq C^\infty(M;C^\infty(N))\simeq C^\infty(M\times N)$ for manifolds $M$ and $N$. 

Finally, we view modular forms as functions on the upper half plane $\HH$ with properties. The sheaf of holomorphic functions on $\HH$ will always be taken to be the one that imposes meromorphicity at infinity, so that by ``modular forms," we always mean ``weakly holomorphic modular forms." More precisely, for an open $U\subset \HH$, the sections $\mathcal{O}(U)$ are the holomorphic functions on $U$ with at most polynomial growth along any geodesic (in the hyperbolic metric) escaping to $\partial \HH$.
 
\subsection*{Acknowledgements}
Ian Grojnowski's vision for deploying equivariant elliptic cohomology to study representation-theoretic problems has been a constant source of inspiration throughout the duration of this project. A tremendous amount of elliptic cohomology was also developed in notes of Mike Hopkins, to whom we owe a great intellectual debt. 
We also wish to thank Matt Ando, Nora Ganter, Tom Nevins, Andrei Okounkov, and Charles Rezk for stimulating conversations, 
Kiran Luecke for comments on an earlier draft, and anonymous referees whose comments helped improve the exposition. Finally, A.T. acknowledges the support of MSRI and the NSF through grants~1705008 and~1440140.

\section{$G$-bundles on elliptic curves} \label{sec:1}

This section uses the language of smooth stacks; we refer to~\S\ref{appen:stacks} for a brief introduction. In particular, see Example~\ref{ex:stackquotient} for the definition of a quotient stack, Definition~\ref{defn:atlas} for the notion of an atlas of a differentiable stack, and Definition~\ref{defn:holomorphic} for the notion of a holomorphic structure on a smooth stack. 

\subsection{Elliptic curves}
Consider the $\SL_2(\Z)$-action on the upper half plane $\HH$ by fractional linear transformations,
\beq
\left(\left[\begin{array}{cc} a & b \\ c& d\end{array}\right] ,\tau\right)\mapsto \frac{a\tau+b}{c\tau+d},\qquad \left[\begin{array}{cc} a & b \\ c& d\end{array}\right]\in \SL_2(\Z), \ \ \tau\in \HH. \label{eq:fraclin}
\eeq
Define the \emph{moduli stack of elliptic curves} as the quotient stack,
$$
\Mell\simeq[\HH/\SL_2(\Z)].
$$
Since the $\SL_2(\Z)$-action preserves the complex structure on $\HH$, the map $\HH\to \Mell$ is a holomorphic atlas that endows the smooth stack $\Mell$ with a complex analytic structure. 

Consider the quotient manifold
\beq
\widetilde{\Euni}:=(\C\times \HH)/\Z^2,\label{eq:Euniv}
\eeq
for the free $\Z^2$-action $(n,m)\cdot (z,\tau)= (z+\tau n+m,\tau)$ where $z\in\C$, $\tau\in \HH$ and $(n,m)\in \Z^2$. There is an action of $\SL_2(\Z)$ on $\widetilde{\Euni}$ that covers the action~\eqref{eq:fraclin} on~$\HH$,
$$
\left(\left[\begin{array}{cc} a & b \\ c &d\end{array}\right],z,\tau\right)\mapsto \left(\frac{z}{c\tau+d},\frac{a\tau+b}{c\tau+d}\right),\qquad z\in \C, \ \tau\in \HH, \ \left[\begin{array}{cc} a & b \\ c &d\end{array}\right]\in \SL_2(\Z).
$$
Define the \emph{universal elliptic curve} as the quotient stack $\Euni\simeq [\widetilde{\Euni}/\SL_2(\Z)]$. The evident map $\C\times \HH\to \EE$ provides a holomorphic atlas for~$\EE$, and the projection $\C\times \HH\to \HH$ induces a map of complex analytic stacks $\Euni \to \Mell.$ We remark that for a complex manifold $S$ and a map $f\colon S\to \Mell$ of complex analytic stacks, the 2-pullback $f^*\Euni\to S$ gives the family of complex analytic elliptic curves over $S$ classified by~$f$. 

There is a similarly defined universal \emph{dual} elliptic curve, $\Euni^\vee\simeq [\widetilde{\Euni}^\vee/\SL_2(\Z)]$, for $\widetilde{\Euni}^\vee$ defined as a quotient as in~\eqref{eq:Euniv} but for the $\Z^2$-action 
\beq
(n,m)\cdot (z,\tau)= (z+n-\tau m,\tau).\label{eq:Eunivee}
\eeq
The evident map $\C\times \HH\to \EE^\vee$ also provides a holomorphic atlas for $\Euni^\vee$, and the obvious map $\Euni^\vee\to \Mell$ is again a map of complex analytic stacks. 

\begin{rmk} \label{rmk:dualell}
More geometrically, the dual of an elliptic curve is its space of degree-zero line bundles. In the complex analytic setting, this is the space of topologically trivial line bundles endowed with flat, unitary connections. We identify a point in $\widetilde{\Euni}^\vee$ with such a line bundle as follows: $(x-\tau y,\tau)\in \widetilde{\Euni}^\vee$ for $x,y\in \R$ gets sent to the line bundle $L$ corresponding to the one-dimensional representation of the fundamental group 
$$
\pi_1(E_\tau)=\pi_1(\C/\langle\tau,1\rangle)\to \U(1),\qquad  \tau m+n\mapsto e^{2\pi i (mx+ny)}. 
$$
\end{rmk}


\subsection{The smooth stack $\Bun_G(\EE)$ of $G$-bundles}\label{sec:BunGE}

A smooth manifold $N$ determines a sheaf on the site of manifolds whose value on $S$ is the set of smooth maps $S\to N$. Below we will often identify a manifold with its representable sheaf, where $N(S)$ denotes the value of this sheaf on the manifold $S$, also called the \emph{$S$-points of $N$}. 

An interesting class of non-representable sheaves on the site of smooth manifolds arises by considering (non-smooth) subsets of a smooth manifold. We refer to Example~\ref{ex:subobject} for a discussion of subobject sheaves.

\begin{defn} \label{defn:C2G}
Define $\mathcal{C}^2(G)$ as the subobject of the representable sheaf $G\times G$ whose $S$-points are smooth maps $S\to G\times G$ that pointwise commute in $G$, 
$$
\mathcal{C}^2(G)(S):=\{h_1,h_2\colon S\to G\mid h_1(s)h_2(s)=h_2(s)h_1(s) \ \forall s\in S\}. 
$$
Equivalently, $\mathcal{C}^2(G)$ is the sheaf of smooth families of homomorphisms $\Z^2\to G$. 
\end{defn}

Typically the sheaf $\mathcal{C}^2(G)$ fails to be representable when $G$ is nonabelian. To avoid cluttering some formulas below, we often use the notation $h\in \mathcal{C}^2(G)$ to denote a pair of commuting elements $h=(h_1,h_2)\in G\times G$, rather than the more cumbersome $h\in \mathcal{C}^2(G)(\pt)$. 

We refer to Example~\ref{ex:coarsequotient} for a discussion of coarse quotients in sheaves. 

\begin{defn}
Let $\mathcal{C}^2[G]$ denote the coarse quotient sheaf, $\mathcal{C}^2(G)\cq G_0=:\mathcal{C}^2[G]$ where $G_0<G$ is the connected component of the identity acting on~$\mathcal{C}^2(G)\subset G\times G$ by restriction of the conjugation action $g\cdot(h_1,h_2)=(gh_1g^{-1},gh_2g^{-1})$ for $g\in G_0$ and $(h_1,h_2)\in G\times G$. 
\end{defn}

There is a left action of $\SL_2(\Z)$ on the sheaf $\mathcal{C}^2(G)$
\beq
(h_1,h_2)\mapsto (h_1^dh_2^{-b},h_1^{-c}h_2^a),\qquad \left[\begin{array}{cc} a & b \\ c & d\end{array}\right]\in \SL_2(\Z), \quad h_1,h_2\colon S\to G.\label{eq:SL2Zaction}
\eeq
This is the precomposition $\SL_2(\Z)$-action on families of homomorphisms~$\Z^2\to G$; note that precomposition actions are naturally right actions, and the signs in~\eqref{eq:SL2Zaction} are the result of turning this into a left action. The action~\eqref{eq:SL2Zaction} descends to the quotient $\mathcal{C}^2[G]$. There is a residual $\pi_0(G)$-action on $\mathcal{C}^2[G]$ by conjugation, and this commutes with the $\SL_2(\Z)$-action. Viewing $\HH$ as a representable sheaf with the $\SL_2(\Z)$-action~\eqref{eq:fraclin},  we obtain an $\SL_2(\Z)\times\pi_0(G)$-action on the sheaf $\HH\times \mathcal{C}^2[G]$ which defines the generalized (action) Lie groupoid $\HH\times \mathcal{C}^2[G]/\SL_2(\Z)\times \pi_0(G)$, i.e., a groupoid object in sheaves, see Definition~\ref{defn:genLiegrpd}. 


\begin{defn}\label{defn:BunG}
Define the smooth stack
\beq
\Bun_G(\EE)\simeq [\HH\times \mathcal{C}^2[G]/\SL_2(\Z)\times \pi_0(G)]\label{eq:BunG}
\eeq
as the stack underlying the generalized Lie groupoid. This stack is natural in $G$: a homomorphism $G\to H$ determines a functor $\Bun_G(\EE)\to \Bun_H(\EE).$
\end{defn}

\begin{rmk}
The projection $\Bun_G(\EE)\to \Mell$ witnesses $\Bun_G(\EE)$ as a \emph{relative} coarse moduli space of $G$-bundles on elliptic curves. Indeed, a pair of commuting elements defines a flat $G$-bundle on an elliptic curve with chosen generators for its fundamental group, and a conjugacy class of such a pair is a $G$-bundle up to isomorphism. Hence for $G$ connected, if we fix an elliptic curve and generators for its fundamental group (specified in terms of $\tau\in \HH$), the fiber of $\Bun_G(\EE)$ at $\tau$ is the moduli space of isomorphism classes of~$G$-bundles.
\end{rmk} 

\begin{rmk} 
Categorically-minded readers might find the above version of $\Bun_G(\EE)$ peculiar, preferring instead the stack that records all isomorphisms between $G$-bundles. However, $\Bun_G(\EE)$ as defined in~\eqref{eq:BunG} turns out to be the right home for complex analytic equivariant elliptic cohomology. Lurie's construction also takes place over a moduli space of $G$-bundles rather than the full moduli stack~\cite[Remark~5.1]{Lurie}. One subtlety in~\eqref{eq:BunG} is that we do not pass completely to the coarse quotient $\mathcal{C}^2(G)\cq G=\mathcal{C}^2[G]\cq \pi_0(G)$ because the action by~$\pi_0(G)$ is important for certain desired applications of equivariant elliptic cohomology for finite groups; see Remark~\ref{rmk:pi0}.
\end{rmk} 

\subsection{Sheaves on $\Bun_G(\EE)$} 

We shall define sheaves on $\Bun_G(\EE)$ as $\SL_2(\Z)\times \pi_0(G)$-equivariant sheaves on $\HH\times \mathcal{C}^2[G]$. Since $\mathcal{C}^2[G]$ typically fails to be representable, we first require the following definition. 

\begin{defn} \label{defn:opensub} A \emph{subsheaf} of a sheaf~$F$ is a morphism of sheaves $U\to F$ with the property that the induced map of sets $U(S)\to F(S)$ is injective for each manifold~$S$.  An \emph{open subsheaf} of a sheaf $F$ is a subsheaf $U \subset F$ with the property that for any representable sheaf~$S$ and any $S$-point $S\to F$, the pullback 
\beq \begin{tikzpicture}[baseline=(basepoint)]; 
\node (A) at (0,0) {$U_S$}; 
\node (B) at (4,0) {$U$}; 
\node (C) at (0,-1) {$S$}; 
\node (D) at (4,-1) {$F$};
\draw[->] (A) to (B); 
\draw[->] (A) to (C); 
\draw[->] (B) to  (D);
\draw[->] (C) to  (D);
\path (0,-.75) coordinate (basepoint); 
\end{tikzpicture}\label{eq:opensubsheafpull}
\eeq
is representable by a manifold, and the map $U_S\to S$ of representable sheaves is determined by the inclusion of an open subsubmanifold in $S$. Let ${\rm Open}(F)$ denote the category whose objects are open subsheaves of $F$ and morphisms are inclusions, which we write e.g., as $U \subset V \subset F$. An \emph{open cover} of a sheaf $F$ is a collection of open subsheaves $\{U_\alpha \to F\}$ with the property that for any representable sheaf~$S$ and any $S$-point $S\to F$, the pullback~\eqref{eq:opensubsheafpull} is representable and determines an open cover $\{(U_\alpha)_S \to S\}$ of the manifold~$S$. 
\end{defn}

\begin{ex}\label{ex:representableopen} An open submanifold $U\subset N$ of a manifold $N$ determines an open subsheaf of the sheaf represented by $N$. Similarly, an open cover of $N$ determines an open cover of the sheaf represented by $N$. \end{ex}

For more properties of the category of open subsheaves, we refer to~\cite[\S A.4]{Ell1}. 

\begin{defn}\label{defn:sheafonsheaf}
A \emph{presheaf} on a sheaf~$F$ is a functor ${\rm Open}(F)^{\rm op}\to {\sf Set}$. A presheaf is a \emph{sheaf} if it satisfies the usual sheaf condition for open covers of $F$. 
\end{defn}

\begin{defn} 
A \emph{presheaf on $\Bun_G(\EE)$} is an $\SL_2(\Z)\times \pi_0(G)$-equivariant presheaf on $\HH\times\mathcal{C}^2[G]$. A presheaf on $\Bun_G(\EE)$ is a \emph{sheaf} if its underlying presheaf on $\HH\times\mathcal{C}^2[G]$ satisfies the sheaf condition. 
\end{defn}


To construct sheaves on $\Bun_G(\EE)$, we will need a supply of open subsheaves. We will build these by starting with open submanifolds of $G\times G$ viewed as open subsheaves via Example~\ref{ex:representableopen}, and then restricting along the inclusion $\mathcal{C}^2(G)\subset G\times G$ to obtain an open subsheaf of $\mathcal{C}^2(G)$. We descend to an open subsheaf of $\mathcal{C}^2[G]$ using the following lemma. 
\begin{lem}\label{lem:cov1}
Let $U$ be a $G_0$-invariant open submanifold of $G\times G$. Identify $U$ with its representable open subsheaf. Then the pulback is a $G_0$-invariant open subsheaf of $\mathcal{C}^2(G)$
\beq
\begin{tikzpicture}[baseline=(basepoint)];
\node (A) at (0,0) {$U\bigcap \mathcal{C}^2(G)$};
\node (B) at (3,0) {$U$};
\node (C) at (0,-1.2) {$\mathcal{C}^2(G)$};
\node (D) at (3,-1.2) {$G\times G,$}; 
\draw[->] (A) to  (B);
\draw[->] (A) to  (C);
\draw[->] (C) to  (D);
\draw[->] (B) to (D);
\path (0,-.75) coordinate (basepoint);
\end{tikzpicture}\label{eq:GGcov1}
\eeq
and the coarse quotient sheaf $(U\bigcap \mathcal{C}^2(G))\cq G_0$ is an open subsheaf of $\mathcal{C}^2[G]$. 
\end{lem}
\bp The argument is largely formal, e.g., see \cite[Proposition~A.40]{Ell1}. \ep

Next, let $B_h^\epsilon=B_{h_1}^\epsilon\times B_{h_2}^\epsilon\subset G\times G$ denote the product of $\epsilon$-balls in $G$ around $h_1$ and $h_2$, which we parameterize by $\Ad$-invariant $\epsilon$-balls in the Lie algebra,
\beq
&&B_h^\epsilon=\{(X_1,X_2)\in B_\epsilon(\mf{g}) \times B_{\epsilon}(\mf{g})\}\to G\times G\quad (X_1,X_2)\mapsto (h_1e^{X_1},h_2e^{X_2}).\label{eq:itsabasis}
\eeq
Let $G_0\cdot B_h^\epsilon\subset G\times G$ denote the orbit of $B_h^\epsilon$ under the conjugation action by~$G_0$. We observe that  $G_0\cdot B_h^\epsilon$ is again an open submanifold of $G\times G$, which we identify with its representable open subsheaf. The restriction of $G_0\cdot B_h^\epsilon\bigcap \mathcal{C}^2(G)$ to the subsheaf $\mathcal{C}^2(G)\hookrightarrow G\times G$ defines a $G_0$-invariant open subsheaf of $\mathcal{C}^2(G)$. Applying Lemma~\ref{lem:cov1} yields the following. 

\begin{cor}\label{cor:opencover1} For any $\delta>0$, the collection $\{U_h^\epsilon\}_{(h,\epsilon)\in \mathcal{C}^2(G)\times(0,\delta)}$ of open subsheaves defined by 
$$
U_h^{\epsilon} := (G_0\cdot B_h^\epsilon\bigcap \mathcal{C}^2(G))\cq G_0
$$ 
is an open cover of~$\mathcal{C}^2[G]$. \end{cor}

The next step is to describe the open subsheaves $U_h^{\epsilon}$ in a manner that will be useful in later constructions. We start with some notation. For each pair of commuting elements $h\in \mathcal{C}^2(G)$, and a choice of maximal commuting subalgebra $\mf{t}_{\mf{g}^h}\subset \mf{g}^h$, define the following (non-open) submanifold of~$G\times G$:
\beq
&&T_h^\epsilon =\{(X_1,X_2)\in B_\epsilon(\mf{t}_{\mf{g}^h}) \times B_{\epsilon}(\mf{t}_{\mf{g}^h})\}\hookrightarrow G\times G,\quad (X_1,X_2)\mapsto (h_1e^{X_1},h_2e^{X_2}) \label{eq:BGopen}
\eeq
where $B_{\epsilon}(\mf{t}_{\mf{g}^h})$ is an $\epsilon$-ball about the origin for an $\Ad$-invariant metric on $\mf{g}$ restricted to~$\mf{t}_{\mf{g}^h}$. There is an evident map of smooth manifolds $T_h^\epsilon\hookrightarrow B_h^\epsilon\to G_0\cdot B_h^\epsilon\subset G\times G$. The induced map of representable sheaves factors through the subsheaf $\mathcal{C}^2(G)\subset G\times G$. Next define the finite group 
\beq
W^h := N_{G_0^h}(T_{G_0^h}) / T_{G_0^h}\label{eq:Weyldefn}
\eeq
where $(G_0)^h=:G_0^h<G_0$ is the $h$-fixed subgroup, $T_{G_0^h}$ denotes a maximal torus of the identity connected component of~$G_0^h$, and $N_{G_0^h}(T_{G_0^h})<G_0^h$ is the normalizer of $T_{G_0^h}$ in~$G_0^h$. The notation $W^h$ is intended to evoke a Weyl group, though in general~$G_0^h$ need not be connected in which case $W^h$ is larger than the Weyl group of the identity component of~$G_0^h$

\begin{lem}\label{lem:cov2} For a fixed pair of commuting elements $h=(h_1,h_2)\in \mathcal{C}^2(G)$, there exists a real number $\delta=\delta(h)>0$ depending on $h$ such that for all $\epsilon\in (0,\delta)$, there is an isomorphism of subsheaves of~$\mathcal{C}^2[G]$,
\beq
\begin{tikzpicture}[baseline=(basepoint)];
\node (A) at (0,0) {$T_h^\epsilon\cq W^h$};
\node (B) at (3,-.6) {$\mathcal{C}^2[G]$};
\node (C) at (0,-1.2) {$U_h^{\epsilon}$};
\draw[->] (A) to (B);
\draw[->] (A) to node [left] {$\simeq $} (C);
\draw[->] (C) to  (B);
\path (0,-.75) coordinate (basepoint);
\end{tikzpicture}\nonumber
\eeq
between the open subsheaf $U_h^{\epsilon}$ of~$\mathcal{C}^2[G]$ and the coarse quotient sheaf~$T_h^\epsilon\cq W^h$.
\end{lem}

\bp
Consider the (strictly) commuting triangle in generalized Lie groupoids
\beq
\begin{tikzpicture}[baseline=(basepoint)];
\node (A) at (0,0) {$T_h^\epsilon/N_{G_0^h}(T_{G_0^h})$};
\node (B) at (3,-.6) {$\mathcal{C}^2(G)/G_0$};
\node (C) at (0,-1.2) {$(G_0\cdot B_\epsilon^h\bigcap \mathcal{C}^2(G))/G_0.$};
\draw[->] (A) to (B);
\draw[->] (A) to node [left] {$\varphi$} (C);
\draw[->] (C) to  (B);
\path (0,-.75) coordinate (basepoint);
\end{tikzpicture}\nonumber
\eeq
The arrows come from regarding the sheaves of objects of these groupoids as subsheaves of the representable sheaf $G\times G$ and taking the obvious inclusions;  on morphisms these arrows are determined the inclusion $N_{G_0^h}(T_{G_0^h})<G_0$. 
For $\epsilon$ sufficiently small, the map $\varphi$ is essentially surjective by Lemma~\ref{lem:ellcentslices}. By Lemma~\ref{lem:conj} and Proposition~\ref{prop:conjugate}, $\varphi$ induces an isomorphism on isomorphism classes of objects (as sets). This implies an isomorphism of subsheaves of $\mathcal{C}^2(G)\cq G_0\subset (G\times G)\cq G_0$. Passing to coarse quotients, $\varphi$ determines an isomorphism of sheaves $T_h^\epsilon\cq N_{G_0^h}(T_{G_0^h})\stackrel{\sim}{\to} U_h^\epsilon$ over~$\mathcal{C}^2[G]$. Next we identify the coarse quotient sheaves~$T_h^\epsilon\cq N_{G_0^h}(T_{G_0^h}) \simeq T_h^{\epsilon}\cq W^h$ using that the conjugation action of $T_{G_0^h}<N_{G_0^h}(T_{G_0^h})$ on $T_{G_0^h}$ is trivial, as $T_{G_0^h}$ is a torus acting on itself by conjugation. Commutativity of the diagram therefore identifies $T_h^{\epsilon}\cq W^h$ with the open subsheaf $U_h^{\epsilon}$ of $\mathcal{C}^2[G]=\mathcal{C}^2(G)\cq G_0$.
%
%
\ep

%



For each pair of commuting elements $h\in \mathcal{C}^2(G)$, we fix once and for all a choice of maximal commuting subalgebra $\mf{t}_{\mf{g}^h}\subset \mf{g}^h$, which in turn gives a map $T_h^\epsilon\to U_h^\epsilon$ for each $U_h^\epsilon$ in the cover of $\mathcal{C}^2[G]$ from Corollary~\ref{cor:opencover1}, where we assume that $\delta=\delta(h)$ is sufficiently small to satisfy the hypothesis of Lemma~\ref{lem:cov2}. Below, we drop the $h$-dependence of $\delta$ in the notation. We use this data to construct sheaves on $\Bun_G(\EE)$ as follows.

\begin{prop}\label{prop:BunGsheaf}
Fixing the choices described above, suppose we are given data:
\begin{enumerate}
\item[(D1)] $W^h$-equivariant sheaves $\mathcal{F}_h^\epsilon$ on each manifold $\HH\times T_h^\epsilon$ for all $\epsilon<\delta$;
\item[(D2)] for each $(g,\gamma)\in G\times \SL_2(\Z)$ determining an isomorphism from an open subset of $V\subset T_h^\epsilon$ to an open subset $V'\subset T_{h'}^{\epsilon'}$, we require the data of an isomorphism of sheaves $(g,\gamma)^*(\mathcal{F}_{h'}^{\epsilon'}|_{V'})\simeq \mathcal{F}_h^\epsilon|_V$ on $V$. 
\end{enumerate}
We require these data satisfy the following conditions. 
\begin{enumerate}
\item[(C1)] The equivariant structure on the sheaf $\mathcal{F}_h^\epsilon$ from (2) associated with automorphisms of $T_h^\epsilon$ for $g\in N_{G_0^h}(T_{G_0^h})$ agrees with the equivariant structure in (1) via the action through the quotient $W^h=N_{G_0^h}(T_{G_0^h})/T_{G_0^h}$.
\item[(C2)] The isomorphisms (D2) satisfy the natural compatibilities for nested open subsets and products of elements of $G\times \SL_2(\Z)$. 
\end{enumerate}
Then the above determines a $\pi_0(G)\times\SL_2(\Z)$-equivariant sheaf on $\HH\times\mathcal{C}^2[G]$, which we identify with a sheaf on $\Bun_G(\EE)\simeq [\HH\times\mathcal{C}^2[G]/\pi_0(G)\times \SL_2(\Z)]$.
%
%
%
\end{prop}
\bp
Throughout, we shall identify open subsets of $\HH\times T_h^\epsilon$ with open subsheaves of the associated representable sheaf. First define a sheaf subordinate to the open cover $\{\HH\times U_h^\epsilon\}$ of $\HH\times\mathcal{C}^2[G]$ as the direct image of the data (D1) along the $W^h$-quotient maps
$$
q_h\colon \HH\times T_h^\epsilon \to \HH\times T_h^\epsilon\cq W^h\simeq \HH\times U_h^\epsilon. 
$$
Explicitly, the sections of the direct image $(q_h)_*\F_h^\epsilon$ assign to an open subsheaf $U\subset \HH\times U_h^\epsilon$ the $W^h$-invariant sections of $\F_h^\epsilon$ on $q_h^{-1}(U)\subset \HH\times T_h^\epsilon$. 

Next we wish to descend the sheaf on $\{\HH\times U_h^\epsilon\}$ defined above to a $\pi_0(G)\times \SL_2(\Z)$-equivariant sheaf on $\HH\times\mathcal{C}^2[G]$. This is both data and property, namely, the \emph{data} of isomorphisms of sheaves covering isomorphisms between open subsheaves specified by the action of $\pi_0(G)\times \SL_2(\Z)$, and the \emph{property} of compatibility between multiplication in the group $\pi_0(G)\times \SL_2(\Z)$ and composition of isomorphisms of sheaves. 

We first construct the descent data. So suppose we are given $([g],\gamma)\in \pi_0(G)\times \SL_2(\Z)$ that sends an open subsheaf $U\subset \HH\times U_h^\epsilon$ to an open subsheaf $U'\subset \HH\times U_{h'}^{\epsilon'}$. Since all maximal abelian subalgebras $\mf{t}_{\mf{g}^h}\subset \mf{g}^h$ are conjugate, there exists a choice $(g,\gamma)\in G\times \SL_2(\Z)$ sending $q_h^{-1}(U)\subset \HH\times T_h^\epsilon$ to $q_{h'}^{-1}(U')\subset \HH\times T_{h'}^{\epsilon'}$. Then the data (D2) specifies an isomorphism between invariant sections. This gives the descent data for the desired $\pi_0(G)\times \SL_2(\Z)$-equivariant sheaf on $\HH\times\mathcal{C}^2[G]$.

It remains to verify the descent property. This amounts to showing that the isomorphisms constructed above are well-defined (i.e., independent of the choice of representative $g$ of $[g]\in \pi_0(G)$) and compatible with the group structure on $\pi_0(G)\times \SL_2(\Z)$. So suppose we have two lifts $(g,\gamma),(g',\gamma)\in G\times \SL_2(\Z)$. Then $(g')^{-1}g\in G_0$; as it must also commute with $h$, $(g')^{-1}g \in G_0^h$. But now, using Lemma~\ref{lem:lcrank}, $g$ and $g'$ must both conjugate $\mf{t}_{\mf{g}^h}$ to $\mf{t}_{\mf{g}^{h'}}$ and hence $(g')^{-1}g$ preserves $\mf{t}_{\mf{g}^h}$. It follows that $(g')^{-1}g \in N_{G_0^h}(T_{G_0^h})$, and so the requirement (C1) implies that $(g')^{-1}g$ acts through $W^h$, and hence the action is trivial on $W^h$-invariant sections. This shows that the isomorphisms of sheaves are well-defined. The compatibility of these isomorphism is a direct consequence of condition (C2), using that the quotient map $G\to G/G_0\simeq \pi_0(G)$ is a homomorphism. 
\ep

\subsection{Holomorphic functions on $\Bun_G(\EE)$}

We endow $\Bun_G(\EE)$ with a sheaf of holomorphic functions whose definition is motivated by the identification between flat $\U(1)$ bundles on an elliptic curve with the dual elliptic curve: the former is a priori only a smooth moduli space whereas the latter has an obvious holomorphic structure; see Remark~\ref{rmk:dualell}. More generally, we have the following holomorphic structure on flat $T$-bundles. 

\begin{ex}\label{ex:torussheaf}
Let $T$ be a torus, so that all pairs of elements commute and
\beq
\Bun_T(\EE)\simeq [\HH\times T\times T/\SL_2(\Z)].\label{eq:univT}
\eeq
Then there is an isomorphism of smooth stacks $\Bun_T(\EE)\simeq \Euni^{\vee,{\rm rk}(T)}$
where 
\beq
\Euni^{\vee,{\rm rk}(T)}:=\underbrace{\mathcal{E}^\vee\times_{\Mell} \cdots \times_{\Mell} \mathcal{E}^\vee}_{{\rm rk}(T){\rm \ times}}\label{eq:fiberedprod}
\eeq
is the iterated fibered product of the universal (dual) elliptic curve. Here we use that the map $\EE^\vee\to \Mell$ is induced by an \emph{isofibration} of complex analytic groupoids so that the 2-fibered product of stacks~\eqref{eq:fiberedprod} agrees with the strict fibered product of groupoids. Hence, we obtain a holomorphic atlas $\HH\times \C^{{\rm rk}(T)}\to \Bun_T(\EE)$ from
\beq
&&u\colon \HH\times \C^{{\rm rk}(T)}\simeq \HH\times \mf{t}_\C\simeq \HH\times \mf{t}\times \mf{t}\stackrel{\exp}{\to} \HH\times T\times T\to \Bun_T(\EE)\nonumber\\ 
&&(\tau,X_1-\tau X_2)\stackrel{u}{\mapsto} (\tau,e^{X_1},e^{X_2}), \ X_1,X_2\in\mf{t}.\label{eq:BunThol}
\eeq
The $\SL_2(\Z)$-actions preserve the evident complex structures, so that we indeed have a holomorphic atlas for~$\Bun_T(\EE)$. 
\end{ex}

For an open subset $U_0\subset \HH$, we promote $U_0\times T_h^\epsilon$ (for $T_h^\epsilon$ from~\eqref{eq:BGopen}) to a complex manifold, with holomorphic structure coming from the isomorphism~$\varphi$
\beq
&&U_0\times T_h^\epsilon \subset U_0\times \mf{t}_{\mf{g}^h}\times \mf{t}_{\mf{g}^h}\stackrel{\varphi}{\to} U_0\times (\mf{t}_{\mf{g}^h})_\C,\quad \varphi(\tau,X_1,X_2)=(\tau,X_1-\tau X_2)\label{eq:BunThol2}
\eeq
that for each $\tau$ identifies $T_h^\epsilon$ as an open submanifold of the complex vector space~$(\mf{t}_{\mf{g}^h})_\C$ with the standard complex structure.

\begin{defn}\label{defn:BunGholo} Let $U\subset  \HH\times\mathcal{C}^2[G]$ be an open subsheaf. A function $f\in C^\infty(U)$ is \emph{holomorphic} if for every $U_0\times T_h^\epsilon$ for $U_0\subset \HH$ and $T_h^\epsilon$ from~\eqref{eq:BGopen}, $f$ is holomorphic on restriction to the pullback 
\beq
\begin{tikzpicture}[baseline=(basepoint)];
\node (A) at (0,0) {$P_h$};
\node (B) at (3,0) {$U$};
\node (C) at (0,-1.2) {$U_0\times T_h^\epsilon$};
\node (D) at (3,-1.2) {$\HH\times \mathcal{C}^2[G]$}; 
\draw[->] (A) to  (B);
\draw[->] (A) to  (C);
\draw[->] (C) to  (D);
\draw[->] (B) to (D);
\path (0,-.75) coordinate (basepoint);
\end{tikzpicture}\label{eq:GGcov}
\eeq
using the evident complex structure on the open subset $P_h\subset U_0\times T_h^\epsilon\subset \HH\times (\mf{t}_{\mf{g}^h})_\C$.
\end{defn}

\begin{lem}
Definition~\ref{defn:BunGholo} defines a subsheaf of smooth functions on $\Bun_G(\EE)$ that when $G=T$ is the sheaf of holomorphic functions defined in Example~\ref{ex:torussheaf}. 
\end{lem}
\bp
Holomorphic structures on a complex vector space are invariant under linear transformations and translations. For the vector space $(\mf{t}_{\mf{g}^h})_\C$, this implies that the subsheaf of functions defined above is invariant under the $W^h$-action (which acts linearly on $(\mf{t}_{\mf{g}^h})_\C$) and that these subsheaves are compatible under restriction (which compares holomorphic functions on $(\mf{t}_{\mf{g}^h})_\C$ related by a translation). We further observe that these subsheaves are invariant under the action of~$\SL_2(\Z)$, since $\mf{g}^h=\mf{g}^{\gamma\cdot h}$ for $\gamma\in \SL_2(\Z)$ and so we may take $(\mf{t}_{\mf{g}^h})_\C=(\mf{t}_{\mf{g}^{\gamma h}})_\C$. By Proposition~\ref{prop:BunGsheaf}, Definition~\ref{defn:BunGholo} therefore determines a $\pi_0(G)\times \SL_2(\Z)$-equivariant subsheaf of smooth functions on~$\HH\times\mathcal{C}^2[G]$. Comparing formulas~\eqref{eq:BunThol} and~\eqref{eq:BunThol2}, we find that this subsheaf agrees with the one from Example~\ref{ex:torussheaf} when~$G=T$. 
\ep

\begin{notation}
Let $\mathcal{O}_{\Bun_G(\EE)}$ denote the sheaf of holomorphic functions on $\Bun_G(\EE)$. 
\end{notation}


\begin{ex}\label{ex:holomorphic}
Let $G$ be connected with torsion-free fundamental group, maximal torus $T<G$, and Weyl group $W=N(T)/ T$. Borel~\cite[Corollary~3.5]{Borel} shows that in this case any pair of commuting elements can be simultaneously conjugated into $T$, and that pairs of elements in~$T$ are conjugate if and only if they are conjugate by an element of $N(T)$. In brief, we have~$\mathcal{C}^2[G]\simeq (T\times T)\cq W$. This gives the description
$$
\Bun_G(\EE)\simeq [(\widetilde{\Euni}^{\vee,{\rm rk}(T)}\cq W)/\SL_2(\Z)].
$$ 
so that holomorphic functions on $\Bun_G(\EE)$ are determined by $W$-invariant holomorphic functions on~$\widetilde{\Euni}^{\vee,{\rm rk}(T)}$. 
More explicitly, a locally-defined function on $\Bun_G(\EE)$ is holomorphic if and only if its pullback along 
\beq
&&\HH\times \mf{t}_\C\stackrel{\varphi}{\simeq} \HH\times\mf{t}\times\mf{t}\to \HH\times T\times T\to \HH\times (T\times T)\cq W \to \Bun_G(\EE)\label{eq:BunGhol}
\eeq
defines a holomorphic function on a (necessarily $W$-invariant) open subset of $\HH\times \mf{t}_\C$. In the above, the isomorphism $\varphi$ is defined in~\eqref{eq:BunThol2}. 
\end{ex}
\begin{ex}\label{ex:opensub}
When $G$ is connected (without any additional hypotheses) a choice of maximal torus $T<G$ induces an inclusion of stacks
\beq
[(\widetilde{\Euni}^{\vee,{\rm rk}(T)}\cq W)/\SL_2(\Z)]\simeq [(\HH\times (T\times T)\cq W)/\SL_2(\Z)]\hookrightarrow \Bun_G(\EE)\label{eq:opensub}
\eeq
as the connected component of trivial bundle on $\Bun_G(\EE)$. Hence, holomorphic functions on the image of this inclusion are determined by holomorphic functions on $\widetilde{\Euni}^{\vee,{\rm rk}(T)}$ as in the previous example. 
\end{ex}

\begin{ex} \label{ex:SO3}
When the fundamental group of $G$ has torsion,~\eqref{eq:opensub} can fail to be surjective even when $G$ is connected. For example, take $G=\SO(3)$. Then pairs of commuting elements are given by either pairs of rotations about a fixed common axis or pairs of reflections about orthogonal axes. In the former case, both elements are in a common maximal torus, whereas there is a unique conjugacy class for the latter pair. Hence $\mc{C}^2[G] = (T \times T) \cq W \sqcup \pt$, where $T = \SO(2)$ parameterizes rotations about a fixed axis and $W = \mb{Z}/2$ acts by inversion. So we find
\beq
\Bun_{\SO(3)}(\EE) \simeq [(\widetilde{\Euni}^{\vee}\cq \Z/2)/\SL_2(\Z)] \sqcup \Mell.\label{eq:obsholo}
\eeq
The forgetful map $\Bun_{\SO(3)}(\EE) \to \Mell$ comes from the projection $\EE^\vee\to \Mell$ on the first component and is the identity on the second component. The holomorphic functions on $\Bun_{\SO(3)}(\EE)$ from Definition~\ref{defn:BunGholo} are then the obvious ones in the description~\eqref{eq:obsholo}. 
\end{ex}

\subsection{Modular forms and theta functions}\label{sec:Loo} 

\begin{defn}\label{defn:holo}
A \emph{holomorphic line bundle} on $\Bun_G(\EE)$ is a locally free rank one sheaf of modules over~$\mathcal{O}_{\Bun_G(\EE)}$. 
\end{defn}

There is a holomorphic line bundle $\omega$ over $\Mell$ whose fiber at a given elliptic curve $E$ is the vector space of holomorphic 1-forms on~$E$. Pulling back along $\HH\to \Mell$, the line $\omega^{\otimes k}$ trivializes with trivializing section determined by the holomorphic 1-form descending from~$dz$ on $\C$ along the quotient~\eqref{eq:Euniv}. Sections can be described explicitly as 
$$
F(\gamma\cdot \tau)=(c\tau+d)^{k}F(\tau)\qquad {\rm for}\qquad \gamma=\left[\begin{array}{cc} a & b \\ c & d\end{array}\right]\in \SL_2(\Z). 
$$
Global sections of $\omega^{\otimes k}$ are then modular forms of weight~$k$. We recall our standing convention from the end of~\S\ref{sec:conv}: we always impose meromorphicity at the cusp and hence global sections of $\omega^{\otimes k}$ are (weakly holomorphic) modular forms. For cohomology theories valued in modular forms, it is customary to double the degree and take the dual grading as follows. 

\begin{defn}
Define the graded commutative algebra of modular forms, $\MF$ whose $2k^{\rm th}$ graded piece~$\MF^{2k}$ is weakly holomorphic modular forms of weight $-k$, and whose $(2k+1)$st graded piece~$\MF^{2k+1}$ is zero.
\end{defn}

\begin{rmk} We recall that 
$$
\MF=\C[c_4,c_6,\Delta^{\pm 1}]/(c_4^3-c_6^2-1728\Delta)
$$
where $c_4=\frac{E_4}{2\zeta(4)}$ and $c_6=\frac{E_6}{2\zeta(6)}$ are normalized Eisenstein series (for our conventions, see~\eqref{eq:Eisenstein}), and $\Delta=\eta^{24}$ is the discriminant. With our grading conventions, $|c_j|=|E_j|=-2j$, and $|\Delta|=-24$. Finally, since $\Delta$ is invertible, we have $\MF^\bullet\simeq \MF^{\bullet+24}$. 
\end{rmk}



%
%

Consider now the case of $G$ connected with maximal torus~$T<G$ and Weyl group $W=N(T)/ T$, and let $X_*(T)={\rm ker}(\mathfrak{t}\to T)$ be the cocharacter lattice. 

\begin{defn} \label{defn:Loo} Let $\ell$ be a $W$-invariant positive definite inner product on $\mf{t}$ satisfying $\ell(n, n) \in 2\mb{Z}$ for $n \in X_*(T)$. The \emph{level $\ell$ Looijenga line} $\mc{L}_\ell$ is a holomorphic line bundle on the substack $[(\mathcal{E}^{\vee, {\rm rk}(T)}\cq W)\sq \SL_2(\Z)]\subset \Bun_G(\EE)$ whose sections are $W$-invariant holomoprhic functions $f$ on $\HH\times\mf{t}_\C\simeq \mb{H} \times \mf{t} \times \mf{t}$ (using the isomorphism~\eqref{eq:BunGhol}) satisfying the $X_*(T)^{\oplus 2}\rtimes \SL_2(\Z)$-equivariance properties 
\beq
f(\tau, z+m+n\tau) &=& \exp\Big(-\pi i(2\ell(n,z) + \ell(n, n) \tau)\Big)f(\tau, z) \nonumber\\
f(\gamma \cdot \tau, z/(c\tau+d)) &=&  \exp \Big( \pi i c(c \tau + d)^{-1} \ell(z,z)\Big)f(\tau, z)\nonumber
\eeq
for $z\in \mf{t}_\C$, $n,m\in X_*(T)$, and $\gamma\in \SL_2(\Z)$, where we use that $\mathcal{E}^{\vee, {\rm rk}(T)} \simeq [(\mb{H} \times \mf{t}_\C)\sq X_*(T)^{\oplus 2}]$. 
\end{defn}

\begin{rmk} As noted in Example~\ref{ex:holomorphic}, if $G$ in addition has torsion-free fundamental group, then the above in fact defines a holomorphic line bundle on $\Bun_G(\EE)$ (as opposed to simply some substack thereof). If one additionally supposes that $G$ is simple and simply-connected, we have that $W$-invariant positive definite inner products~$\ell$ are naturally identified with elements of~$\H^4(BG;\Z)\simeq \Z$. We often make this identification, taking~$\ell\in \Z$ in these cases. \end{rmk} 

Now consider the case that $G$ is finite so that $\pi_0(G)=G$ and $\mathcal{C}^2(G)=\mathcal{C}^2[G]$. For a 3-cocycle $\ell\colon G\times G\times G\to \U(1)$ defining a class $[\ell]\in \H^3(BG;\U(1))\simeq \H^4(BG;\Z)$, Freed and Quinn~\cite{FreedQuinn} construct a line bundle on $[\mathcal{C}^2(G)/G\times\SL_2(\Z)]$; see also~\cite[\S2]{GanterHecke}. The vector space of sections of this line bundle is the value of Chern--Simons theory for the group~$G$ on the torus. 
An explicit cocycle for the line bundle on the groupoid $\mathcal{C}^2(G)/G$ is given by (e.g., see~\cite[\S3.4]{Willerton})
\beq
&&\frac{\ell(g,h_1,h_2)\ell(gh_2g^{-1},g,h_1)\ell(gh_1g^{-1},gh_2g^{-1},g)}{\ell(g,h_2,h_1)\ell(gh_1g^{-1},g,h_2)\ell(gh_2g^{-1},gh_1g^{-1},g)}\in C^\infty(\mathcal{C}^2(G)\times G;\U(1))\label{eq:conj}
\eeq
where $h=(h_1,h_2)\in \mathcal{C}^2(G)$, $g\in G$. 
Pulling back along
$$
\Bun_G(\EE)\to [\mathcal{C}^2(G)/\SL_2(\Z)\times G]
$$
defines line bundles $\mathcal{L}^\ell$ on $\Bun_G(\EE)$. Since $\mathcal{C}^2(G)/\SL_2(\Z)\times G$ is a discrete groupoid, $\mathcal{L}^\ell$ canonically comes with the structure of a holomorphic line bundle on $\Bun_G(\EE)$. 

\begin{rmk}\label{rmk:pi0}
By virtue of~\eqref{eq:conj}, $\mathcal{L}^\ell$ does \emph{not} pull back from the coarse quotient,
$$
\Bun_G(\EE)=(\HH\times\mathcal{C}^2(G))/(G\times\SL_2(\Z))\to (\HH\times\mathcal{C}^2(G)\cq G)/\SL_2(\Z),
$$
but rather depends on the conjugation action of $\pi_0(G)=G$ on $\mathcal{C}^2(G)$. This motivated our Definition~\ref{defn:BunG} of $\Bun_G(\EE)$. The line bundles $\mathcal{L}^\ell$ are important for applications of equivariant elliptic cohomology to generalized moonshine (see the next remark) and discrete torsion~\cite{Vafatorsion,Sharpediscrete,Andodiscrete,DBE_Equiv}. 
\end{rmk}

\begin{rmk}\label{rmk:moonshine}
For $G=\M$ the monster group, there is a particularly interesting line bundle on $\Bun_G(\EE)$. Johnson-Freyd has identified a class $[\omega^\sharp]\in \H^4(B\M;\Z)$ of order~24~\cite{TheoMoonshine}. The resulting line bundle $\mathcal{L}^{\omega^\sharp}$ is important in generalized moonshine~\cite{Mason,CarnahanI}. Roughly, the statement of generalized moonshine is that there exists a lift of the modular function $J(\tau)\in \mathcal{O}(\HH)$ to a holomorphic section $J_{\omega^\sharp}\in \Gamma(\Bun_\M(\EE);\mathcal{L}^{\omega^\sharp})$ so that the restriction of $J_{\omega^\sharp}$ to the trivial $\M$-bundle recovers the function~$J$. Explicitly, such a lift is data for each pair of commuting elements in the monster with compatibility properties for conjugation and $\SL_2(\Z)$-actions. Ganter has explained how the section $J_{\omega^\sharp}$ (and certain important additional properties, see~\cite[\S1.1]{GanterHecke}) can be rephrased in the language of equivariant elliptic cohomology and its power operations. It is expected that a deeper picture will emerge from a better geometric understanding of equivariant elliptic cohomology~\cite{Morava}.
\end{rmk}

\section{A cocycle model for equivariant elliptic cohomology}\label{sec:users}

We briefly review the notation used throughout this section. A pair of commuting elements in a compact Lie group~$G$ is denoted by~$h = (h_1, h_2)\in \mathcal{C}^2(G)$. For $g\in G$, let $ghg^{-1}=(gh_1g^{-1},gh_2g^{-1})$ denote the conjugate commuting pair. For $G$ acting on a manifold~$M$, let $M^h=M^{h_1,h_2}$ denote the submanifold of~$M$ fixed by~$h_1$ and~$h_2$. The connected component of the identity of $G$ is denoted $G_0< G$, and let~$G_0^h=(G_0)^h<G_0$ denote the subgroup fixed by the conjugation action of $h_1$ and $h_2$ on $G_0$. When $h_1,h_2\in G_0$ are in the connected component of the identity, we note that $G_0^h=C(h)=C(h_1,h_2)$ is the centralizer of~$h_1$ and~$h_2$. Let $\mf{g}$ be the Lie algebra of $G$, equipped with its adjoint $G$-action. We observe that the Lie algebra of $G_0^h$ is~$\mf{g}^h$, the subalgebra of $\fg$ fixed by the adjoint action of $h_1,h_2$. Let $\mf{t}_{\mf{g}^h}$ be a maximal commuting subalgebra of $\mf{g}^h$, or equivalently, $\mf{t}_{\mf{g}^h}$ is the Lie algebra of a maximal torus of the identity component of~$G_0^h$.

\subsection{The sheaf of equivariant elliptic cocycles on $\Bun_G(\EE)$}\label{sec:defn}

The goal of this subsection is to assign to a $G$-manifold $M$ a complex of sheaves $\dEll^\bullet_G(M)$ on $\Bun_G(\EE)$. Roughly, $\dEll^\bullet_G(M)$ is constructed from stitching together $G_0^h$-equivariant de~Rham complexes of~$M^h$ for all $h\in \mathcal{C}^2(G)$. To state an important compatibility condition between these complexes, we require some control over how fixed point sets $M^h$ vary with~$h$. Recall the notation $T_h^\epsilon:=B_\epsilon(\mf{t}_{\mf{g}^h}) \times B_{\epsilon}(\mf{t}_{\mf{g}^h})\subset \mf{t}_{\fg^h}\times \mf{t}_{\fg^h}$ from~\eqref{eq:BGopen}. Informally, the following lemma states that fixed points~$M^h$ get smaller for small deformations of~$h$ parameterized by~$T_h^\epsilon$.


\begin{lem} \label{lem:BlockGetzler} Fix a $G$-manifold $M$. For any pair of commuting elements $h=(h_1,h_2)\in \mathcal{C}^2(G)$, there exists a real number~$\delta=\delta(h)>0$ such that for all $\epsilon\in (0,\delta)$, any $(X_1,X_2)\in T_h^\epsilon$ has the property that
\beq
M^{h'}\subset M^{h}, \qquad G_0^{h'}<G_0^h,\label{eqlem:BlockGetzler}
\eeq
where $h'=(h_1',h_2') = (h_1e^{X_1}, h_2e^{X_2})$. Furthermore, for such $(X_1, X_2) \in T_h^\epsilon$, the vector fields on $M^h$ generated by~$X_1$ and $X_2$ vanish on the submanifold~$M^{h'}\subset M^{h}$. 
\end{lem}

\bp 
Block and Getzler~\cite[Lemma~1.3]{BlockGetzler} prove a version of the above for fixed points by a single element $h$ deformed by an element $X \in \mf{g}^h$, and we apply their lemma twice. Indeed, their lemma provides a ball so that $X_1\in B_{\epsilon_1}(\mf{t}_{\mf{g}^{h_1}})$ has $M^{h_1e^{X_1}} \subset M^{h_1}$ and another ball so that $X_2\in B_{\epsilon_2}(\mf{t}_{\mf{g}^{h_2}})$ has $M^{h_2e^{X_2}} \subset M^{h_2}$. Setting $\epsilon={\rm min}(\epsilon_1,\epsilon_2)$ and considering the special case where $h_1$ and $h_2$ commute and $X_1,X_2\in \mf{t}_{\mf{g}^h}$, we have
$$
M^{h'}=M^{h_1',h_2'}=M^{h_1e^{X_1}, h_2e^{X_2}} = M^{h_1e^{X_1}} \cap M^{h_2e^{X_2}} \subset M^{h_1} \cap M^{h_2} = M^{h_1, h_2}=M^h.
$$ 
The same argument applied to $M=G_0$ shows that $G_0^{h'}<G_0^h$, proving the first assertion of the lemma. For the second assertion, for a fixed $X_1,X_2$ we
claim that $M^{h_1e^{\epsilon X_1},h_2e^{\epsilon X_2}}\subset M^{h_1,h_2}$ is constant for~$\epsilon>0$ sufficiently small. This is a consequence of our assumption (see the end of~\S\ref{sec:conv}) that~$M$ can be equivariantly embedded in a finite-dimensional $G$-representation. Indeed, choosing an invariant metric on the representation reduces to the case that $M=V$ is a (finite-dimensional) vector space with its standard action by $G={\rm O}(V)$ the orthogonal group.  Fixed-point loci are now determined by the subspace corresponding to eigenvalue $1$ and the constancy of the fixed-point loci for $\epsilon$ sufficiently small may be seen directly. Hence $X_1$ and $X_2$ both vanish on~$M^{h'}$. 
\ep

Given a commutative algebra~$A$ over $\C$ and a $G$-manifold $M$, define
\beq
\Omega^{\bullet}_G(M;A[\beta, \beta^{-1}]):=\bigoplus_j \mathcal{O}_0(\mf{g}_\C;\Omega^j(M;A[\beta, \beta^{-1}]))^G\qquad  |\beta|=- 2, \label{eq:Cartancplx}
\eeq
as the stalk at $0\in \mf{g}_\C$ of $G$-invariant holomorphic functions on $\mf{g}_\C$ valued in $\Omega^j(M;A[\beta, \beta^{-1}])$ for the adjoint $G$-action on $\mf{g}_\C$ and the $G$-action on $\Omega^j(M)$ induced by the $G$-action on~$M$. Endow this with the total grading from differential forms and the graded ring $A[\beta,\beta^{-1}]$. Equip $\Omega^{\bullet}_G(M;A[\beta,\beta^{-1}])$ with the Cartan differential $Q=d-\beta^{-1} \iota$ (see~\eqref{eq:Cartandiff}). 

\begin{rmk}
We emphasize that the grading on~\eqref{eq:Cartancplx} is not the usual $\Z$-grading for equivariant de~Rham cohomology with values in a graded ring: germs of holomorphic functions on the Lie algebra (and in particular, polynomials) are in degree zero in~\eqref{eq:Cartancplx}. This choice is essentially forced upon us because power series rings can only be equipped with the trivial grading. Another option is to work with $\Z/2$-graded complexes. This approach works well for equivariant elliptic cohomology at a fixed elliptic curve (as in~\cite{Grojnowski}), where elliptic cohomology is 2-periodic and so can be computed by a $\Z/2$-graded complex. However for families of elliptic curves this 2-periodicity is typically broken, so one can no longer express equivariant elliptic cohomology in terms of a sheaf of~$\Z/2$-graded complexes. 
\end{rmk}

We now give the main definition of the paper. Below let~$U_0\subset \HH$ denote an open subset and $U_h^{\epsilon} \subset \mathcal{C}^2[G]$ an open subsheaf defined in Corollary~\ref{cor:opencover1}. We recall from Lemma~\ref{lem:cov2} that $U_h^{\epsilon}\simeq T_h^\epsilon\cq W^h$, and we demand that $\epsilon<\delta=\delta(h)$ is chosen so that deformations in $T_h^\epsilon$ satisfy~\eqref{eqlem:BlockGetzler}, where the existence of such a $\delta$ is guaranteed by Lemma~\ref{lem:BlockGetzler}. 

\begin{defn} \label{defn:ellcocycle}
Given a $G$-manifold $M$, for each $U_0\times U_h^{\epsilon}\subset \HH\times\mathcal{C}^2[G]$ defined above, let $\dEll^\bullet_G(M)(U_0\times U_h^{\epsilon})$ denote the cdga whose elements $\alpha\in \dEll^\bullet_G(M)(U_0\times U_h^{\epsilon})$ are sets $\{\alpha_{h'}\}$ of compatible equivariant differential forms defined as follows. For all $h'\in \mathcal{C}^2(G)$ such that $[h']\in U_h^\epsilon\subset \mathcal{C}^2[G]$, we require the data of
\beq
\alpha_{h'}\in  \Omega^{\bullet}_{G_0^{h'}} (M^{h'};\mathcal{O}(U_0)[\beta, \beta^{-1}])\simeq \mathcal{O}(U_0;\Omega^{\bullet}_{G_0^{h'}} (M^{h'};\C[\beta, \beta^{-1}])).\label{eq:cocycledef}
\eeq
These data are required to satisfy:
\begin{enumerate}
\item \emph{Invariance}: for all $g\in G_0$, there is an equality of equivariant differential forms
\beq
\alpha_{h'}=g^*\alpha_{gh'g^{-1}}\label{eq:conjugationinvariance}
\eeq
where $g^*$ is the pullback along left multiplication by $g$, $M^{h'}\to M^{gh'g^{-1}}$. 

\item \emph{Analyticity}: for $h_1'=h_1e^{X_1}$ and $h_2'=h_2e^{X_2}$ as in Lemma~\ref{lem:BlockGetzler}, there is an equality of germs of holomorphic functions on $(\mf{t}_{\fg^{h'}})_\C\subset \fg^{h'}_\C$ determined by $\alpha_h$ and $\alpha_{h'}$,
\beq
&&\alpha_{h'}(X)={\rm res} (\alpha_{h}(X+(X_1-\tau X_2)))\in \Omega^{\bullet}_{G_0^{h'}}(M^{h'};\O(U_0)[\beta,\beta^{-1}])\label{eq:analyticity}
\eeq
where $X\in (\mf{t}_{\mf{g}^{h'}})_{\mb{C}}$ and ${\rm res}\colon \Omega^{\bullet}_{G_0^h}(M^{h};\O(U_0)[\beta,\beta^{-1}])\to \Omega^{\bullet}_{G_0^{h'}}(M^{h'};\O(U_0)[\beta,\beta^{-1}])$ is restriction along the inclusions $M^{h'}\hookrightarrow M^h$ and $G_0^{h'}<G_0^h$ from Lemma~\ref{lem:BlockGetzler}. 
\end{enumerate}
Define the differential $Q$ on $\dEll^\bullet_G(M)(U_0\times U_h^{\epsilon})$ via the Cartan differentials applied to each~$\alpha_{h'}$ in~\eqref{eq:cocycledef}. Compatibility of these differentials with the analyticity property follows from the last statement in Lemma~\ref{lem:BlockGetzler}. 
If $([g],\gamma)\in \pi_0(G)\times \SL_2(\Z)$ maps the open subsheaf $U_0\times U_h^{\epsilon}$ to an open subsheaf of $V_0\times U_{k}^{\epsilon'}\subset \HH\times\mathcal{C}^2[G]$, define the restriction map
\beq
\dEll^\bullet_G(M)(V_0\times U_{k}^{\epsilon'})\to \dEll^\bullet_G(M)(U_0\times U_h^{\epsilon})\label{eq:restrictionmapscdga}
\eeq 
in terms of maps for each~$[h']\in U_h^{\epsilon}$, 
%
\beq
\Omega^{\bullet}_{G_0^{\gamma\cdot gh'g^{-1}}} (M^{\gamma\cdot gh'g^{-1}};\mathcal{O}(V_0)[\beta, \beta^{-1}]){\longrightarrow} \Omega^{\bullet}_{G_0^{h'}} (M^{h'};\mathcal{O}(U_0)[\beta, \beta^{-1}])\label{eq:equiv}
\eeq
using the pullback of functions along $U_0\hookrightarrow V_0$ and the isomorphisms $M^{h'}\simeq M^{\gamma\cdot gh'g^{-1}}$, $G_0^{h'}\simeq G_0^{\gamma\cdot gh'g^{-1}}$. We then modify this pullback map by rescaling the Lie algebra~$\mf{g}^{h'}$ by $c\tau+d$ (so $z\in (\mf{g}^{h'})^\vee$ is sent to $\frac{z}{c\tau+d}$), and sending $\beta$ to $\beta/(c\tau+d)$. The map~\eqref{eq:equiv} is independent of the choice of lift $g\in G$ because of the $G_0$-invariance property~\eqref{eq:conjugationinvariance}. 
\end{defn}

\begin{rmk}
By the analyticity property, the data of $\{\alpha_{h'}\}$ comprising $\alpha\in \dEll^\bullet_G(M)(U_0\times U_h^{\epsilon})$ is completely determined by the equivariant differential form~$\alpha_h$. We carry around the additional data~$\{\alpha_{h'}\}$ to make the definition of the restriction maps~\eqref{eq:restrictionmapscdga} transparent. 
\end{rmk}

\begin{prop} The cdgas $\dEll^\bullet_G(M)(U_0\times U_h^{\epsilon})$ have the following properties. 
\begin{enumerate}
\item There is a uniquely determined sheaf of cdgas $\dEll^\bullet_G(M)$ on $\Bun_G(\EE)$ that when viewed as a $\pi_0(G)\times \SL_2(\Z)$-equivariant sheaf on $\HH\times \mathcal{C}^2[G]$ takes the value $\dEll^\bullet_G(M)(U_0\times U_h^{\epsilon})$ on the open subsheaf $U_0\times U_h^\epsilon\subset \HH\times \mathcal{C}^2[G]$, and has restriction maps determined by~\eqref{eq:restrictionmapscdga} and~\eqref{eq:equiv}. 
\item The sheaf $\dEll_G^\bullet(M)$ is functorial in the pair $(M,G)$: a $G$-equivariant map $M\to M'$ induces a morphism of sheaves of cdgas $\dEll^\bullet_G(M')\to \dEll^\bullet_G(M)$ on $\Bun_G(\EE)$, and a homomorphism $G\to H$ induces a map of sheaves of cdgas~$\dEll^\bullet_H(M)\to \dEll^\bullet_G(M)$ over the map $\Bun_G(\EE)\to \Bun_H(\EE)$. 
\item The sheaf $\dEll_G^\bullet(M)$ has Mayer--Vietoris sequences: a $G$-invariant open cover of~$M$ determines an exact sequence of sheaves of cdgas on~$\Bun_G(\EE)$.
\end{enumerate}
\end{prop}

\bp
Part (1) follows from realizing the values $\dEll_G^\bullet(M)(U_0\times U_h^{\epsilon})$ as coming from a sheaf on $\Bun_G(\EE)$ constructed via Proposition~\ref{prop:BunGsheaf}. To spell this out, we first observe that the cdga $\dEll^\bullet_G(M)(U_0\times U_h^\epsilon)$ arises as the $W^h$-invariant elements of cdga associated with $\HH\times T_h^\epsilon$ using the isomorphism of cdgas, 
\beq
\Omega^{\bullet}_{G_0^h} (M^{h};\mathcal{O}(U_0)[\beta, \beta^{-1}])\simeq \Omega^\bullet_{T_{G_0^{h}}}(M^{h};\mathcal{O}(U_0)[\beta,\beta^{-1}])^{W^h},\label{eq:Thesheaf}
\eeq
where $T_{G_0^h}<G_0^h$ is a maximal torus of the connected component of the identity and $W^h$ is defined in~\eqref{eq:Weyldefn}.
We next realize the right hand side of~\eqref{eq:Thesheaf} as the $W^h$-invariant global sections of a sheaf of cdgas on $\HH\times T_h^\epsilon$. Indeed, define a sheaf that to the basis of open subsets $U_0\times T_{h'}^{\epsilon'}\subset \HH\times T_h^\epsilon$ assigns data $\{\alpha_k\in \Omega^\bullet_{T_{G_0^{k}}}(M^{k};\mathcal{O}(U_0)[\beta,\beta^{-1}])\}$ where $k$ and $h'$ differ by (the exponential of) a point in~$T_{h'}^{\epsilon'}$, and the $\alpha_k$ satisfy the compatibility condition~\eqref{eq:analyticity}. Note that since $\HH\times T_h^\epsilon$ is a manifold, a sheaf is completely determined by its values on a basis of open subsets (e.g., see~\cite[Lemma~6.30.6]{stacks-project}) and the open subsets $U_0\times T_{h'}^{\epsilon'}\subset \HH\times T_h^\epsilon$ afford such a basis. We promote this to an $W^h$-equivariant sheaf on $\HH\times T_h^\epsilon$ by pulling back along the left action of the normalizer $N(T_{G_0^{h}})$ on $M^h$ and the conjugation action of $N(T_{G_0^{h}})$ on $T_{G_0^{h}}$. Since the action of the torus $T_{G_0^{h}}$ on equivariant cohomology is trivial, this action by the normalizer factors through the quotient $W^h=N(T_{G_0^{h}})/T_{G_0^{h}}$ as desired. This completes the construction of the $W^h$-equivariant sheaf of cdgas on each $\HH\times T_h^\epsilon$ whose $W^h$-invariant sections are computed by~\eqref{eq:Thesheaf}. Next we define the equivariance data required in Proposition~\ref{prop:BunGsheaf} using maps completely analogous to~\eqref{eq:equiv}. The required conditions in Proposition~\ref{prop:BunGsheaf} follow from the fact the Weyl action is compatible with the equivariant structure using the invariance property~\eqref{eq:conjugationinvariance}. Naturality of this structure for restrictions and the action by $\pi_0(G)\times \SL_2(\Z)$ follows from composition of pullbacks and that the action on coefficients going into the definition of~\eqref{eq:equiv} is a well-defined action. 

Property (2) in the statement of the present proposition follows from the naturality of the data~\eqref{eq:cocycledef} in the $G$-manifold $M$, while property (3) follows from applying naturality of Mayer--Vietoris sequences in equivariant de~Rham cohomology to the data~\eqref{eq:cocycledef}. 
\ep

\begin{defn}\label{defn:cohom}
Define the sheaves $\Ell_G^\bullet(M)$ on $\Bun_G(\EE)$ as the cohomology sheaves of the complex of sheaves $\dEll_G^\bullet(M)$. 
\end{defn}

\subsection{Stalks and elliptic Atiyah--Segal completion} \label{sec:complete}

The Atiyah--Segal completion theorem compares equivariant K-theory and Borel equivariant K-theory. When applied to $G$ acting on the point, it gives a map
\beq
\Rep(G)\simeq \K_G(\pt)\to \K_G(\pt)_{\widehat{I}}\simeq \K(BG)\label{eq:AtiyahSegal}
\eeq
witnessing the target as a completion of $\Rep(G)$ at the augmentation ideal~$I$. For the delocalized K-theory of Block--Getzler~\cite{BlockGetzler}, Vergne~\cite{DufloVergne}, and Vergne~\cite{Vergne}, the completion map~\eqref{eq:AtiyahSegal} takes the form
$$
C^\infty(G)^G\to C^\infty_0(\mf{g})^G
$$
sending a smooth class function $f\in C^\infty(G)^G$ to its germ at~$e\in G$. This is related to~\eqref{eq:AtiyahSegal} by sending a $G$-representation $\rho$ to its character ${\rm Tr}(\rho)\in C^\infty(G)^G$ (as a smooth class function), Taylor expanding this class function at~$e\in G$, and identifying $C^\infty_0(\mf{g})^G$ with $\Z/2$-graded Borel equivariant cohomology of a point. We now explain a similar structure in equivariant elliptic cohomology. 

Given a pair of commuting elements $h\in \mathcal{C}^2(G)$, we obtain a map
\beq
j_h\colon \HH\sq \Gamma \to \Bun_G(\EE)\label{eq:jh}
\eeq
that on objects includes at $[h]\in \mathcal{C}^2[G]$, and where $\Gamma<\pi_0(G)\times \SL_2(\Z)$ is the stabilizer of $[h]\in \mathcal{C}^2[G]$.

\begin{prop}\label{prop:partialstalk}
There is an isomorphism of $\Gamma$-equivariant sheaves on $\HH$ that on $U\subset \HH$ is given by
$$
j_h^*\dEll_G^\bullet(M)(U)\simeq \Omega^{\bullet}_{G_0^h} (M^{h};\mathcal{O}(U)[\beta, \beta^{-1}]). 
$$
\end{prop}
\bp
We recall that $j_h^*\dEll_G^\bullet(M)(U)$ is the colimit of the values of $\dEll_G^\bullet(M)$ for open subsheaves containing the image of $j_h$. 
We define a map,
\beq
j_h^*\dEll_G^\bullet(M)(U)\to \Omega^{\bullet}_{G_0^h} (M^{h};\mathcal{O}(U)[\beta, \beta^{-1}]). \label{eq:itsamap}
\eeq
that extracts $\alpha_h\in \Omega^{\bullet}_{G_0^h} (M^{h};\mathcal{O}(U)[\beta, \beta^{-1}])$ from  the data of a section over $U\times U_h^\epsilon$ from~\eqref{eq:cocycledef}. However, the analyticity condition implies that a section on $U\times U_h^\epsilon$ for $\epsilon$ sufficiently small is completely determined by~$\alpha_h$. Therefore, the map~\eqref{eq:itsamap} is indeed an isomorphism. \ep


\begin{cor} 
There is an isomorphism of $\Gamma$-equivariant sheaves on $\HH$ that on $U\subset \HH$ is given by
$$
j_h^*\Ell_G^k(M)(U)\simeq \left\{\begin{array}{ll} \displaystyle \H_{G_0^h}^\ev (M^h;\mathcal{O}(U)) & k={\rm even} \\ \displaystyle \H_{G_0^h}^\odd (M^h;\mathcal{O}(U)) & k={\rm odd}\end{array}\right.
$$
i.e., the 2-periodic Borel equivariant cohomology of $M^h$ with its $G_0^h$ action. 
\end{cor} 

A special case of the above takes $h=(e,e)$, where~\eqref{eq:jh} is the map $j_e\colon \Mell\to \Bun_G(\EE)$ that assigns to each elliptic curve the trivial $G$-bundle over that curve. This allows us to compare complex analytic equivariant elliptic cohomology to the Borel equivariant refinement of $\dEll(M)$ as follows. For a $G$-manifold $M$, the Borel equivariant refinement is the sheaf $\dEll_{G,{\rm Bor}}$ on $\Mell$ whose value on $U\subset \HH\to \Mell$ is
\beq
\dEll^\bullet_{G,{\rm Bor}}(M)(U)=\Omega^{\bullet}_{G} (M;\mathcal{O}(U)[\beta, \beta^{-1}]),\label{eq:Borel}
\eeq
i.e., a chain complex that computes the 2-periodic Borel equivariant cohomology of $M$.

\begin{thm}[Atiyah--Segal completion]\label{thm:complete} Let $j_e\colon \HH/\SL_2(\Z)\times \pi_0(G)\hookrightarrow \Bun_G(\EE)$ be the inclusion at the trivial $G$-bundle. The $\pi_0(G)$-invariant sections of $j_e^*\dEll_G^\bullet(M)$ determines a natural isomorphism of sheaves of commutative differential graded algebras on~$\Mell$
$$
j_e^*\dEll_G^\bullet(M)^{\pi_0(G)}\stackrel{\sim}{\to} \dEll^\bullet_{G,{\rm Bor}}(M).
$$
\end{thm}

\bp
This follows from Proposition~\ref{prop:partialstalk}, using that 
$$
\Omega^{\bullet}_{G_0} (M;\mathcal{O}(U)[\beta, \beta^{-1}])^{\pi_0(G)}\simeq \Omega^{\bullet}_{G} (M;\mathcal{O}(U)[\beta, \beta^{-1}])
$$
from the isomorphism $\pi_0(G)\simeq G/G_0$. 
\ep


\subsection{Holomorphicity and periodicity}
The terminology ``analytic" in Definition~\ref{defn:ellcocycle} is justified by the following. 

\begin{prop} \label{prop:analytic}
There is a canonical isomorphism of sheaves on $\Bun_G(\EE)$
\beq
\dEll_G^0(\pt) \simeq \mc{O}_{\Bun_G(\EE)}.
\eeq
For a $G$-manifold $M$, this implies $\dEll^\bullet_G(M)$ is canonically a sheaf of $\mathcal{O}_{\Bun_G(\EE)}$-modules. 
\end{prop}
\bp
Let~$U_0\subset \HH$ be an open subset and $U_h^{\epsilon} \subset \mathcal{C}^2[G]$ an open subsheaf with $\epsilon$ sufficiently small to satisfy Lemma~\ref{lem:BlockGetzler}. Define a map 
$$
\mc{O}_{\Bun_G(\EE)}\to \dEll_G^0(\pt)
$$
that to an $f\in \mathcal{O}_{\Bun_G(\EE)}(U_0\times U_h^{\epsilon})$ assigns
$$
j_h^*f\in \mathcal{O}_0(\mf{t}_{\mf{g}^h};\mc{O}(U_0))^{W^h}\simeq \Omega^{0}_{G_0^h}(\pt;\mc{O}(U_0)[\beta,\beta^{-1}])
$$ 
for each $h$ in the image. We observe that this is a well-defined morphism of sheaves: the holomorphic condition from Definition~\ref{defn:BunGholo} implies the analytic condition in Definition~\ref{defn:ellcocycle}. By inspection, this morphism induces an isomorphism on stalks, so gives the claimed isomorphism of sheaves. \ep

 Recall that $\omega^{\otimes k}$ denotes the sheaf on $\Bun_G(\EE)$ that is the pullback of $\omega^{\otimes k}$ in $\mc{O}$-modules under the forgetful morphism $\Bun_G(\EE) \to \Mell$. 
 The sheaves $\dEll_G^{\bullet}(M)$ exhibit a form of $2$-periodicity twisted by~$\omega$:

\begin{prop} [Twisted Bott periodicity] \label{prop:Bottperiodicity} There is a natural isomorphism of chain complexes of sheaves of $\mathcal{O}$-modules 
\begin{eqnarray*} 
\dEll_G^{\bullet+2}(M) \otimes \omega &\to& \dEll_G^{\bullet}(M) \\ (\alpha, f) &\mapsto&  \alpha \beta f, 
\end{eqnarray*} 
where the formula defines a map on local sections, and the tensor product is in sheaves of $\mathcal{O}$-modules. \end{prop}
\bp
Indeed, one may identify $\beta^{-1}$ with a trivializing section of the Hodge bundle pulled back along~$\HH\to \Mell$. This is because the transformation properties of $\beta^{-1}$ under $\SL_2(\Z)$ are precisely as for a section of the Hodge bundle. For a fixed family of elliptic curves associated to an open submanifold $U\subset \HH$, the Hodge bundle trivializes and so $\Gamma(U;\omega^{\otimes j})\simeq \beta^{-j}\mathcal{O}(U)$. Hence, in the displayed formula above, $\beta f$ is a cocycle in $\dEll_G^{-2}(\pt)$, pulled back along the canonical map $M\to \pt$. This verifies the claimed isomorphism.
\ep
\begin{rmk}
The failure of~$\beta$ to be $\SL_2(\Z)$-invariant means that it does not descend to a global section along $\HH\to \Mell$. This implies that global sections of~$\dEll_G^\bullet(M)$ over $\Bun_G(\EE)$ are no longer 2-periodic. However, $\Delta^{-1}\beta^{-12}$ is a globally defined invertible element of degree~24 where $\Delta\in\mathcal{O}(\HH)$ is the discriminant (an invertible weight~12 modular form). This gives the global sections of $\dEll_G^\bullet(M)$ a 24-periodicity. 
\end{rmk}

\begin{prop}\label{prop:point} For $G$ acting on $\pt$, the $G$-equivariant elliptic cocycles are the sheaves
\beq \dEll_G^n(\pt) \simeq \begin{cases} \omega^{\otimes -n/2} &\text{for }n\text{ even} \\ 0 &\text{for }n\text{ odd} \end{cases}\label{eq:dEll^n} \eeq
equipped with the zero differential. 
\end{prop}
\bp
This follows from Propositions~\ref{prop:analytic} and~\ref{prop:Bottperiodicity}, together with the observation that there are no nonzero cocycles in odd degrees.\ep

As a corollary to Proposition~\ref{prop:Bottperiodicity}, we also have twisted Bott periodicity at the level of cohomology sheaves.

\begin{cor} [Twisted Bott periodicity] There are natural isomorphisms of sheaves $$\Ell^{\bullet+2}_G(M) \otimes \omega \to \Ell^\bullet_G(M).$$ \end{cor}

\begin{rmk} \label{rmk:derived}
The chain complex of sheaves $\dEll_G^\bullet(M)$ allows one to consider spaces of \emph{derived} global sections over $\Bun_G(\EE)$, i.e., the hypercohomology groups $\HH^*(\Bun_G(\EE);\dEll_G^\bullet(M))$. The applications considered in this paper either concern the non-derived sections (i.e., $\HH^0(\Bun_G(\EE);\dEll_G^\bullet(M))$) or the sheaf $\dEll_G^\bullet(M)$ rather than its global sections. Therefore, we postpone a full discussion of the derived global sections for future work. For now we observe that the higher cohomology is often nontrivial. For example, when~$G=\U(1)$, and $M=\pt$, Serre duality shows that the nonvanishing cohomology groups are
$$
\HH^0(\Bun_{\U(1)}(\EE);\dEll_{\U(1)}^\bullet(\pt))\simeq \H^0(\EE^\vee;\omega^{\bullet/2})=\MF^\bullet,
$$
$$
\HH^1(\Bun_{\U(1)}(\EE);\dEll_{\U(1)}^\bullet(\pt))\simeq \H^1(\EE^\vee;\omega^{\bullet/2})=\MF^{\bullet},
$$
where we emphasize that the degree~1 cohomology is a module over the degree zero cohomology. These fit together as
$$
\R\Gamma(\Bun_{\U(1)}(\EE);\dEll_{\U(1)}^\bullet(\pt))\simeq \MF^\bullet\oplus \MF^{\bullet-1}
$$ 
where as an algebra, the second summand is a square zero extension of the first. This is compatible with Gepner and Meier's computation of $\U(1)$-equivariant topological modular forms over $\Z$~\cite[Corollary~1.2]{LenartDavid}. 
\end{rmk}

\subsection{Twistings and loop group representations}
When $G$ is connected with torsion-free fundamental group, Proposition~\ref{prop:analytic} shows that we shouldn't expect $\dEll^0_G(\pt)$ to have many interesting global sections. Indeed, the only global holomorphic functions on an elliptic curve are constant, and so (for example) global sections of $\dEll^0_G(\pt)$ pull back from functions on~$\Mell$: the group plays no role. 
More generally, if~$G$ acts on~$M$ so that the stabilizers are connected with torsion-free fundamental group, the global sections of $\dEll^\bullet_G(M)$ are just the ordinary de~Rham complex valued in modular forms. However, global sections are more interesting for twisted versions of equivariant elliptic cohomology. 

\begin{defn}
Let $\mathcal{L}$ be a holomorphic line bundle on $\Bun_G(\EE)$. The \emph{$\mathcal{L}$-twisted equivariant elliptic cocycles} of a $G$-manifold $M$ is the sheaf of chain complexes $\dEll_G^\bullet(M)\otimes \mathcal{L}$ on~$\Bun_G(\EE)$. 
\end{defn}

Note that in this definition, the category of twists for $G$-equivariant elliptic cohomology is the category of line bundles on $\Bun_G(\EE)$. An important class of twists for finite groups $G$ come from the Freed--Quinn line bundles; see~\eqref{eq:conj}. For (twisted) elliptic Thom and Euler classes associated with connected Lie groups, the relevant line bundles are the Looijenga line bundles from Definition~\ref{defn:Loo}. The $\mathcal{L}$-twisted $G$-equivariant elliptic cohomology of a point for the Looijenga twist is the sheaf whose sections are (nonabelian) theta functions. 

\begin{prop}\label{prop:Verlinde} Let $G$ be a simple, simply connected compact Lie group and $\mathcal{L}_\ell$ be the level $\ell$ Looijenga line bundle over $\Bun_G(\EE)$. Then the space of global sections of the twisted equivariant elliptic cohomology sheaf 
$$
\Gamma(\Bun_G(\EE);\dEll_G^{\bullet}(\pt)\otimes \mathcal{L}_\ell) \simeq \Rep^\ell(LG) \underset{\MF^0}{\otimes} \MF^{\bullet}
$$
is the free module over the ring of modular forms generated by super characters of positive energy representations of~$LG$ at level~$\ell$, i.e., the vector space underlying the Verlinde algebra. 
\end{prop}
\bp
From Proposition~\ref{prop:point} and the remarks after its proof, $\Gamma(\Bun_G(\EE);\dEll_G^{\bullet}(\pt))\simeq \MF^\bullet$. Then the claim follows from the fact that global sections of the Looijenga line bundle are spanned by the characters of loop group representations at the relevant level (e.g., see~\cite[Corollary~10.9]{Ando}), and super characters are differences of ordinary characters.
\ep

\subsection{A few examples}\label{sec:nonequiv}

\begin{ex}[Trivial group]\label{ex:trivialgroup}
When $G=\{e\}$ is the trivial group, $\dEll^\bullet(M)=\dEll_{\{e\}}^\bullet(M)$ is a sheaf on $\Bun_{\{e\}}(\EE)= \Mell$ whose value on $U\subset \HH\to \Mell$ is the 2-periodic de~Rham complex
$$
\dEll^\bullet(M)(U)=\Omega^{\bullet} (M;\mathcal{O}(U)[\beta, \beta^{-1}]).
$$
The global sections of $\dEll(M)$ are given by
$$
\Gamma(\Mell;\dEll^\bullet(M))\simeq \bigoplus_{j+k=\bullet} \Omega^j (M;\MF^k)
$$
i.e., the de~Rham complex of differential forms valued in modular forms. 
We note that the complexification of topological modular forms, $\TMF\otimes \C$, is an ordinary cohomology theory with values in the graded ring $\TMF(\pt)\otimes \C\cong \MF$. Hence, the above complex is a cocycle model for~$\TMF\otimes \C$. 
\end{ex}
\begin{rmk}
When $G=\{e\}$, the map from global sections to derived global sections of $\dEll(M)$ is a quasi-isomorphism; this follows from the $\SL_2(\Z)$-action on $\HH$ has finite stabilizer groups. 
\end{rmk}

The first nontrivial example in ordinary equivariant cohomology is for the $\U(1)$-action on $S^2$ via rotation. We compute the equivariant elliptic cohomology for this example. 

\begin{ex}[$\U(1)$ acting on $S^2$]\label{ex:S2}

 Consider the sheaf~$\dEll_{\U(1)}(S^2)$ on $\Bun_{\U(1)}(\EE)\simeq \Euni^\vee$ for the $\U(1)$-action on $S^2=\CP^1$ that rotates the sphere about an axis. We observe that for $h\in \U(1)\times \U(1)$ not equal to $(e,e)$, the fixed points are the poles $(S^2)^h=\{{\rm poles}\}$. Hence by Proposition~\ref{prop:point} we have the isomorphism of sheaves
\beq
\dEll_{\U(1)}^{2k}(S^2)|_{\Euni^\vee\setminus \{0{\ \rm section}\}}\simeq \omega^{-k}\oplus \omega^{-k}.\label{eq:sheaf0}
\eeq
Next, a section defined in a small neighborhood of the zero section in $\Euni^\vee\simeq\Bun_{\U(1)}(\EE)$ is determined by an element of the stalk~$\Omega^{\bullet}_{\U(1)}(S^2;\mathcal{O}(\HH)[\beta,\beta^{-1}])$. Extending this section to a larger neighborhood demands a compatibility with~\eqref{eq:sheaf0} given by the restriction map
$$
\Omega^{\bullet}_{\U(1)}(\{{\rm poles}\};\mathcal{O}(\HH)[\beta,\beta^{-1}])\simeq \Big(\mathcal{O}_0(\C;\mathcal{O}(\HH)[\beta,\beta^{-1}])\Big)^{\oplus 2}\stackrel{{\rm res}}{\to} \Big(\mathcal{O}_0(\C\setminus 0;\mathcal{O}(\HH)[\beta,\beta^{-1}])\Big)^{\oplus 2},
$$
where the last map restricts a germ of a holomorphic function at the origin in $\C$ to one in a punctured neighborhood of the origin. We then identify this neighborhood in $\C$ with a neighborhood of zero in $\Euni^\vee$. We identify a function on this punctured neighborhood with one on a punctured neighborhood of $0$ in $\Euni^\vee$ (which is uniquely specified from the analytic condition in Defintion~\ref{defn:ellcocycle}). Finally, we identify $\beta^n$ with a section of $\omega^{\otimes n}$. A global section of $\dEll_{\U(1)}(S^2)$ is therefore given by an element of $\Omega^{\bullet}_{\U(1)}(S^2;\mathcal{O}(\HH)[\beta,\beta^{-1}])^{\SL_2(\Z)}=\Omega^{\bullet}_{\U(1)}(S^2;\MF)$ (i.e., a Borel equivariant cocycle) whose restriction to the poles is a constant function on the Lie algebra of~$\U(1)$. This gives a cocycle-level description. We compute the associated cohomology (using different techniques) in~\S\ref{ex:S22}. 
\end{ex}

\begin{ex}[$\U(1)$ acting on $S^2$ with twisting]
More generally, $\Bun_{\U(1)}(\EE)$ admits Looijenga line bundles $\mc{L}_{\ell}$ parametrized by levels $\ell \in \H^4(B\U(1); \mb{Z}) \simeq \mb{Z}$ and we have the twisted equivariant elliptic cohomology $\dEll_{\U(1)}(S^2) \otimes \mc{L}_{\ell}$. We recall that global sections of $\dEll_{\U(1)}(\pt) \otimes \mc{L}_{\ell}$ over $\Bun_{\U(1)}(\EE)$ are $\theta$-functions (or \emph{Jacobi forms}) of index $\ell$, 
$$\Gamma(\Bun_{\U(1)}(\EE), \dEll_{\U(1)}(\pt) \otimes \mc{L}_{\ell}) \simeq \bigoplus_k \JF_{k, \ell},$$ 
where $\JF_{k, \ell}$ is the space of (weakly holomorphic) Jacobi forms of weight $k$ and index $\ell$ sitting in degree~$-2k$. 
Global sections of $\dEll_{\U(1)}(S^2) \otimes \mc{L}_{\ell}$ are then given by 
$$ \Gamma(\Bun_{\U(1)}(\EE), \dEll_{\U(1)}(S^2) \otimes \mc{L}_{\ell}) \simeq \Omega^{\bullet}_{\U(1)}(S^2; \MF) \underset{\Omega^{\bullet}_{\U(1)}(\{\mathrm{poles}\}; \MF)}{\times} \Omega^{\bullet}(\{\mathrm{poles}\}; \JF_{*, \ell}).$$ We note that there is no twisting by $\mc{L}$ in the factor $\Omega^{\bullet}_{\U(1)}(S^2; \MF)$ because this corresponds to the pair of commuting elements $(e,e)$, over which $\mc{L}$ trivializes canonically. In words, elements of the above fibered product are $\U(1)$-equivariant, modular form-valued differential forms on $S^2$ whose restriction to the poles are germs of specified Jacobi forms of index $\ell$. 
\end{ex}
%
%

\section{Comparing with Grojnowski's equivariant elliptic cohomology}\label{sec:Groj}

Throughout this section, let $G$ be a connected Lie group, $M$ a $G$-manifold, and $\tau\in \HH$ a point defining a (marked) elliptic curve $E=E_\tau=\C/\langle \tau,1\rangle$. Grojnowski~\cite{Grojnowski} constructs a $\Z/2$-graded sheaf~$\Ell_G^{\rm Groj}(M)$ of $\O_{(E^{\vee})^{\times {\rm rk}(T)}\cq W}$-modules on $(E^{\vee})^{\times {\rm rk}(T)}\cq W$. We compare this with the restriction of the cocycle model from the previous section along the map $(E^{\vee})^{\times {\rm rk}(T)}\cq W\to \Bun_G(\EE)$. We recall $\Ell_G^\bullet(M)$ are the cohomology sheaves of $\dEll_G^\bullet(M)$; see~Definition~\ref{defn:cohom}. 

\begin{thm} \label{thm:Groj}The pullback in $\mc{O}$-modules of the sheaf $\Ell_G^\bullet(M)$ along $\{\tau\}\times (T\times T)\cq W\hookrightarrow \Bun_G(\EE)$ is naturally isomorphic to the 2-periodic version of the $\Z/2$-graded sheaf~$\Ell_G^{\rm Groj}(M)$.
\end{thm}

\subsection{A review of Grojnowski's equivariant elliptic cohomology}
Our presentation below hews closely to the original source~\cite{Grojnowski}, though we also refer to~\cite[\S3]{RosuEquivariant} for an accounting when~$G=\U(1)$. 
To begin, let $G=T$ be a torus. As we have fixed a curve $E=E_\tau$, the identification~\eqref{eq:univT} specializes to $T\times T\simeq (E^{\vee})^{\times {\rm rk}(T)}$, endowing the smooth Lie group $T\times T$ with a complex analytic structure. For $h=(h_1,h_2)\in T\times T\simeq (E^{\vee})^{\times  {\rm rk}(T)}$, let 
$$
L_h\colon (E^{\vee})^{\times  {\rm rk}(T)}\times T\times T\stackrel{h\cdot}{\to} T\times T\simeq (E^{\vee})^{\times  {\rm rk}(T)}
$$ 
denote the map induced by left multiplication by~$h$, and let $T_h^\epsilon \subset T\times T\simeq (E^{\vee})^{\times  {\rm rk}(T)}$ be an open subset specified as in~\eqref{eq:BGopen} where $\epsilon$ satisfies the hypothesis of Lemma~\ref{lem:BlockGetzler}, i.e., for all $h'=(h_1',h_2')\in T_h^\epsilon$, we have $M^{h'}\subset M^{h}$ and $G^{h'}<G^h$. We note that in the abelian case, $T_h^\epsilon=U_h^\epsilon$ for $U_h^\epsilon$ defined as in Corollary~\ref{cor:opencover1}.

\begin{defn}[Grojnowski]
For a $T$-manifold~$M$, define a sheaf~$\Ell_T^{\rm Groj}(M)$ on $(E^{\vee})^{\times  {\rm rk}(T)}$ that assigns to each~$T_h^\epsilon \subset (E^{\vee})^{\times  {\rm rk}(T)}$ with $\epsilon$ satisfying Lemma~\ref{lem:BlockGetzler} the $\Z/2$-graded $\mathcal{O}_{T_h^\epsilon}$-module
$$
\Gamma(T_h^\epsilon;\Ell_T^{\rm Groj}(M)):=L_h^*\left(\H_T(M^h)\otimes_{\H_T(\pt)} L_{h^{-1}}^*\mathcal{O}(T_h^\epsilon)\right)
$$
where we have identified the polynomial algebra $\H_T(\pt)\simeq S(\mf{t}_\C^\vee)$ (in degree zero) with a subalgebra of holomorphic functions~$\mathcal{O}(T_e^\epsilon)$ using that $L_{h^{-1}}(T_h^\epsilon)=T_0^\epsilon\subset \mf{t}_\C$ can be identifies with an open ball around $0\in \mf{t}_\C$. Define restriction maps on open subsets $T_{h'}^\epsilon \subset T_h^\epsilon$ with $h\notin T_{h'}^{\epsilon'}$ by
\beq
&&i^*\colon \H_{T}(M^h)\otimes_{\H_T(\pt)} L_{h^{-1}}^*\mathcal{O}(T_h^\epsilon) \to \H_{T}(M^{h'})\otimes_{\H_T(\pt)} L_{h'^{-1}}^*\mathcal{O}(T_{h'}^\epsilon) \label{whatatransition}
\eeq
induced by pulling back along the inclusions $M^{h'}\hookrightarrow M^h$ and the isomorphism from pulling back along left multiplication
$$
L_{h'^{-1}h}^*\colon L_{h^{-1}}^*\mathcal{O}(T_h^\epsilon)\stackrel{\sim}{\to}L_{h'^{-1}}^*\mathcal{O}(T_{h'}^\epsilon).
$$ 
By Atiyah--Bott localization~\cite[Theorem~3.5]{AtiyahBott},~\eqref{whatatransition} is an isomorphism, and so this data on opens defines a sheaf without a need for further sheafification.
\end{defn}

Let $T<G$ be a maximal torus with associated Weyl group $W=N_G(T)/T$. For $h=(h_1,h_2)\in T\times T\subset G\times G$ a pair of commuting elements in the torus, observe that~$T<G^h$ is a maximal torus for the connected component of the identity of $G^h$. In this case the finite group~$W^h$ defined in~\eqref{eq:Weyldefn} has the simpler description $W^h = (G^h \cap N(T)) / T$, and in particular, $W^h< W$ is a subgroup of the Weyl group of~$G$.


Let $T_h^\epsilon \subset T\times T$ be an open subset as above satisfying the additional properties
$$
w\cdot T_h^\epsilon=T_h^\epsilon \ \ {\rm for} \ \ w\in W^h, \qquad w\cdot T_h^\epsilon \cap T_h^\epsilon=\emptyset \ \ {\rm for} \ \  w \in W\setminus W^h.
$$ 
This condition can always be arranged by shrinking the previously defined $T_h^\epsilon$, as $W^h$ is finite and~$w\cdot h=h$ for $w\in W^h$. Let $W\cdot T_h^\epsilon\subset T\times T$ denote the orbit of $T_h^\epsilon$ under the action of the Weyl group, so that $W\cdot T_h^\epsilon$ is a $W$-invariant open subset of $(E^{\vee})^{\times {\rm rk}(T)}$. 

\begin{defn}[Grojnowski]
For $G$ connected, define a sheaf~$\Ell_G^{\rm Groj}(M)$ on $(E^{\vee})^{\times {\rm rk}(T)}\cq W$ that assigns to each $W$-invariant open~$W\cdot T_h^\epsilon \subset E^{\vee, {\rm rk}(T)}$ the $\Z/2$-graded $\mathcal{O}_{W\cdot T_h^\epsilon}$-module
\beq
\Gamma(W\cdot T_h^\epsilon,\Ell_G^{\rm Groj}(M))&:=&L_h^*\left(\H_{G^h}(M^h)\otimes_{\H_{G^h}(\pt)} L_{h^{-1}}^*\mathcal{O}(W\cdot T_h^\epsilon)^{W^h}\right)\label{eq:Grojdefn}
\eeq
where we use the isomorphism $\H_{G^h}(\pt)\simeq \H_T^{W^h}(\pt)$ to define the tensor product. The transition maps are defined identically to those in the case that $G=T$.

\end{defn}

\subsection{The comparison map}

\begin{proof}[Proof of Theorem~\ref{thm:Groj}] 
Let $U\subset T\times T$ be a $W$-invariant open subset. Then a section of $\dEll^\bullet_G(M)$ on $\{\tau\}\times U$ is given by the data of a collection
$$
\alpha_{h}\in \mathcal{O}_0(\mathfrak{g}^h_{\mb{C}};\Omega^\bullet(M^h;\C[\beta, \beta^{-1}]))^{G^h}\simeq \mathcal{O}_0(\mathfrak{t}_\C;\Omega^\bullet(M^h;\C[\beta, \beta^{-1}])^T)^{W^h}
$$
for all $[h]\in U$, which we identify with a $W^h$-invariant form on the right hand side. The $\{\alpha_h\}$ are required to satisfy the conjugation invariance and analyticity properties. The conjugation invariance property is equivalent to a condition on $\alpha_h$ and $\alpha_{h'}$ when $h'=w\cdot h$ for $w\in W$. Hence, the collection $\{\alpha_{w\cdot h}\}_{w\in W}$ is determined by a single $W$-invariant form 
$$
\alpha_{h}^W\in \Big( \bigoplus_{w \in W/W^h} \mathcal{O}_0(\mathfrak{t}_\C;\Omega^\bullet(M^{w \cdot h};\C[\beta,\beta^{-1}])^T) \Big)^{W}. 
$$
We observe the above element $\alpha_{h}^W$ determines an $\Omega^\bullet(M^{w \cdot h};\C[\beta,\beta^{-1}])$-valued holomorphic function on some ball $B_\epsilon(\mathfrak{t}_\C)\subset \mathfrak{t}_\C$ centered at $0\in \mf{t}_\C$ for some $\epsilon>0$
$$
\tilde{\alpha}_h^W\in \Big( \bigoplus_{w \in W/W^h} \mathcal{O}(B_\epsilon(\mathfrak{t}_\C);\Omega^\bullet(M^{w \cdot h};\C[\beta,\beta^{-1}])^T) \Big)^{W}.
$$
The image under the exponential map of such $\epsilon$-balls $B_\epsilon(\mathfrak{t}_\C)$ cover $U$. If each $\tilde{\alpha}_h^W$ is also closed under the Cartan differential, we obtain classes
\begin{eqnarray*} 
[\tilde{\alpha}_h^W] &\in& \Big( \bigoplus_{w \in W/W^h} \H(\mathcal{O}(B_\epsilon(\mathfrak{t}_\C);\Omega^\bullet(M^{w \cdot h};\C[\beta,\beta^{-1}])^T),Q) \Big)^W \\ &\simeq&  \Bigg( \Big( \bigoplus_{w \in W/W^h} \H_T(M^{w \cdot h}) \Big) \otimes_{\H_T(\pt)} \mathcal{O}(B_\epsilon(\mathfrak{t}_\C);\C[\beta,\beta^{-1}])\Bigg)^W 
\\ &\simeq& \Big( \H_T(M^h) \otimes_{\H_T(\pt)} \mc{O}(B_{\epsilon}(\mf{t}_{\mb{C}}); \mb{C}[\beta,\beta^{-1}]) \Big)^{W_h},
\end{eqnarray*}
where the last isomorphism comes from restricting to the summand labeled by the identity coset $W_h\subset W$. 
Then finally, the correspondence between 2-periodic cohomology and $\Z/2$-graded cohomology yields a class corresponding to $[\tilde{\alpha}_h^W]$ in Grojnowski's equivariant elliptic cohomology sheaf~\eqref{eq:Grojdefn} on each open ball $W\cdot T_h^\epsilon$. The analyticity condition guarantees that these classes glue to give a section of $\Ell_G^{\rm Groj}(M)$ over~$U$: the translations in Grojnowski's formulas are precisely the translations appearing in the analytic condition. This determines the morphism of sheaves in the statement of the theorem. To see that is an isomorphism of sheaves it suffices to demonstrate an isomorphism on stalks, but this is clear from the maps defined on each~$B_\epsilon(\mathfrak{t}_\C)$. 
\ep

\section{Comparing with Devoto's equivariant elliptic cohomology}\label{sec:Devoto}

In this section we compare our model with previous ones for $G$-equivariant elliptic cohomology where $G$ is finite. The definition of Devoto's equivariant elliptic cohomology we adopt is used by Ganter~\cite{GanterHecke} and Morava~\cite{Morava} in their studies of generalized moonshine; it is also the complexification of a version of equivariant elliptic cohomology appearing in the work of Baker and Thomas~\cite{BakerThomas}. These definitions are based on the early work of Devoto~\cite{DevotoII,DevotoI}, simplifying his construction over $\Z[1/2,1/3]$ to one over~$\C$, and replacing the congruence subgroup~$\Gamma_0(2)$ by the full modular group~$\SL_2(\Z)$. As such, we refer to this finite group version of equivariant elliptic cohomology as \emph{Devoto's equivariant elliptic cohomology,} $\Ell_G^{{\rm Dev}}(M)$, to be defined shortly; we first state the main theorem of the section.

\begin{thm} \label{thm:Devoto} For $G$ finite, the space of global sections of $\Ell_G^\bullet(M)$ over $\Bun_G(\EE)$ is Devoto's equivariant elliptic cohomology over $\C$, i.e., $$\Gamma(\Bun_G(\EE), \Ell_G^{\bullet}(M)) \simeq \Ell_G^{{\rm Dev},\bullet}(M).$$ \end{thm}

\subsection{A review of Devoto's equivariant elliptic cohomology}

Consider 
\beq
\SL_2(\Z)\circlearrowright \Big(\bigoplus_{h\in \mathcal{C}^2(G)}\H^\bullet(M^{h};\mathcal{O}(\HH))\Big)^{G}\label{eq:Devtil}
\eeq
where $\SL_2(\Z)$ acts through the indexing set $\mathcal{C}^2(G)=\Hom(\Z^2,G)$ by precomposition and on $\HH$ through the usual fractional linear transformations. The $G$-invariants in~\eqref{eq:Devtil} are taken with respect to the $G$-action by conjugation on $\mathcal{C}^2(G)$ and left multiplication by~$g$ on fixed point sets, $M^{h}\to M^{ghg^{-1}}$. The following is an adaptation of~\cite[Definition~3.2]{DevotoI} to complex coefficients and the full modular group~$\SL_2(\Z)$.

\begin{defn}\label{defn:Devoto}
Let $G$ be a finite group and $M$ a $G$-manifold. Define \emph{Devoto's $G$-equivariant elliptic cohomology} of $M$ as a subspace
$$
\Ell_G^{{\rm Dev},k}(M)\subset \bigoplus_j \Big(\bigoplus_{h\in \mathcal{C}^2(G)}\H^j(M^{h};\mathcal{O}(\HH))\Big)^{G}
$$
whose $j^{\rm th}$ summand consists of functions that transform under the $\SL_2(\Z)$-action~\eqref{eq:Devtil} with weight $(j-k)/2$ (so in particular, $j-k$ must be even for the $j^{\rm th}$ summand to be nonzero).
\end{defn}

\begin{rmk} We recall our convention (stated at the end of~\S\ref{sec:conv}) that $\mc{O}_{\mb{H}}$ denotes the sheaf whose sections are holomorphic functions with polynomial growth along any geodesic that escapes to the infinity, or equivalently, the meromorphicity condition at the cusps. We re-emphasize this point now as the modularity condition for classes in Devoto's equivariant elliptic cohomology will typically be for finite-index subgroups of $\SL_2(\Z)$. Our convention agrees with the usual notion of weakly holomorphic modular forms of higher level (in terms of imposing meromorphicity at \emph{all} cusps). Devoto imposes this same condition in terms of Fourier expansions in $e^{2\pi i \tau/|G|}=q^{1/|G|}$ for $\tau\in \HH$.\end{rmk}


\subsection{The comparison map}

\begin{proof}[Proof of Theorem~\ref{thm:Devoto}] 
We evaluate $\Ell_G^\bullet(M)$ on the cover $\HH\times \mathcal{C}^2(G)$ of $\Bun_G(\EE)$, and then compute the action of $G\times\SL_2(\Z)$. On this cover, a section is the data of 
$$
[\alpha_h]\in \H^\bullet(M^h;\mathcal{O}(\HH)[\beta,\beta^{-1}])
$$
for each pair of commuting elements $h\in \mathcal{C}^2(G)$ satisfying a conjugation invariance property and an $\SL_2(\Z)$-equivariance property; the analytic property in this case is trivially satisfied because the Lie algebra is the zero vector space (and $\mathcal{C}^2(G)$ is discrete). Conjugation invariance implies that the set $\{[\alpha_h]\}_{h\in \mathcal{C}^2(G)}$ determines a class
$$
[\alpha]\in \Big(\bigoplus_h \H^\bullet(M^h;\mathcal{O}(\HH)[\beta,\beta^{-1}])\Big)^G.
$$
Finally, the $\SL_2(\Z)$-invariance extracts Devoto's $\Ell_G^{{\rm Dev},k}(M)$: invariant classes come with a power of $\beta$ that reads off the weight.
\ep

%

\section{Loop group representations and cocycle representatives of Thom classes}\label{sec:Thom}

In this section we construct cocycle representatives of universal Euler and Thom classes in complex analytic equivariant elliptic cohomology. These refinements can be understood as coming from the representation theory of loop groups, giving an elliptic version of Chern--Weil theory: characteristic classes in (non-equivariant) complex analytic elliptic cohomology are determined by universal equivariant classes, which in turn are constructed out of Lie-theoretic data. The approach applies to both real and complex vector bundles, recovering universal characteristic classes for the complexifications of the ${\rm MString}$- and $\MU\langle 6\rangle$-orientations of $\TMF$, respectively. 

The construction of equivariant elliptic Thom classes for a fixed elliptic curve was first sketched by Grojnowski~\cite[\S2.5-2.6]{Grojnowski}. For a torus $T$, aspects of Grojnowski's $T$-equivariant Thom classes were studied further by Rosu and Ando~\cite{RosuEquivariant,AndoOrientation}. The cocycle-level description below is new, which leads to a more explicit description of the underlying characteristic classes and verification of their claimed properties. 



\subsection{Review from K-theory and ordinary cohomology}
Let $V\to M$ be a real $d$-dimensional oriented vector bundle. The Thom class of $V$ in ordinary cohomology $[{\rm Th}_V]\in \H^{d}_{\rm cs}(V)$ has compact vertical support and the property that the exterior product map
$$
\H^\bullet(M)\stackrel{\sim}{\to} \H^{\bullet+d}_{\rm cs}(V),\qquad [\alpha]\mapsto [\alpha]\boxtimes[{\rm Th}_V]
$$
is an isomorphism, called the \emph{Thom isomorphism}. The \emph{Euler class} $[{\rm Eu}_V]\in \H^d(M)$ is the pullback of $[{\rm Th}_V]$ along the zero section $0\hookrightarrow V$. The Thom class determines pushforwards in cohomology using the Pontrjagin--Thom collapse map. The Euler and Thom class are both natural for the oriented vector bundle~$V$, and so they are determined by the Euler and Thom classes for the universal bundle over~$\BSO(n)$. An analogous story for complex vector bundles again yields universal Euler and Thom classes for the universal bundle over~$\BU(n)$. The cohomology of these classifying spaces is the equivariant cohomology of a point 
$$
\H(\BU(n))=\H_{{\rm U}(n)}(\pt),\qquad \H(\BSO(n))=\H_{\SO(n)}(\pt)
$$
so that universal Euler and Thom classes are (canonically) classes in the coefficient ring of equivariant cohomology. 

For K-theory one again finds Euler and Thom classes living in equivariant refinements. However, the existence of refinements is a more interesting question if one considers the non-Borel version of equivariant $\K$-theory coming from equivariant vector bundles. For example, we recall that the Euler class of a complex vector bundle $V$ in K-theory is the class underlying the virtual vector bundle $\Lambda^{\ev}V\ominus \Lambda^\odd V$, or equivalently, the $\Z/2$-graded vector bundle given by the total exterior bundle $\Lambda^\bullet V$ of~$V$ (compare Example~\ref{ex:Kthycplx} below). By universal properties, the Euler class is determined by the corresponding virtual vector bundle on $\BU(n)$. It admits an equivariant refinement,
\beq
\Rep(\U(n))=\K_{\U(n)}(\pt)\stackrel{{\rm completion}}{\longrightarrow} \K(\BU(n))\label{eq:KtheoryEuler}
\eeq
as the virtual representation $\Lambda^{\ev} R\ominus \Lambda^{\odd} R$ where $R$ is the defining representation of $\U(n)$ and~\eqref{eq:KtheoryEuler} is the Atiyah--Segal completion map. 
There is a similar story for equivariant refinements of K-theory Thom classes, as well as analogous constructions in KO-theory built from spinor representations $\bS^+\ominus\bS^-$~\cite[Part III]{ABS}; see also Example~\ref{eq:spinor}. 

Below we construct refinements of Euler and Thom classes in elliptic cohomology with two goals that run in analogy to~\eqref{eq:KtheoryEuler}: (i) refine pre-existing non-equivariant classes, and (ii) give representation-theoretic meaning to the refinements. 





\subsection{Positive energy representations and the Weierstrass sigma function}\label{sec:Weier}

We briefly review positive energy representations of loop groups; a standard reference is~\cite[\S9]{PressleySegal}. The \emph{loop group} of a compact Lie group consists of smooth maps $LG:=C^\infty(S^1,G)$ endowed with pointwise multiplication. Transgression of a class $[\ell]\in \H^4(BG;\Z)$ determines an $S^1$ central extension
\beq
1 \to S^1\to \widetilde{LG}\to LG \to 1.\label{eq:extension}
\eeq
When $G$ is simple and simply connected, $\H^4(BG;\Z)\simeq \Z$ and $[\ell]\in \Z$ is customarily called the \emph{level} of the extension~\eqref{eq:extension}; we use the terminology of a level for arbitrary compact Lie groups even though $[\ell]\in \H^4(BG;\Z)$ need not be determined by an integer. 

The loop group $LG$ has an $S^1$-action from precomposing with the rotation action of $S^1$ on itself; this is called the action of the \emph{energy circle} (to distinguish from the \emph{central circle} in~\eqref{eq:extension}). The action of the energy circle lifts to $\widetilde{LG}$ and one may form the semidirect product $S^1\ltimes \widetilde{LG}$. A \emph{positive energy representation of $LG$} is a representation $\mathcal{H}$ of $S^1\ltimes \widetilde{LG}$ such that the weight spaces of the energy circle are finite-dimensional and bounded below. In more detail, define the $n^{\rm th}$ weight space $\mc{H}_n \subset \mc{H}$ by 
$$
\mc{H}_n := \{v \in \mc{H} \mid \lambda \cdot v = \lambda^n v, \ \forall \lambda \in S^1\},
$$ 
where $\lambda\in S^1$ is in the energy circle. Then the positive energy condition demands that $\mc{H}_n$ is finite-dimensional for all $n$ and there exists some $N \in \mb{Z}$ such that $\mc{H}_n = \{0\}$ for $n < N$. 

Given a positive energy representation $\mathcal{H}$ of a loop group $LG$, each weight space $\mathcal{H}_n$ is itself a finite-dimensional $G$-representation; this comes from restricting along the embedding $G\hookrightarrow LG$ as the constant loops, using that the energy circle acts trivially on the constant loops and that the central extension~\eqref{eq:extension} canonically splits over the constant loops. This leads to the definition of the \emph{character} of a positive energy representation as the formal power series 
\beq
\chi_{\mc{H}} := \sum_{n \ge N} \chi_n q^n\in (\Rep(G)\otimes \C)(\!(q)\!)\label{eq:LGchar}
\eeq
where $\chi_{n}\in \Rep(G)\otimes \C\subset C^\infty(G)^G$ can be identified with the character of the associated $G$-representation on $\mathcal{H}_n$, i.e., a class function on~$G$. More generally, we can consider formal differences $\mathcal{H}^+\ominus \mathcal{H}^-$ of positive energy representations of loop groups, where then the character~\eqref{eq:LGchar} is the difference $\chi_{n}=\chi_{n}^+-\chi_{n}^-$ of characters of $G$-representations on $\mathcal{H}_n$. Equivalently, one can consider $\Z/2$-graded positive energy representations whose characters are defined using the supertrace. Finally, we note that~\eqref{eq:LGchar} is only defined as a formal series in $q$; however, it turns out that taking~$q = \exp(2\pi i \tau)$ gives a (convergent) power series expansion of a holomorphic function of $\tau\in \HH$. By reduction to maximal tori, the following example essentially determines this holomorphic behavior.

\begin{ex}\label{eq:spinor}
First recall that there are two irreducible representations of $\Spin(2)$, denoted $\bS^+$ and~$\bS^-$. The character of the $\Z/2$-graded representation $\bS^+\ominus\bS^-$ is
\beq
2i\sinh(2 \pi i z)=\exp(\pi i z) - \exp(-\pi i z)=y^{1/2}-y^{-1/2} \label{eq:spinschar}
\eeq
where $y = \mathrm{exp}(2\pi i z)\in C^\infty(\SO(2))$ is the coordinate function that admits a square root when pulled back along the double cover $\Spin(2)\to \SO(2)$.  
Generalizing to loop groups, 
consider the level $[\ell] \in \H^4(B\mathrm{Spin}(2);\Z) \simeq \Z$ given by one of the two generators. For one choice, there will exist no positive energy representations, while for the other there are precisely four, typically denoted $\mb{S}^{\pm}, \mb{S}_{\pm}$, e.g., see~\cite[\S1.2]{Liu}. For the intended applications in equivariant elliptic cohomology, the most important of these is the $\Z/2$-graded positive energy representation $\mb{S}^+ \ominus \mb{S}^-$ of $L\Spin(2)$ whose (super) character is the holomorphic function
\beq
\chi_{\mb{S}^+ \ominus \mb{S}^-}(q,y) = (y^{1/2}-y^{-1/2}) \prod_{n > 0} (1 - q^n y)(1 - q^ny^{-1})\in \mathcal{O}(\HH\times \C).\label{eq:LGcharchi}
\eeq
Below we refer to $\mb{S}^+ \ominus \mb{S}^-$ as the level~1 \emph{vacuum representation} of $L\Spin(2)$. Similarly, there is a $\Z/2$-graded level~1 vacuum representation of $L\Spin(2n)$  whose super character is
$$
\chi(q,y_1,y_2,\dots,y_n)=\prod_{i=1}^n \chi_{\mb{S}^+ \ominus \mb{S}^-}(q,y_i)
$$
where the functions $\chi_{\mb{S}^+ \ominus \mb{S}^-}(q,y_i)$ on the right are given by~\eqref{eq:LGcharchi} for coordinates $y_i$ on the standard maximal torus of $\SO(2n)$ pulled back to $\Spin(2n)$. 
\end{ex}

\begin{ex} Starting instead with the virtual representation $1\ominus R$ of $\U(1)$ with character $1-y\in C^\infty(\U(1))$, there is an extension to a vacuum representation of the loop group $L\U(1)$ with character
\beq
\chi(q,y)=\left((1-y)\prod_{n>0}(1-q^ny)(1-q^ny^{-1})\right),\label{eq:LU1char}
\eeq
e.g., see~\cite[\S11]{Ando}. 
\end{ex}

A common normalization of the character~\eqref{eq:LGcharchi} leads to the \emph{Weierstrass sigma function}
\beq
\sigma(\tau,z)&:=&\left((y^{1/2}-y^{-1/2})\prod_{n>0}\frac{(1-q^ny)(1-q^ny^{-1})}{(1-q^n)^2}\right)\nonumber\\
&=&z\exp\left(-\sum_{k>0} \frac{E_{2k}(\tau) z^k}{2k}\right)\in \mathcal{O}(\HH\times \C)\label{eq:AHR}
\eeq
where $\tau\in \HH$, $q=\exp(2\pi i \tau)$, $z\in \C$, $y = \exp(2 \pi i z)$ (so $y^{1/2} = \exp(\pi i z)$), and $E_{2k}(\tau)\in \mc{O}(\HH)$ is the $2k^{\rm th}$ Eisenstein series, defined as
\beq
&&E_{2k}(\tau)=\sum_{n,m\in \Z_*^2} \frac{1}{(m\tau+n)^{2k}},\quad \qquad \Z_*^2=\{(n,m)\in \Z^2 \mid (n,m)\ne (0,0)\}\label{eq:Eisenstein}
\eeq
for $k>1$, and we take $E_{2}$ to be the standard holomorphic version of the 2nd Eisenstein series (the above sum is conditionally convergent when $k=1$). The equality~\eqref{eq:AHR} was first demonstrated by Zagier~\cite{Zagiermodular}; see also~\cite[Proposition~10.9]{AHR}. 

%

We also consider the normalization of the $L\U(1)$ character~\eqref{eq:LU1char}, 
\beq
\upsilon(\tau,z):=-y^{1/2}\sigma(\tau,z)=\left((1-y)\prod_{n>0}\frac{(1-q^ny)(1-q^ny^{-1})}{(1-q^n)^2}\right).\label{eq:upsilondef}
\eeq

The relevance of the $\sigma$-function in elliptic cohomology originally came by way of the \emph{Witten genus}, the Hirzebruch genus associated with the power series 
\beq
\Wit(q,z)=\frac{z}{\sigma(\tau,z)}=\exp\left(\sum_{k>0} \frac{E_{2k}z^k}{2k}\right)\in \C[\! [ z,q]\!].\label{eq:Witten}
\eeq
The families refinement of this genus leads to the $\MU\langle 6\rangle$ and ${\rm MString}$ orientations of topological modular forms reviewed in the next subsection~\cite{HopkinsICM94,AHSI,AHR}. 


\subsection{Characteristic classes in (non-equivariant) elliptic cohomology over $\C$}
The universal elliptic cohomology theory of topological modular forms has Thom and Euler classes for $\U\langle 6\rangle$ and ${\rm O}\langle 8\rangle$ bundles. We recall that $\BU\langle 6\rangle$ is the classifying space for complex vector bundles with vanishing first and second Chern classes, $c_1=c_2=0$; $\BO\langle 8\rangle={\rm BString}$ classifies real vector bundles with vanishing first and second Stiefel--Whitney classes, as well as the vanishing of the fractional first Pontryagin class, $w_1=w_2=\frac{p_1}{2}=0$. These classifying spaces sit in the diagram
\beq
 \begin{tikzpicture}[baseline=(basepoint)];
\node (A) at (0,0) {$\BU\langle 6\rangle$};
\node (B) at (3,0) {$\BSU$};
\node (C) at (6,0) {$\BU$};
\node (D) at (0,-1.5) {${\rm BString}$};
\node (E) at (3,-1.5) {${\rm BSpin}$};
\node (F) at (6,-1.5) {$\BSO.$};
\draw[->] (A) to (B);
\draw[->] (B) to (C);
\draw[->] (C) to (F);
\draw[->] (E) to (F);
\draw[->] (A) to (D);
\draw[->] (B) to (E);
\draw[->] (D) to (E);
\path (0,-.75) coordinate (basepoint);
\end{tikzpicture}\label{diag:BUBO}
\eeq

\begin{rmk}
The notation $\BU\langle 6\rangle$ and $\BO\langle 8\rangle$ comes from canonical maps $\BU\langle 6\rangle\to \BU$ and $\BO\langle 8\rangle\to \BO$ giving the 5-connected cover and the 7-connected cover in the Whitehead towers of $\BU$ and $\BO$, respectively. 
\end{rmk}

Let $\MString=\MO\langle 8\rangle$ and $\MU\langle 6\rangle$ denote the Thom spectra associated with the universal bundles on ${\rm BString}=\BO\langle 8\rangle$ and $\BU\langle 6\rangle$, respectively. The $\sigma$-orientation of $\TMF$ is a map of ring spectra
\beq
\sigma\colon \MString\to \TMF\label{eq:sigmaorientation}
\eeq
that assigns a (vertically) compactly supported Thom class ${\rm Th}_V\in \TMF^{m}_{\rm cs}(V)$ to an $m$-dimensional real vector bundle $V\to M$ with string structure~\cite{AHSI,AHR}. 
%
The Chern--Dold character is a map
$$
{\rm ch}\colon \TMF(M)\to \H(M;\TMF(\pt)\otimes \C)\simeq \H(M;\MF)
$$
from $\TMF$ to ordinary cohomology with coefficients in the graded ring of modular forms~$\MF$. The Riemann--Roch theorem compares the Thom class in $\TMF$ with the Thom class $u_V$ in ordinary cohomology by means of the commuting square
\beq
\begin{tikzpicture}[baseline=(basepoint)];
\node (A) at (0,0) {$\TMF^\bullet(M)$};
\node (B) at (5,0) {$\H^\bullet(M;\MF)$};
\node (C) at (0,-1.5) {$\TMF^{\bullet +m}_{\rm cs}(V)$};
\node (D) at (5,-1.5) {$\H^{\bullet +m}_{\rm cs}(V;\MF).$}; 
\draw[->] (A) to node [above=1pt] {$\ch$} (B);
\draw[->] (A) to node [left=1pt] {${\rm Th}_V $} (C);
\draw[->] (C) to node [above=1pt] {$\ch$} (D);
\draw[->] (B) to node [right=1pt] {$[u_V\cdot \Wit(V)^{-1}]$}(D);
\path (0,-.75) coordinate (basepoint);
\end{tikzpicture}\label{diag:WitRR1}
\eeq
where the vertical arrows are exterior multiplication with the indicated class, $\Wit(V)$ is the characteristic class associated with the power series~\eqref{eq:Witten}, and the cohomology groups $\TMF^{\bullet +m}_{\rm cs}(V)$ and $\H^{\bullet +m}_{\rm cs}(V;\MF)$ are with compact vertical support. This defines the elliptic Thom class in $ \H_{\rm cs}^m(V;\MF)$ as the class $[u_V\cdot \Wit(V)^{-1}]$. The elliptic Euler class is gotten by pulling back along the zero section, and is therefore $[e_V\cdot \Wit(V)^{-1}]$ where $e_V$ is the ordinary Euler class of $V$. We have two flavors of these classes, depending on whether~$V$ is real (as was assumed above) or complex, corresponding to the $\MU\langle 6\rangle$ or ${\rm MString}$ orientation respectively. These orientations are related by precomposing~\eqref{eq:sigmaorientation} with the map $\MU\langle 6\rangle\to \MO\langle 8\rangle$ coming from taking Thom spectra of universal bundles in~\eqref{diag:BUBO}. 

%

\subsection{Equivariant elliptic Euler forms} \label{sec:Euler}
Below, $\mf{t}\simeq \R^n$ denotes the Lie algebra of the standard maximal tori of $\U(n)$ and $\SO(2n)$, given by diagonal unitary matrices in the former case, and block diagonal matrices whose blocks are $2\times 2$ rotation matrices in the latter case. The standard (real) coordinates on $\mf{t}$ also give standard complex coordinates $\{z_1,\dots, z_n\}$ on $\mf{t}_\C$. 


\begin{defn} For $q=\exp(2\pi i\tau)$ with $\tau\in \HH$, consider the holomorphic functions on $\HH\times\mf{t}_\C$ given by
\beq
\sigma_{\SO(2n)}(\tau,z_1,\dots, z_n)&=&\prod_{j=1}^n \sigma(\tau,z_j)\in \mathcal{O}(\HH\times\mathfrak{t}_\C).\label{Eu:SO}
\eeq
\beq
&&\sigma_{\U(n)}(\tau,z_1,\dots, z_n)=\prod_{j=1}^n\upsilon(\tau,z_j)=\prod_{j=1}^n e^{-\pi i z_j} \sigma(\tau,z_j) \in \mathcal{O}(\HH\times\mathfrak{t}_\C)\label{Eu:U}
\eeq

\end{defn}

We recall that the transformation properties of the function $\sigma(\tau,z)$ with respect to the action of $\Z^2\rtimes \SL_2(\Z)$ on $\HH\times \C$ show that it is a Jacobi form of weight~$-1$ and index~$1/2$, e.g., see~\cite[Equations 140-141]{Ell1}. Equivalently, these transformations define a cocycle for a line bundle on the quotient $[\HH\times \C/\Z^2\rtimes \SL_2(\Z)]\simeq \EE^\vee\simeq \Bun_{\SO(2)}(\EE)$ for which~$\sigma(\tau,z)$ determines a section. The same reasoning shows that the transformation properties of the functions~\eqref{Eu:SO} and~\eqref{Eu:U} for the action of $X_*(T)^{\oplus 2}\rtimes \SL_2(\Z)$ on $\HH\times\mathfrak{t}_\C$ define cocycles for line bundles on $\Bun_{\Spin(2n)}(\EE)$ and $\Bun_{\U(n)}(\EE)$ respectively, where we consider descent of these functions along
\beq
&&\label{eq:covering} \HH\times\mathfrak{t}_\C\to \Bun_G(\EE)\qquad G=\U(n),\Spin(2n)
\eeq
where the map is from~\eqref{eq:BunGhol} and we use Example~\ref{ex:holomorphic} to understand the target. Here, the action of $X_*(T)^{\oplus 2}$ on $\HH\times \mf{t}_\C\simeq \HH\times \mf{t}\times\mf{t}$ is by the cocharacter lattice for $\Spin(2n)$ or $\U(n)$, whose quotient gives $(\HH\times \mf{t}\times\mf{t})/X_*(T)\simeq \HH\times T\times T$ for $T$ the maximal torus of $\Spin(2n)$ or $\U(n)$. In the spin case, the action of this cocharacter lattice on $\HH\times\mf{t}\times\mf{t}$ can be understood explicitly (since $\mf{t}$ is the Lie algebra of the maximal torus of $\SO(2n)$) by describing the cocharacter lattice of $\Spin(2n)$ as sublattice of the cocharacter lattice of $\SO(2n)$; e.g., see the proof of Proposition~\ref{thm:Euler2}. The reason for the intermediate descent to $\Bun_{\Spin(2n)}(\EE)$ in the spin case (rather than all the way to $\Bun_{\SO(2n)}(\EE)$) is described in Remark~\ref{rmk:partialdescent}. 


\begin{defn} \label{defn:generalA} For $G = \U(n)$ or $\Spin(2n)$, let $\mc{A}_G$ denote the holomorphic line bundle over $\Bun_{G}(\EE)$ defined by the transformation properties of $\sigma_{\U(n)}$ or $\sigma_{\SO(2n)}$, as described above. These holomorphic line bundles have preferred sections determined by $\sigma_{\U(n)}$ and $\sigma_{\SO(2n)}$, and we use the same notation for these sections. Define $\mc{L}_G := \mc{A}_G \otimes \omega^n$. \end{defn}

\begin{rmk} \label{rmk:L12}
The notation $\mc{L}_G$ is meant to evoke the Looijenga line bundles in Definition~\ref{defn:Loo}, and we will shortly see comparison results in Proposition~\ref{thm:Euler2} below. In the case of $\U(1)$, $\mc{L}_{\U(1)}$ is in fact canonically a square root of the Looijenga line of level one, i.e. $\mc{L}_{\U(1)}^{\otimes 2} \simeq \mc{L}_1$. As such, one often denotes $\mc{L}_{\U(1)}$ by $\mc{L}_{1/2}$, i.e., as a Looijenga line of level one-half.
\end{rmk}

\begin{rmk} \label{rmk:partialdescent}
It is also possible to define a line bundle $\mc{A}_{\SO(2n)}$ over $\Bun_{\SO(2n)}(\EE)$ using the transformation law for $\sigma_{\SO(2n)}$, but there are some technical issues to address. Namely,~$\SO(2n)$ does not have torsion-free fundamental group, so the map~$\mb{H} \times \mf{t}_{\mb{C}}\to \Bun_{\SO(2n)}(\EE)$ has as image the connected component of the trivial bundle, see Example~\ref{ex:opensub}. Hence, transformation properties of the function $\sigma_{\SO(2n)}$ can only construct a line bundle on this image. The \emph{determinant line bundle} (e.g., ~\cite[\S7]{Ell1}) extends this partially-defined line bundle to the whole of $\Bun_{\SO(2n)}(\EE)$ and pulls back along the homomorphism $\Spin(2n)\to \SO(2n)$ to $\mc{L}_{\Spin(2n)}$ from Definition~\ref{defn:generalA}. However, for the applications below it suffices to work with $\mc{L}_{\Spin(2n)}$ on $\Bun_{\Spin(2n)}(\EE)$. \end{rmk}



\begin{prop}\label{thm:Euler1} 
The functions $\sigma_{\SO(2n)}$ and $\sigma_{\U(n)}$ given by the formulas~\eqref{Eu:SO} and~\eqref{Eu:U} respectively define twisted equivariant elliptic Euler classes as the global sections 
\beq
\Eu_{\Spin(2n)}&:=&\beta^{-n}\cdot \sigma_{\SO(2n)}\in \Gamma(\Bun_{\Spin(2n)}(\EE);\dEll_{\Spin(2n)}^{2n}(\pt)\otimes\mc{L}_{\Spin(2n)})\nonumber\\
\Eu_{\U(n)}&:=&\beta^{-n}\cdot \sigma_{\U(n)}\in \Gamma(\Bun_{\U(n)}(\EE);\dEll_{\U(n)}^{2n}(\pt)\otimes\mc{L}_{\U(n)})\nonumber
\eeq
of equivariant elliptic cocycles as twisted by the holomorphic lines $\mc{L}_G$. \end{prop}

\bp 
For $G=\U(n),\Spin(2n)$, we have the isomorphisms of sheaves on $\Bun_G(\EE)$, 
$$
\dEll_{G}^{2n}(\pt)\otimes\mc{L}_{G}\simeq \dEll_{G}^{2n}(\pt)\otimes(\mc{A}_G \otimes \omega^n) \simeq \mc{A}_G
$$
using the definition $\mc{L}_G := \mc{A}_G \otimes \omega^n$ and periodicity from Proposition~\ref{prop:Bottperiodicity}. The above composition sends $\beta^{-n}\cdot \sigma_{G}$ to $\sigma_G$, which is a global section of $\mc{A}_G$ by definition. We conclude that $\beta^{-n}\cdot \sigma_{G}$ is a global section of $\dEll_{G}^{2n}(\pt)\otimes\mc{L}_{G}$. \ep


Naturality of $\dEll_{G}^\bullet$ for homomorphisms of groups gives the following. 

\begin{cor}\label{cor:Eu} Let $G$ be a compact Lie group. For any homomorphism $\rho\colon G\to \U(n)$ or $\rho\colon G\to \Spin(2n)$, we obtain a $G$-equivariant elliptic Euler class by pullback,
$$
{\rm Eu}_\rho\in \Gamma(\Bun_G(\EE),\dEll_G^{2n}(\pt)\otimes\mathcal{L}_\rho)
$$
where $\mathcal{L}_\rho$ is the holomorphic line bundle on $\Bun_G(\EE)$ that pulls back from $\mc{L}_{\U(n)}$ or $\mc{L}_{\Spin(2n)}$ on $\Bun_{\U(n)}(\EE)$ or $\Bun_{\Spin(2n)}(\EE)$, respectively, along the map induced by~$\rho$. 
\end{cor}

Next we compare the line bundle $\mc{L}_G$ with the more classical Looijenga line. For this comparison, we restrict to the substack $\Bun_{\SU(n)}\subset \Bun_{\U(n)}(\EE)$. 
Then from Definition~\ref{defn:Loo} we recall that for the simply-connected\footnote{Although $\Spin(2)$ is not simply-connected, its Looijenga lines are still classified by a level $\ell \in \H^4(B\Spin(2); \mb{Z}) \simeq \mb{Z}$, and the results below apply to this case as well. Similarly, $\Spin(4) \simeq \SU(2) \times \SU(2)$ is not simple and its levels are labelled by pairs of integers, but there is a canonical copy of $\mb{Z} \hookrightarrow \H^4(B \Spin(4); \mb{Z})$ arising, for example, from pullback under the natural inclusions $\Spin(4) \hookrightarrow \Spin(2n)$ for $n \ge 3$. }, simple Lie groups $\SU(n)$ and $\Spin(2n)$, the possible Looijenga line bundles are labeled by an integer~$\ell\in \Z$, called the \emph{level}. 


\begin{prop}\label{thm:Euler2} 
The line bundle $\mc{L}_{\Spin(2n)}$ is isomorphic to the level~1 Looijenga line~$\mc{L}_1$ on $\Bun_{\Spin(2n)}(\EE)$. The pullback of $\mc{L}_{\U(n)}$ along the functor $\Bun_{\SU(n)}(\EE) \to \Bun_{\U(n)}(\EE)$ associated with the inclusion $\SU(n)\hookrightarrow \U(n)$ is isomorphic to the level~1 Looijenga line~$\mc{L}_1$ on $\Bun_{\SU(n)}(\EE)$. This identifies the sections determined by Proposition~\ref{thm:Euler1}
\beq
\Eu_{\SU(n)}&:=&\beta^{-n}\cdot \sigma_{\U(n)} \in \Gamma(\Bun_{\SU(n)}(\EE);\dEll_{\SU(n)}^{2n}(\pt)\otimes\mathcal{L}_1)\nonumber\\
\Eu_{\Spin(2n)}&:=&\beta^{-n}\cdot \sigma_{\SO(2n)}\in \Gamma(\Bun_{\Spin(2n)}(\EE);\dEll_{\Spin(2n)}^{2n}(\pt)\otimes\mathcal{L}_1)\nonumber
\eeq
with super characters of level~1 vacuum representations of $L\SU(n)$ and $L\Spin(2n)$, respectively. 
\end{prop}

\begin{proof}
The identifications between $\mc{L}_G$ and Looijenga line bundles come from the comparing the classical formulas for the transformation properties of $\sigma_{\U(n)}$ and $\sigma_{\SO(2n)}$ (e.g., see~\cite[Equations 146-149]{Ell1}) with Definition~\ref{defn:Loo} of the Looijenga line bundle. In computing these transformation properties in the $\SU(n)$ case, we use that the restriction along $\SU(n)<\U(n)$ on the cover~\eqref{eq:covering} corresponds to restriction of~\eqref{Eu:U} to the subspace of $\HH\times\mf{t}_\C$ with $c_1=\sum_j z_j=0$. In the $\Spin(2n)$ case, we use that $\sigma_{\SO(2n)}$ as a section of a line bundle over $\Bun_{\Spin(2n)}(\EE)$ considers the quotient of the cover~\eqref{eq:covering} by the sublattice ${\rm ker}(\Z^{n}\stackrel{+}{\to} \Z/2)\subset \Z^n\simeq X_*(T)$ of the cocharacter lattice of $\SO(2n)$. The identification of sections with characters of loop group representations then follows from Proposition~\ref{prop:Verlinde}, together with well-known formulas for the level~1 characters of vacuum representations. 
%
%
\end{proof}


\subsection{Equivariant elliptic Thom forms}


For $G=\Spin(2)\simeq \U(1)$, recall that the ordinary (non-elliptic) equivariant Mathai--Quillen Thom form on $V=\R^2\cong \C$ is given by 
$$
u_V=\frac{1}{\pi}e^{-|x|^2}(\beta^{-1}z+d\vol)\in \Omega_{G}^2(V)
$$
where $z\in \mathfrak{t}^\vee_\C\cong \C$ is the monomial generator for polynomial functions on $\mathfrak{t}_\C$ (the occurrence of $\beta$ is from our grading conventions, see~\S\ref{sec:defn}),  $d\vol\in \Omega^2(V)$ is the orientation 2-form, and $e^{-|x|^2}$ is the Gaussian on $V$ relative to the standard norm~$|-|$. Our convention here and throughout is that for a vector space $V$, $\Omega_{G}^\bullet(V)$ denotes differential forms that are rapidly decreasing (which computes compactly supported cohomology, e.g., see~\cite[\S4]{MathaiQuillen}). Also recall that the ordinary equivariant Euler class is given by $\beta^{-1}z\in \Omega^2_{G}(\pt)$. 

In the respective cases of $\U(n)$ and $\Spin(2n)$ below, let $V=\C^n$ be the standard representation of $\U(n)$, and $V=\R^{2n}$ denote the representation of $\Spin(2n)$ as factoring through the standard representation of $\SO(2n)$. 
For $h=(h_1,h_2)$ a pair of commuting elements, let $V^h_\perp\subset V$ be the orthogonal complement of the fixed point subspace $V^h\subset V$. Without loss of generality, we may assume that $h$ is in the maximal torus of $\U(n)$ or $\Spin(2n)$. As before, we use the standard coordinates on the Lie algebra $\mf{t}\simeq \R^n$ of the maximal torus~$T$. Next choose logarithms $\tilde{h}_1,\tilde{h}_2\in \mathfrak{t}$ so that $h_i=\exp(\tilde{h}_i)$ and a permutation of the standard coordinates of $\mathfrak{t}\simeq \R^{n}$ so that
\beq
&&\tilde{h}_1={\rm diag}(0,\dots,0,\tilde{h}_1^{k+1},\dots, \tilde{h}_1^n,),\qquad \tilde{h}_2={\rm diag}(0,\dots, 0,\tilde{h}_2^{k+1},\dots, \tilde{h}_2^n), \quad \tilde{h}_i^j\in \R.\label{eq:coordchoice}
\eeq
Above, the zero entries correspond to the subspace $V^h\subset V$ on which $h=(h_1,h_2)$ acts trivially, and the remaining entries correspond to $V^h_\perp$ the orthogonal complement. As before, let $W^h$ denote the Weyl group of the centralizer $G^h=C(h)=C(h_1,h_2)$, which is necessarily connected in this case for the Lie groups $\U(n)$ and $\Spin(2n)$. 

\begin{defn}
The \emph{$\U(n)$-equivariant elliptic Thom form} at~$\tilde{h}=(\tilde{h}_1,\tilde{h}_2)$ is defined as
\beq
&&(\Th_{\U(n)})_{\tilde{h}}=\left(\prod_{j=1}^k u_j \frac{\upsilon(\tau,z_j)}{z_j}\right) \left(\prod_{j=k+1}^n \beta\upsilon(\tau,z_j+\tilde{h}_1^j-\tau \tilde{h}_2^j)\right)  \in \Omega^{2n}_T(V^h;\mathcal{O}(\HH))^{W^h}\label{eq:Thom1}
\eeq
where $u_j=e^{-|x|^2}(\beta^{-1}z_j+d\vol_j)$ is the Mathai--Quillen Thom form associated with the Chern root~$z_j$. 
Similarly, the \emph{$\Spin(2n)$-equivariant elliptic Thom form} is defined as
\beq
&&(\Th_{\Spin(2n)})_{\tilde{h}}=\left(\prod_{j=1}^k u_j \frac{\sigma(\tau,z_j)}{z_j}\right) \left(\prod_{j=k+1}^n \beta\sigma(\tau,z_j+\tilde{h}_1^j-\tau \tilde{h}_2^j)\right) \in \Omega^{2n}_T(V^h;\mathcal{O}(\HH))^{W^h}.\label{eq:Thom2}
\eeq
The \emph{$\SU(n)$-equivariant elliptic Thom form} at~$\tilde{h}=(\tilde{h}_1,\tilde{h}_2)$ is defined as the restriction of $(\Th_{\U(n)})_{\tilde{h}}$ along $\SU(n)\subset \U(n)$. 
\end{defn}

\begin{rmk}
We recall our notational convention that the $T$-equivariant differential forms~\eqref{eq:Thom1} and~\eqref{eq:Thom2} have rapidly decreasing support. This is guaranteed by the rapidly decreasing (ordinary) $T$-equivariant Thom forms $\prod u_j$ in~\eqref{eq:Thom1} and~\eqref{eq:Thom2}. This standard Thom form is modified by the invertible functions $\sigma(\tau,z_j)/z_j$, whereas the second factor in these formulas is an ($\tilde{h}$-translated) elliptic Euler class of the orthogonal complement, $V^h_\perp$. These modifications to the standard Thom form do not change the rapidly decreasing property. Furthermore, we observe that the resulting equivariant Thom class is formally analogous to formulas for the image of the equivariant Thom class in K-theory under the (delocalized) Chern character, e.g., see~\cite[pg.~245]{BGV}.
\end{rmk} 



\begin{prop}\label{thm:Thom}
For $G=\U(n)$ or $\Spin(2n)$, the values $(\Th_T)_{\tilde h}$ assemble to give a global section of $\dEll_G^{2n}(V)\otimes \mc{L}_G$ that implements the universal Thom isomorphism in equivariant elliptic cohomology twisted by the holomorphic line $\mc{L}_G$ of Definition~\ref{defn:generalA},
$$
\dEll_G^\bullet(\pt)\stackrel{\sim}{\to} \dEll_G^{\bullet+2n}(V)\otimes\mc{L}_G,\qquad \alpha\mapsto \alpha\boxtimes{\rm Th}_G
$$
as a quasi-isomorphism of sheaves of chain complexes over $\Bun_G(\EE)$, where (following the prior notation) the target consists of cocycles that are rapidly decreasing on~$V$. 
\end{prop}

\bp
The Thom isomorphism statement is equivalent to showing that the elliptic Thom form is a nowhere vanishing section of the claimed line bundle. This statement can be checked locally, and the definitions~\eqref{eq:Thom1} and~\eqref{eq:Thom2} show that~$\Th_G$ is nonzero at every stalk. To see this, first note that nonzero in the stalk means that the power series~\eqref{eq:Thom1} and~\eqref{eq:Thom2} define functions that are nowhere vanishing in some neighborhood of $z_1=z_2=\dots =z_n=0$. The function~$\sigma(\tau,z_j)/z_j$ vanishes at lattice points with the exception of $z_j=0$, so is nonzero on a neighborhood. In the other factor, $\sigma(\tau,z_j+\tilde{h}_1^j-\tau \tilde{h}_2^j)$ vanishes at lattice points shifted by~$\tilde{h}_1^j-\tau \tilde{h}_2^j$. This shift corresponds to the action of $(h_1,h_2)$ on $V^h_\perp$, and so is necessarily a nonzero shift, implying that there exists a neighborhood of $z_j=0$ on which $\sigma(\tau,z_j+\tilde{h}_1^j-\tau \tilde{h}_2^j)$ is not zero. This verifies the Thom isomorphism statement for~\eqref{eq:Thom1}; the argument for~\eqref{eq:Thom2} is identical.


Showing that the stalk-level definition lifts to a global section is a bit more delicate. First we observe that the $W^h$-action on $(\Th_{\U(n)})_{\tilde{h}}$ permutes the factors in $\prod_{j=1}^k u_j \frac{\upsilon(\tau,z_j)}{z_j}$ and $\prod_{j=k+1}^n \beta\upsilon(\tau,z_j+\tilde{h}_1^j-\tau \tilde{h}_2^j)$ separately, leaving the overall function $(\Th_{\U(n)})_{\tilde{h}}$ invariant. The statement for $(\Th_{\Spin(2n)})_{\tilde{h}}$ is completely analogous. The action of the full Weyl group~$W$ also permutes these factors, but the specific action depends on the permutation of coordinates defining the~$\{\tilde{h}_1^j,\tilde{h}_2^j\}$ in~\eqref{eq:coordchoice}; however, by inspection the formulas for $\Th_G$ are invariant under these reorderings. 

It remains to check the analyticity condition from Definition~\ref{defn:ellcocycle}. For this we consider deformations of $(\tilde{h}_1, \tilde{h}_2)$ together with a translation in the Lie algebra dependence of the equivariant differential form as in~\eqref{eq:analyticity}. It suffices to check the following two types of deformations: (i) deformations in the first $k$ coordinates of~\eqref{eq:coordchoice}, and (ii) deformations of the last $n-k$ coordinates in~\eqref{eq:coordchoice}. In case (ii), compatibility is easy to check because (for small deformations) the fixed point sets are unchanged through the deformation. Compatibility then follows because we are just pulling back functions on the Lie algebra by a translation. In case (i), first we recall that the ordinary equivariant Thom class~$u_j$ restricts at the origin to the ordinary equivariant Euler class~$z_j$. In our case, this means for ${\rm res}$ as in~\eqref{eq:analyticity},
$$
{\rm res}\left(u_j \frac{\sigma(\tau,z_j)}{z_j}\right)=\sigma(\tau,z_j).
$$
Compatibility for case (i) then amounts to showing this restriction pulls back correctly along a translation in the Lie algebra, which it manifestly does. Finally, because $u_j$ transforms the same way as $z_j$ under the action of the cocharacter lattice and $\SL_2(\Z)$, we have that $u_j \frac{\sigma(\tau,z_j)}{z_j}$ transforms the same as $\sigma(\tau,z_j)$. Hence, the transformation properties of the elliptic Thom class are the same as those for the elliptic Euler class from Propositions~\ref{thm:Euler1} and~\ref{thm:Euler2}. Therefore, the stalks~\eqref{eq:Thom1} and~\eqref{eq:Thom2} glue together to give a section of $\mathcal{L}_G$, and we have produced the elliptic Thom form as a global section of the sheaf~$\dEll_G^{2n}(V)\otimes\mc{L}_G$.
\ep


Analogously to Corollary~\ref{cor:Eu}, naturality gives the following. 

\begin{cor} Let $G$ be a compact Lie group. For any homomorphism $\rho\colon G\to \U(n)$ or $\rho\colon G\to \Spin(2n)$, we obtain a $G$-equivariant elliptic Thom class by pullback,
$$
{\rm Th}_\rho\in \Gamma(\Bun_G(\EE),\dEll_G^{2n}(V)\otimes\mathcal{L}_\rho)
$$
where $\mathcal{L}_\rho$ is the holomorphic line of Corollary~\ref{cor:Eu}. 
\end{cor}

\subsection{The elliptic Chern--Weil map}
Let $A$ be a graded commutative $\C$-algebra and $V\to M$ a real or complex vector bundle classified by a map $f\colon M\to BG$ for $G=\U(n)$ or ${\rm O}(n)$, respectively. The Chern--Weil map in ordinary cohomology is 
\beq
\Poly(\mathfrak{g};A)^G\simeq \H_G(\pt;A)\simeq \H({\rm B}G;A)\stackrel{f^*}{\to} \H(M;A).\label{eq:CWmap}
\eeq
To any invariant polynomial on the Lie algebra, this map associates a characteristic class of $V$ in the cohomology of~$M$. When $V$ is equipped with a connection,~\eqref{eq:CWmap} refines to a map of chain complexes, $\Poly(\mathfrak{g};A)^G\to (\Omega^\bullet(M;A),d)$, where the source has trivial differential. Using Proposition~\ref{prop:Verlinde}, we will construct elliptic versions of the Chern--Weil maps
\beq
\begin{array}{ccc} \Rep^\ell(L\SU(n))\otimes_{\MF^0} \MF&\dashrightarrow& (\Omega^\bullet(M;\MF),d), \\ \Rep^\ell(L\Spin(2n))\otimes_{\MF^0} \MF&\dashrightarrow& (\Omega^\bullet(M;\MF),d)\end{array}\label{eq:ordEu}
\eeq
that send (characters of) level $\ell$~representations of loop groups to cocycle representatives of characteristic classes for complex vector bundles with $\U\langle 6\rangle$-structure or real vector bundles with ${\rm O}\langle 8\rangle$-structure.

The first step in constructing~\eqref{eq:ordEu} is an $\mc{L}_\ell$-twisted version of the completion map from Theorem~\ref{thm:complete} on the $\SL_2(\Z)$-cover $\HH$ of $\Mell$. Consider the composition
$$
e\colon \HH\hookrightarrow \HH\times \mf{t}\times \mf{t}\to \HH\times (T\times T)\cq W\simeq \HH\times \mathcal{C}^2[G]\to \Bun_G(\EE)
$$
where the first arrow includes at $0\in \mf{t}\times\mf{t}\simeq\mf{t}_\C$, the second map is induced by the exponential map from the Lie algebra to the Lie group, and the remaining maps are from Example~\ref{ex:holomorphic}. By Definition~\ref{defn:Loo}, sections of the Looijenga line bundle are functions on $\HH\times\mf{t}\times\mf{t}\simeq \HH\times\mf{t}_\C$ with properties. Said differently, the pullback of~$\mc{L}_\ell$ along $\HH\times \mf{t}_\C\to \Bun_G(\EE)$ has a preferred trivialization that identifies sections of the pullback with holomorphic functions. In particular, this gives a canonical trivialization of $\mc{L}_\ell$ in a neighborhood of $e\colon \HH\to \Bun_G(\EE)$ corresponding to a neighborhood of $\HH\times \{0\}$ in the cover $\HH\times\mf{t}\times\mf{t}$. This permits the following. 

\begin{constr}\label{constr:Lootriv}
Let $\mc{L}_\ell$ denote the level $\ell$ Looijenga line for $G=\SU(n)$ or $\Spin(2n)$; see Definition~\ref{defn:Loo}. The restriction of the sheaf $\dEll_G^\bullet(\pt)\otimes \mathcal{L}_\ell$ along the map $e\colon \HH\to \Bun_G(\EE)$ together with the trivialization of $\mc{L}_\ell$ specified above gives an isomorphism of sheaves of commutative differential graded algebras on $\HH$ that on global sections is
\beq
\Gamma(\HH;e^*\dEll_G^\bullet(\pt)\otimes\mc{L}_\ell)\stackrel{\sim}{\to} \Omega_G^\bullet(\pt;\O(\HH)[\beta,\beta^{-1}]),\label{eq:twistcomplete}
\eeq
where the target is the Cartan model for Borel equivariant cohomology of the point with coefficients in $\O(\HH)[\beta,\beta^{-1}]$. 
\end{constr}

\begin{defn}
Define the \emph{level~$\ell$ elliptic Chern--Weil map} as the composition
\beq
\Rep^\ell(LG)\otimes_{\MF^0} \MF&\simeq& \Gamma(\Bun_G(\EE);\dEll_G^\bullet(\pt)\otimes\mc{L}_\ell)\nonumber\\
&&\to  \Omega_G^\bullet(\pt;\O(\HH)[\beta,\beta^{-1}])\to \Omega^\bullet(M;\mathcal{O}(\HH)[\beta,\beta^{-1}])\label{eq:EllCW}
\eeq
 where the isomorphism is from Proposition~\ref{prop:Verlinde}, the middle map is restriction to $\HH$ along~$e$ followed by~\eqref{eq:twistcomplete} and the final map is the usual Chern--Weil map determined by a vector bundle with $G$-structure and $G$-invariant connection for $G=\SU(n)$ or $\Spin(2n)$. 
\end{defn}

The nontriviality of the line bundle $\mc{L}_\ell$ manifests in the image of the elliptic Chern--Weil map as a possible failure of invariance under the action of~$\SL_2(\Z)$ on~$\HH$. We analyze this question of descent from $\HH$ to $[\HH/\SL_2(\Z)]\simeq \Mell$ for the equivariant elliptic Euler class, which we recall corresponds to the vacuum representation of the appropriate loop group at level~1. We recall from Example~\ref{ex:trivialgroup} that 
$$
(\Omega^\bullet(M;\mathcal{O}(\HH)[\beta,\beta^{-1}])^{\SL_2(\Z)},d)\simeq (\Omega^\bullet(M;\MF),d)
$$
is a cochain model for $\TMF(M)\otimes \C\simeq \H(M;\MF)$, i.e., cohomology with coefficients in modular forms. 

\begin{thm}\label{thm:compat}
For~$G=\SU(n)$ and $V\to M$ a complex vector bundle with $c_1(V)=c_2(V)=0$, the image of $\Eu_G$ under~\eqref{eq:EllCW} (for $\ell=1$) is a cocycle representative for the elliptic Euler class in $\TMF(M)\otimes \C$
coming from the $\MU\langle 6\rangle$ orientation of $\TMF\otimes \C$. 

Similarly, for~$G=\Spin(2n)$ and $V\to M$ a real vector bundle with spin structure and $\frac{p_1}{2}(V)=0$, the image of $\Eu_G$ along~\eqref{eq:EllCW} (for $\ell=1$) is a cocycle representative for the elliptic Euler class 
coming from the ${\rm MString}$ orientation of $\TMF\otimes \C$.
\end{thm}

\bp
The image of the section $\Eu_G$ under~\eqref{eq:twistcomplete} is a cocycle representative of the Borel equivariant characteristic class defined in terms of Chern roots $z_j$ via the formulas~\eqref{Eu:U} and~\eqref{Eu:SO}, using a choice of $G$-invariant connection on~$V\to M$. The image under the elliptic Chern--Weil map sends the Lie algebra dependence to traces of powers of curvature of the connnection. The obstruction to the underlying class in $\H(M;\mathcal{O}(\HH)[\beta,\beta^{-1}])$ being $\SL_2(\Z)$-invariant is the coefficient of the 2nd Eisenstein series for the description of the Witten class as in~\eqref{eq:Witten}. At the level of the Euler cocycle, this coefficient is precisely the Chern--Weil representative for $c_2(V)$ or $p_1(V)$ in the complex and real cases, respectively. 
%
\ep

\begin{rmk} An analogous result to the above holds for the images of cocycle representatives of elliptic Thom classes under the elliptic Chern--Weil map: the equivariant refinement does indeed recover the standard non-equivariant Thom class.  \end{rmk}

\subsection{Some examples}\label{ex:S22}

To give a flavor for how to compute with Euler and Thom classes, we spell out a couple examples. 

\begin{ex} \label{ex:derived} This is a continuation of Example~\ref{ex:S2} for $S^1=\U(1)$ acting on $S^2$ by rotation about an axis. First we identify compactly supported cohomology for $\U(1)$ acting on $V=\C$ with relative cohomology of the 2-sphere
$$
\Ell_{\U(1)}^\bullet(V) \simeq \Ell^\bullet_{\U(1)}(\CP^1, \infty) 
$$
where $\infty\in \CP^1\simeq S^2$ is the standard base point at infinity. Applying the Thom isomorphism for $\U(1)$-equivariant elliptic cohomology from Proposition~\ref{thm:Thom}, we observe 
\begin{align} 
\Ell_{\U(1)}^2(S^2, \infty) \simeq \Ell_{\U(1)}^\bullet(V)\otimes(\mc{L}_{\U(1)} \otimes \mc{L}_{\U(1)}^\vee) \simeq \Ell_{\U(1)}^{0}(\pt) \otimes \mc{L}_{\U(1)}^\vee\simeq \mc{L}_{\U(1)}^\vee. \nonumber \label{eq:untwistedThomex} 
\end{align}
 One may then apply the long exact sequence for the pair $(S^2,\infty)$ in sheaves of chain complexes on $\Bun_{\U(1)}(\EE)\simeq \EE^\vee$. We obtain the following:
 \begin{eqnarray*} \Ell^0_{\U(1)}(S^2) &\simeq& \Ell^0_{\U(1)}(\pt) \oplus \Ell^0_{\U(1)}(S^2, \infty) \\ 
 &\simeq& \mc{O} \oplus \Big(\Ell^{2}_{\U(1)}(S^2, \infty) \otimes \omega\Big) \\ 
 &\simeq& \mc{O} \oplus \Big(\mc{L}_{\U(1)}^\vee\otimes \omega\Big) \end{eqnarray*} 
 as a consequence of the $\mc{L}_{\U(1)}$-twisted Thom isomorphism and $\omega$-twisted Bott periodicity.
As a sanity check, we observe that $\mc{L}_{\U(1)}^\vee\otimes \omega$ is indeed a trivial bundle away from the zero section in $\Euni^\vee$
 and this conforms with the computations in Example~\ref{ex:S2}.

More generally, we have the isomorphism of sheaves 
\beq 
\label{eq:twistedThomex} \Ell_{\U(1)}^0(S^2) \otimes \mc{L}_{m/2} \simeq \mc{L}_{m/2} \oplus \Big( \mc{L}_{(m-1)/2} \otimes \omega \Big),\qquad m\in \Z,\nonumber 
\eeq 
where $\mathcal{L}_{m/2}:=(\mathcal{L}_{1/2})^{\otimes m}$, using the notation $\mathcal{L}_{1/2}=\mathcal{L}_{\U(1)}$ justified in Remark~\ref{rmk:L12}. 
We observe this sheaf has global sections if and only if $m$ is nonnegative. However, there are nontrivial \emph{derived} global sections for any~$m \in \mb{Z}$, e.g., by Serre duality on $\Bun_{\U(1)}(\EE)\simeq \EE^\vee$.


In the literature, authors often identify the (quasi-coherent) sheaf $\Ell_G^0(M)$ with a scheme by taking the relative Spec over $\Bun_G(\EE)$, especially when the cohomology is concentrated in even degrees. We now explain this perspective for $\Ell_{\U(1)}^0(S^2)$, i.e., for $\Spec_{\EE^\vee} \Big( \mc{O} \oplus ( \mc{L}^{-1/2} \otimes \omega) \Big) \simeq \Spec_{\EE^\vee} \Big( \mc{O} \oplus \mc{O}(-0) \Big)$, where we freely use the isomorphism of $\mc{L}^{1/2} \otimes \omega^{-1} \simeq \mc{O}(0)$, via the function $\upsilon(\tau,z)\in \mathcal{O}(\HH\times \C)$ determining a section of $\mc{L}^{1/2} \otimes \omega^{-1}$ that vanishes to first order at the zero section $0\colon \Mell\to \EE^\vee$. We determine the algebra structure on $\mc{O} \oplus \mc{O}(-0)$ by the Mayer--Vietoris sequence for the standard  $\U(1)$-equivariant cover of $S^2$ by the upper and lower hemispheres. One finds 
\begin{eqnarray*} 
\Ell^0_{\U(1)}(S^2) &\simeq& \ker\Big( \Ell^0_{\U(1)}(\pt) \oplus \Ell^0_{\U(1)}(\pt) \to \Ell^0_{\U(1)}(\U(1)) \Big) \\ &\simeq& \ker(\mc{O} \oplus \mc{O} \to \mc{O}_0), 
\end{eqnarray*} 
where the above computations are in the category of sheaves on $\EE^\vee$ and $\mc{O}_0$ is the structure sheaf of the zero section $0\colon \Mell\to \Bun_{\U(1)}(\EE)$. Indeed, the above description makes it clear that as a coherent sheaf, the above kernel is isomorphic to $\mc{O} \oplus \mc{O}(-0)$, but the Mayer-Vietoris description has the additional property of making manifest the algebra structure. Pullbacks of sheaves of algebras become pushouts of schemes under (relative) Spec, so we have that $\Spec_{\EE^\vee}(\Ell_{\U(1)}^0(S^2))$ is simply two copies of the (universal) elliptic curve $\EE^\vee$ glued along their zero sections. \end{ex}

\begin{ex} More generally, consider $\U(1)$ acting on $S^2$ by $n$ times the rotation action; to emphasize the dependence of the equivariant structure on $n$, we denote this representation sphere as $S^2[n]$.

We repeat the computation of $\Ell_{\U(1)}^0(S^2[n])$ from the previous example using the Thom isomorphism, only now we use the ``charge $n$'' representation of $\U(1)$ on $V=\R^2$ with character~$y^n\in C^\infty(\U(1))$. By naturality, the Thom class of this representation is the pullback of the universal Thom class from $[\pt\sq \U(1)]$ along the multiplication by $n$ map $\U(1)\to \U(1)$. Hence, the induced twisting bundle on $\EE^\vee$ is given by the pullback of the bundle $\mc{L}^{1/2}$ under the map $\EE^\vee \overset{n}{\to} \EE^\vee$ and if we apply the Thom isomorphism as before, we find 
$$
\Ell_{\U(1)}^0(S^2[n]) \simeq \mc{O} \oplus \Big( n^* \mc{L}^{-1/2} \otimes \omega \Big).
$$ 

Next we repeat the Mayer--Vietoris computation from the previous example, using the same cover to find 
\begin{eqnarray*} \Ell_{\U(1)}^0(S^2[n]) &\simeq& \ker\Big( \Ell_{\U(1)}^0(\pt) \oplus \Ell_{\U(1)}^0(\pt) \to \Ell_{\U(1)}^0(S^1[n]) \Big) \\ &\simeq& \ker \Big(\mc{O} \oplus \mc{O} \to \mc{O}_{n{\text -}{\rm torsion}}\Big), \end{eqnarray*} 
where we use similar notation for $S^1[n]$ as an $S^1$ with its $\U(1)$-equivariant structure given as $n$ times the usual. The above description makes it clear that $\Spec_{\EE^\vee}(\Ell_{\U(1)}^0(S^2[n]))$ is now two copies of $\EE^\vee$ glued along their $n$-torsion subschemes, while as a sheaf, one may rewrite $\Ell_{\U(1)}^0(S^2[n])$ as $\mc{O} \oplus \mc{O}(-\{n{\text -}{\rm torsion}\})$. Indeed, as $\mc{L}^{1/2} \otimes \omega^{-1} \simeq \mc{O}(0)$, we have $n^* \mc{L}^{1/2} \otimes \omega^{-1} \simeq n^* \Big( \mc{L}^{1/2} \otimes \omega^{-1} \Big) \simeq n^* \mc{O}(0) \simeq \mc{O}(\{n{\text -}{\rm torsion}\})$ and so our two descriptions indeed agree: the $n$-torsion of an elliptic curve over~$\C$ is a subgroup of order~$n^2$. 

We recall that on a fixed elliptic curve, $\mc{O}(\{n{\text -}{\rm torsion}\}) \simeq \mc{O}(0)^{n^2}$. This follows from Abel's theorem, or equivalently the fact that addition in the Picard group of an elliptic curve corresponds to addition in the elliptic curve itself. This affords an explicit description of $n^* \mc{L}^{1/2} \otimes \omega^{-1}$ as follows.
Using the relative Picard functor, the isomorphism $\mc{O}(\{n{\text -}{\rm torsion}\}) \simeq \mc{O}(0)^{n^2}$ still exists in moduli up to twists coming from the base. Hence over the universal dual curve~$\mc{E}^\vee$, we have $\mc{O}(\{n{\text -}{\rm torsion}\}) \simeq \mc{O}(0)^{n^2} \otimes \omega^{\ell}$ for some $\ell$. But pulling back along the zero section $0\colon \Mell\to \EE^\vee$ and using, for example, that  
$$\mc{O}(0)|_{\Mell} \simeq \omega^{-1},$$ 
we find $\ell = n^2 - 1$. Hence, $n^* \mc{L}^{1/2} \otimes \omega^{-1} \simeq \mc{L}^{n^2/2} \otimes \omega^{-1}.$\end{ex}

\section{Equivariant orientations and the theorem of the cube}\label{sec:FGL}
This section studies a more algebro-geometric point of view on the string orientation following the constructions in~\cite{HopkinsICM94,AHSI}. This leads to a canonical string orientation of elliptic cohomology relying on the theorem of the cube. We show that this refines equivariantly, yielding a \emph{unique} string orientation of equivariant elliptic cohomology. 

\subsection{Background: Elliptic cohomology and the theorem of the cube}
Let $h$ be a multiplicative cohomology theory and $\tilde{h}$ the associated reduced cohomology theory. The isomorphism~$\tilde{h}^2(S^2)\simeq h^0(\pt)$ identifies a canonical generator of $\tilde{h}^2(S^2)$ as an $h^0(\pt)$-module. 

\begin{defn} A \emph{complex orientation} (or \emph{$\MU$-orientation}) of a cohomology theory $h$ is an element $\tilde{\che}\in \tilde{h}^2(\CP^\infty)$ whose restriction to $\CP^1=S^2$ is the canonical generator of~$\tilde{h}^2(S^2)$. 
\end{defn}

A complex orientation defines a Chern (equivalently, Euler) class for line bundles valued in $h$-cohomology, where $\tilde{\che}=\tilde{\che}(\mathcal{O}(1))$ is defined to be the Chern class of the tautological line on $\CP^\infty$. From this class one can build $h$-valued Chern classes for all (virtual) vector bundles using the splitting principle and the Whitney sum formula. 

Now suppose that $h$ is \emph{even} ($h^\bullet=\{0\}$ for $\bullet$ odd) and \emph{2-periodic} (there exists an invertible element $\beta\in h^{-2}(\pt)$). Then the Atiyah--Hirzebruch spectral sequence can be used to show that a complex orientation for~$h$ exists. The class~$\beta$ allows one to put a choice of complex orientation~$\tilde{\che}$ in degree zero, $\che=\tilde{\che}\beta \in h^0(\CP^\infty)$. Pulling $\che$ back along the three maps
\beq
&&\che\in h^0(\CP^\infty)\stackrel{p_1^*,p_2^*,m^*}{\longrightarrow} h^0(\CP^\infty\times \CP^\infty),\quad p_1,p_2,m\colon \CP^\infty\times \CP^\infty\to \CP^\infty\label{eq:FGL}
\eeq
gives a formula, $m^*\che=F(p_1^*\che,p_2^*\che)$ where~$F$ is a formal power series in two variables satisfying properties codifying (homotopy) associativity and unitality of the multiplication map~$m$. Quillen observed that these properties make $F$ into a \emph{formal group law} over~$h^0(\pt)$~\cite{Quillen}. 

\begin{ex}\label{ex:Kthycplx}
Complex K-theory is even, 2-periodic, and complex oriented with $\che=1-[\mathcal{O}(1)]\in \K^0(\CP^\infty)$ where $\mc{O}(1)$ is the tautological line bundle on $\CP^\infty$. Note that as a $\Z/2$-graded bundle, $1-[\mathcal{O}(1)]=[\underline{\C}\ominus \mathcal{O}(1)]=[\Lambda^\bullet\mathcal{O}(1)]$ is the total exterior power~$\mathcal{O}(1)$. The associated formal group law is the multiplicative formal group law, i.e., for line bundles $L$ and $L'$ we have
\beq
\che(L\otimes L')=\che(L)+\che(L')-\che(L)\cdot \che(L'). \label{eq:multFGL}
\eeq
\end{ex}

Recall that a formal group law is equivalent to the data of a formal group with a choice of \emph{coordinate}, i.e., a function on the formal group that vanishes to first order at the identity. Hence, for an even, 2-periodic, complex oriented cohomology theory $h$, forgetting the choice of~$\che$ leaves the formal spectrum ${\rm Spf}(h^0(\CP^\infty))$ with the structure of a formal group. 

\begin{defn} 
An \emph{elliptic cohomology theory is} (i) an elliptic curve $E$ defined over a commutative ring~$R$, (ii) an even, 2-periodic cohomology theory~$h$ with $h^0(\pt)\simeq R$, and (iii) an isomorphism of formal groups ${\rm Spf}(h^0(\CP^\infty))\simeq \widehat{E}$ where $\widehat{E}$ is the formal completion of~$E$ at its identity section. 
\end{defn}

Choosing an $\MU$-orientation of elliptic cohomology is an under-constrained problem: there are typically many choices of coordinate on an elliptic formal group. As described by Hopkins~\cite{HopkinsICM94}, if we instead ask for an \emph{a priori} weaker structure, namely an $\MU\langle 6\rangle$- or $\MO\langle8\rangle$-orientation, there is a more canonical choice. 
Just as one can define Chern classes for all complex vector bundles from the data of the top Chern class of the universal line bundle, there is a similar type of splitting principle for characteristic classes of $\U\langle 6\rangle$-bundles. Namely, all $\U\langle 6\rangle$-bundles formally split into direct sums of trivial bundles and virtual bundles pulled back from
\beq
V_3=(L_1-1)\otimes(L_2-1)\otimes (L_3-1)\label{eq:V3}
\eeq
over $B\U(1)^{\times 3}\simeq (\CP^\infty)^{\times 3}$. 
Hence, a theory of $\MU\langle 6\rangle$ characteristic classes is determined by a universal characteristic class~\cite[\S4-6]{HopkinsICM94}
\beq
s\in h(\CP^\infty\times \CP^\infty\times \CP^\infty).\label{eq:split}
\eeq
This class is required to satisfy the additional consistency conditions:
\begin{enumerate}
\item[(rigid)] $e^*s=1\in h^0(\pt)$ where $e$ is inclusion of the basepoint $e\colon \pt\hookrightarrow \CP^\infty\times \CP^\infty\times \CP^\infty$;
\item[(symmetric)] $s$ pulls back to itself along the maps $\CP^\infty\times \CP^\infty\times \CP^\infty\to \CP^\infty\times \CP^\infty\times \CP^\infty$ that permute the factors; and
\item[(cocycle)] $(m_{12}^*s) (p_{134}^*s)=(m_{23}^*s) (p_{234}^*s)$ where $m_{i(i+1)}\colon (\CP^\infty)^{\times 4}\to (\CP^\infty)^{\times 3}$ is multiplication on the $i$ and $(i+1)^{\rm st}$ factors, and $p_{ijk}\colon (\CP^\infty)^{\times 4}\to (\CP^\infty)^{\times 3}$ is the projection to the $i,j$ and $k$th factors. 
\end{enumerate}
When the cohomology theory $h$ is part of the data of an elliptic cohomology theory, Ando--Hopkins--Strickland~\cite{AHSI} show that a class~\eqref{eq:split} satisfying these consistency conditions may be produced from a cubical structure on the line bundle~$\mc{O}(-0)$ on the elliptic curve~$E$, as we review presently. Recall that sections of $\mc{O}(-0)$ are functions that vanish to first order at $0\in E$. A \emph{cubical structure} for a line bundle $\mc{L}$ on $E$ is the data of a section~$s$ of a line bundle $\Theta^3(\mc{L})$ on $E\times E\times E$ whose fiber at $(x,y,z)\in E\times E\times E$ is 
$$
\Theta^3(\mc{L})_{(x,y,z)}=\mc{L}_{x+y+z}\otimes \mc{L}_x\otimes \mc{L}_y\otimes\mc{L}_z\otimes \mc{L}_{x+y}^\vee\otimes \mc{L}_{x+z}^\vee\otimes \mc{L}_{y+z}^\vee\otimes \mc{L}_e^\vee.
$$
This section is required to satisfy analogous properties to the three above:
\begin{enumerate}
\item[(rigid)] $s(e,e,e)=1$;
\item[(symmetric)] $s(z_{\sigma(1)},z_{\sigma(2)},z_{\sigma(3)})=s(z_1,z_2,z_3)$ for any permutation $\sigma$;
\item[(cocycle)] $s(w+x,y,z)s(w,x,z)=s(w,x+y,z)s(x,y,z)$
\end{enumerate}
where there are implicit (canonical) isomorphisms between line bundles used in the equalities. The theorem of the cube (or Abel's theorem) shows that there is a \emph{unique} cubical structure on $\mc{O}(-0)$. When passing from the elliptic group to the formal group, this determines a canonical $\MU\langle 6\rangle$-orientation of an elliptic cohomology theory. Further work of Hopkins shows that if the line bundle $\mc{L}$ has the additional structure of an isomorphism $\mc{L}_x\simeq \mc{L}_{-x}$ and the section $s$ satisfies $s(x,y,-x-y)=1$, then the $\MU\langle 6\rangle$-orientation extends to an $\MO\langle 8\rangle$-orientation. Under certain conditions on the elliptic cohomology theory~\cite[Theorem~6.2]{HopkinsICM94}, this additional condition is guaranteed, giving a canonical $\MO\langle8\rangle=\MString$-orientation of such elliptic cohomology theories.

\subsection{Orientations in complex analytic elliptic cohomology}
We give a quick overview of (non-equivariant) orientations in complex analytic elliptic cohomology. These facts are surely known to the experts; most of the following can be deduced from the introduction of~\cite{AHSI}. Consider the elliptic curve $E_\tau=\C/\Z\oplus \tau \Z$. Viewing $\C$ as a complex analytic group under addition, the quotient map~$\C\to E_\tau$ is a homomorphism with discrete kernel and so determines an isomorphism of formal groups over~$\C$
\beq
\widehat{\mathbb{G}}_a\simeq \widehat{E}_\tau\label{eq:elllog}
\eeq 
where $\widehat{\mathbb{G}}_a$ is the additive formal group. Consider~$\H(-;\C[\beta,\beta^{-1}])$, ordinary cohomology with values in the graded ring~$\C[\beta,\beta^{-1}]$ where $|\beta|=-2$. The formal group associated with ordinary cohomology is the additive formal group, so the isomorphism~\eqref{eq:elllog} gives an elliptic cohomology theory $\Ell_\tau$ whose underlying cohomology theory is $\H(-;\C[\beta,\beta^{-1}])$. The standard coordinate~$z$ on $\C$ determines a coordinate on $\widehat{E}_\tau$ giving a complex orientation of $\Ell_\tau$ associated with the additive formal group law,
\beq
\che(L\otimes L')=\che(L)+\che(L').\label{eq:additive}
\eeq
This choice of coordinate gives the identification 
$$
\widetilde{\Ell}_\tau(\CP^\infty):=\widetilde{\H}(\CP^\infty;\C[\beta,\beta^{-1}])\simeq \C[\![\che]\!][\beta,\beta^{-1}],\qquad |\beta|=-2, \ |\che|=0
$$
where~$\tilde{\che}=\beta^{-1}\che$ is the standard degree 2 generator of the cohomology of $\CP^\infty$.

More generally, recall the $\HH$-family of complex analytic elliptic curves $\EE$ from~\eqref{eq:Euniv}. The quotient map
$$
\HH\times \C\to \widetilde{\EE}
$$
gives an $\HH$-family of isomorphisms~\eqref{eq:elllog} of formal groups over $\mathcal{O}(\HH)$. This gives a \emph{complex analytic} elliptic cohomology theory defined by the elliptic curve~$\widetilde{\EE}$,  the cohomology theory~$\H(-;\mc{O}(\HH)[\beta,\beta^{-1}])$ and the $\HH$-family of isomorphisms~\eqref{eq:elllog}. This complex analytic elliptic cohomology theory has an $\SL_2(\Z)$-action induced by the action on coefficients from fractional linear transformations on $\mc{O}(\HH)$ and $\beta\mapsto \beta/(c\tau+d)$. Considering this action applied to $\mathcal{O}(U)$ for open submanifolds $U\subset \HH$ defines a sheaf of cohomology theories denoted $\Ell$ on the stack $\Mell\simeq [\HH/\SL_2(\Z)]$. We observe that the global sections of $\Ell$ are cohomology with values in modular forms, i.e., $\TMF\otimes \C$. Furthermore, for any $\tau\in \HH$, the sheaf $\Ell$ restricts to the elliptic cohomology theory $\Ell_\tau$ from the previous paragraph via the evaluation map $\ev_\tau\colon \mc{O}(\HH)\to \C$. 

The standard coordinate from $z$ on $\C$ determines an $\SL_2(\Z)$-invariant complex orientation of~$\Ell$: the Chern class $\tilde{\che}=\beta^{-1}\che$ pulls back to itself under isomorphisms~$\widetilde{\EE}\to \widetilde{\EE}$ associated with elements of~$\SL_2(\Z)$ since $z\mapsto z/(c\tau+d)$ (therefore $\che\mapsto \che/(c\tau+d)$) and $\beta^{-1}\mapsto (c\tau+d)\beta^{-1}$. In particular, this determines a complex orientation of $\TMF\otimes \C$. 

Although the coordinate~$z$ is perhaps the most obvious one, there is a huge amount of freedom in choosing complex orientations of the~$\Ell_\tau$ and $\Ell$. Indeed, any holomorphic function on $\HH\times \C$ that vanishes to first order on $\HH\times \{0\}\hookrightarrow \HH\times \C$ defines a different orientation of~$\Ell$. Such a function can be expressed as a power series in~$z$ with coefficients in $\mc{O}(\HH)$ whose lowest order nonvanishing term is~$z$. In the language of formal group laws, this is the statement that all coordinates are related to the coordinate~$z$ via an isomorphism of formal group laws. We consider two such choices, namely the variants of the Weierstrass sigma function from~\S\ref{sec:Weier}
\beq
\sigma(\tau,z), \upsilon(\tau,z)\in\mc{O}(\HH\times \C).\label{eq:sigmaor}
\eeq
These coordinates lead to different tensor product formulas for Chern classes than~\eqref{eq:additive}. Furthermore, the orientations from~\eqref{eq:sigmaor} are not invariant under the $\SL_2(\Z)$-action on~$E$, and hence fail to descend to a consistent complex orientation of the sheaf~$\Ell$ or its global sections,~$\TMF\otimes \C$.


The construction of $\MU\langle 6\rangle$- and $\MO\langle 8\rangle$-orientations from a cubical structure can be made completely explicit for elliptic curves over~$\C$. In this case, the coordinate $z$ on $\C$ from~\eqref{eq:elllog} allows one to express the cubical structure $\widetilde{\EE}\times_\HH\widetilde{\EE}\times_\HH\widetilde{\EE}$ in terms of a function on the universal cover $\C\times \C\times \C\times \HH\to \widetilde{\EE}\times_\HH\widetilde{\EE}\times_\HH\widetilde{\EE}$. One can check explicitly that the (necessarily unique) cubical structure in these coordinates is given by
\beq
&&s=\frac{\sigma(\tau,x+y)\sigma(\tau,x+z)\sigma(\tau,y+z)\sigma(\tau,0)}{\sigma(\tau,x+y+z)\sigma(\tau,x)\sigma(\tau,y)\sigma(\tau,z)}=\frac{\upsilon(\tau,x+y)\upsilon(\tau,x+z)\upsilon(\tau,y+z)\upsilon(\tau,0)}{\upsilon(\tau,x+y+z)\upsilon(\tau,x)\upsilon(\tau,y)\upsilon(\tau,z)}\label{eq:cubical}
\eeq
which we interpret as a class in $s\in \H^0(\CP^\infty\times \CP^\infty\times \CP^\infty;\mc{O}(\HH)[\beta,\beta^{-1}])\simeq \mathcal{O}(\HH)[\![x,y,z]\!]$. We observe further that $s$ is $\SL_2(\Z)$-invariant; this follows from the standard transformation properties of the $\sigma$-function. Hence,~\eqref{eq:cubical} determines a compatible family of $\MU\langle 6\rangle$-orientations for the sheaf of cohomology theories~$\Ell$ as well as the global sections~$\TMF\otimes \C$. We observe that the pullback of $\mc{O}(-0)$ under the inversion map $\widetilde{\EE}\to \widetilde{\EE}$ is canonically isomorphic to $\mc{O}(-0)$, so that we can ask for the additional condition on $s$ to obtain an $\MO\langle8\rangle$-structure. By inspection (e.g., because~$\sigma$ is odd) the cubical structure $s$ satisfies this additional requirement and hence gives an $\MO\langle 8\rangle$-orientation. 

We further observe that the class $s\in \H^0(\CP^\infty\times \CP^\infty\times \CP^\infty;\mc{O}(\HH)[\beta,\beta^{-1}])$ is the top Chern class of~$V_3$ relative to the complex orientations given by~\eqref{eq:sigmaor}. Indeed, the value of the $\MU\langle 6\rangle$-orientation on any complex vector bundle can be computed using the splitting principle and the complex orientation associated with~\eqref{eq:sigmaor}. To summarize, although these complex orientations of $\Ell$ fail to descend to $\Mell$, they determine the canonical $\MU\langle 6\rangle$-orientation that does descend. This turns out to mirror the equivariant refinement of the string orientation.

\subsection{Equivariant refinements of orientations}
We start with a motivating example. 

\begin{ex}
This is a continuation of Example~\ref{ex:Kthycplx}. We can ask for an equivariant refinement of the complex orientation of $\K$-theory relative to the Atiyah--Segal completion map,
\beq
\begin{array}{ccl}
\Rep(\U(1))=\K_{\U(1)}(\pt)&\stackrel{{\rm completion}}{\longrightarrow}& \K(B\U(1))=\K(\CP^\infty)\\
\phantom{\Rep(\U(1))=}\stackrel{\downin}{\che_{\U(1)}} &\stackrel{?}{\mapsto} & \phantom{blah}\stackrel{\downin}{\che} 
\end{array}\label{eq:Kcomplete}
\eeq
i.e., a virtual representation that maps to the chosen complex orientation. There is indeed a unique such virtual representation, namely $\che_{\U(1)}=1-R$ where $R$ is the defining representation of~$\U(1)$. 
\end{ex}

Using the elliptic Atiyah--Segal completion map from~\S\ref{sec:complete}, we can ask for a similar equivariant refinement of a complex orientation of elliptic cohomology over~$\C$,
\beq
\begin{array}{ccl}
\Gamma(\mc{O}_{\EE^\vee})=\Gamma(\Ell_{\U(1)}(\pt))&\stackrel{{\rm completion}}{\longrightarrow}& \Ell(B\U(1))=\Ell(\CP^\infty).\\
\phantom{\Gamma(\mc{O}_E)=}\stackrel{\downin}{\che_{\U(1)}} &\stackrel{?}{\mapsto} & \phantom{blah}\stackrel{\downin}{\che} 
\end{array}\label{eq:ellcomplete1}
\eeq
However, one immediately finds that no such class can exist, even for elliptic cohomology for a single elliptic curve: the class~$\che$ defines a function on a formal neighborhood of $0\in E^\vee_\tau$ that vanishes to first order at zero, and since globally defined functions on $E^\vee_\tau$ are constant, any putative class $\che_{\U(1)}$ is the zero class. Stated in more algebro-geometric language, a lift~\eqref{eq:ellcomplete1} for a fixed curve $E_\tau$ is asking for a global section of~$\mathcal{O}(-0)$ on $E_\tau^\vee$,
\beq
\begin{array}{ccl}
\Gamma(E^\vee_\tau;\mc{O}(-0))&\stackrel{{\rm completion}}{\longrightarrow}& \Ell_\tau(B\U(1))=\Ell(\CP^\infty).\\
\stackrel{\downin}{\che_{\U(1)}} &\stackrel{?}{\mapsto} & \phantom{blah}\stackrel{\downin}{\che} 
\end{array}\nonumber
\eeq
and the only such global section $\che_{\U(1)}$ is the zero section. Under completion this is sent to the zero class in $\Ell(\CP^\infty)$, which does not define a complex orientation. We summarize this observation as follows:

\begin{prop} \label{eq:nogo} No $\MU$-orientation of a complex analytic elliptic cohomology theory $\Ell_\tau$ may be refined to an equivariant $\MU$-orientation of the corresponding complex analytic equivariant elliptic cohomology theory defined over $E^\vee_\tau$. \end{prop}

Although this is not particularly deep, it highlights a crucial point: Chern classes in elliptic cohomology---even for a single elliptic curve---do not admit equivariant refinements. 

The resolution to this is to introduce a \emph{twisting}. This twisting refers to a relaxing of the setup in~\eqref{eq:ellcomplete1}, 
\beq
\begin{array}{ccl}
\Gamma(E^\vee_\tau;\mc{O}(-0)\otimes\mc{L}\otimes \omega^{-1})&\stackrel{{\rm completion}}{\longrightarrow}& \Ell^2_\tau(B\U(1))=\Ell^2(\CP^\infty)\\
\stackrel{\downin}{\tilde{\che}_{\U(1)}} &\stackrel{?}{\mapsto} & \phantom{blah}\stackrel{\downin}{\tilde{\che}} 
\end{array}\label{eq:ellcomplete2}
\eeq
where $\mc{L}$ is a line bundle on $E_\tau^\vee$, and for convenience we have changed points of view, taking Chern classes $\tilde{\che}$ and $\tilde{\che}_{\U(1)}$ in degree~2. The twisted completion map~\eqref{eq:ellcomplete2} requires additional data, namely a trivialization of $\mc{L}$ near $0\in E_\tau^\vee$ to identify the section $\tilde{\che}_{\U(1)}$ with a class in~$\Ell^2_\tau(B\U(1))$. 

\begin{defn}\label{defn:twequiv}
Let $E$ denote the restriction of $\widetilde{\EE}$ to a holomorphic submanifold $U\subset \HH$. A \emph{twisted equivariant refinement} of a complex orientation of a complex analytic elliptic cohomology theory associated with $E$ is a line bundle $\mc{L}$ on $E^\vee$ together with a nowhere vanishing section $\tilde{\che}_{\U(1)}\in \Gamma(E^\vee;\mc{O}(-0)\otimes \mc{L}\otimes \omega^{-1})$ and a choice of trivialization of $\mc{L}$ near the zero section $0\colon U\to E^\vee$ that identifies the restriction of $\tilde{\che}_{\U(1)}$ with the non-equivariant Chern class $\tilde{\che}$. 
\end{defn}

\begin{rmk}
We recall that a nowhere vanishing section $\tilde{\che}_{\U(1)}\in \Gamma(E^\vee;\mc{O}(-0)\otimes \mc{L}\otimes \omega^{-1})$ is the data as a section $\tilde{\che}_{\U(1)}\in \Gamma(E^\vee;\mc{L}\otimes \omega^{-1})$ that is nowhere vanishing away from the zero section of~$E^\vee$ and vanishes to precisely first order at $0$. 
\end{rmk}

With respect to a fixed elliptic cohomology theory, the freedom to choose a complex orientation is absorbed by the many ways to trivialize a fixed line bundle---in the notation of the previous definition, the line bundle~$\mc{L}$ and the section $\tilde{\che}_{\U(1)}$ are essentially unique:

\begin{prop} \label{prop:twistedequiv}
Let~$E$ be a family of elliptic curves as in Definition~\ref{defn:twequiv}. 
\begin{enumerate}
\item Any complex orientation of the complex analytic elliptic cohomology theory associated to $E$ admits a twisted equivariant refinement. 
\item The data $(\mc{L},\tilde{\che}_{\U(1)})$ of the twisted equivariant refinement are unique up to unique isomorphism, with any $(\mc{L},\tilde{\che}_{\U(1)})$ having a unique isomorphism to line bundle determined by the function $\upsilon(\tau,z)$ defined in~\eqref{eq:upsilondef}. 
\end{enumerate}
\end{prop}

\begin{rmk} The line bundle on $\Bun_{\U(1)}(\EE)\simeq \EE^\vee$ determined by $\upsilon(\tau,z)$ has well-known descriptions, e.g., as the Quillen determinant line bundle for the family of twisted $\bar\partial$-operators parameterized by $\Pic(\mathcal{E})=\mathcal{E}^\vee$ (e.g.,~\cite[Lemma~7.2]{Ell1}) or as a square root of the level~1 Looijenga line bundle (e.g.,~\cite[Remark~7.9]{Ell1}). \end{rmk}

\begin{proof}[Proof of Proposition~\ref{prop:twistedequiv}]
We tackle the uniqueness (2) first. Given two line bundles $\mc{L}$ and $\mc{L}'$ with sections $\tilde{\che}_{\U(1)}$ and $\tilde{\che}_{\U(1)}'$ satisfying the requirements, $\tilde{\che}_{\U(1)}'/\tilde{\che}_{\U(1)}$ is a nowhere vanishing section of $\mc{L}'\otimes \mc{L}^\vee$ and so determines an isomorphism $\mc{L}'\simeq \mathcal{L}$ that sends $\tilde{\che}_{\U(1)}'$ to $\tilde{\che}_{\U(1)}$. Hence $(\mc{L},\tilde{\che}_{\U(1)})$ is unique up to unique isomorphism. 

To prove (1), we construct an equivariant refinement where $\mc{L}$ is the line bundle with section $\tilde{\che}_{\U(1)}$ determined by the function $\upsilon(\tau,z)$ defined in~\S\ref{sec:Weier}. Recall this line bundle is defined as the trivial line bundle on $\HH\times \C$ with descent data to $\EE^\vee$ constructed from the transformation properties of $\upsilon(\tau,z)$. This specifies a preferred trivialization near the zero section $0\colon \HH\to \widetilde{\EE}^\vee$: view the section $\tilde{\che}_{\U(1)}$ as the function $\upsilon(\tau,z)$ on $\HH\times\C$ and restrict to a neighborhood of~$\HH\times \{0\}\subset \HH\times \C$. Now given any $E\to U$ we can restrict along the associated inclusions $U\subset \HH$ and $E^\vee\subset \widetilde{\EE}^\vee$ to obtain a line bundle on $E^\vee$ with section and a trivialization in the neighborhood of the zero section $0\colon U\to E^\vee$. This recovers the complex orientation specified by the coordinate $\upsilon(\tau,z)$, as described near~\eqref{eq:sigmaor}. All other complex orientations arise from changing the coordinate for the corresponding formal group law, but changes of coordinate exactly correspond to changes of trivialization of $\mc{L}$ near the zero section of $E^\vee$, so all coordinates can be recovered this way. 
\ep

One can similarly ask for an equivariant refinement of the $\MU\langle 6\rangle$-orientation and $\MString$-orientation, namely as a class
\beq
\begin{array}{ccl}
\Gamma((E^\vee)^{\times 3};\Theta^3(\mathcal{O}(-0)))&\stackrel{{\rm completion}}{\longrightarrow}& \Ell(B\U(1)\times B\U(1)\times B\U(1))\\
\stackrel{\downin}{s_{\U(1)}} &\stackrel{?}{\mapsto} & \phantom{blvblahblahblah}\stackrel{\downin}{s} 
\end{array}\label{eq:ellcomplete3}
\eeq
lifting the class $s$ defined by~\eqref{eq:cubical} to a section of $\Theta^3(\mathcal{O}(-0))$ on $(E^\vee)^{\times 3}\simeq \Bun_{\U(1)^{\times 3}}(E)$. 

\begin{defn}
An \emph{equivariant refinement of the $\MO\langle 8\rangle$-orientation} is a $\Theta^3(\mathcal{O}(-0))$-twisted $\U(1)^{\times 3}$-equivariant elliptic cohomology class whose image under~\eqref{eq:ellcomplete3} is the Ando--Hopkins--Strickland characteristic class for the canonical $\MO\langle 8\rangle$-orientation.
\end{defn}

\begin{thm}\label{thm:cube}
Let~$E$ be a family of elliptic curves as in Definition~\ref{defn:twequiv}. 
\begin{enumerate}
\item There exists a unique equivariant refinement of the $\MO\langle 8\rangle$-orientation for complex analytic elliptic cohomology for the curve $E$. 
\item Furthermore, the equivariant refinement $s_{\U(1)}$ equals the twisted equivariant Euler class of Proposition~\ref{prop:twistedequiv} for the virtual vector bundle $V_3$ from~\eqref{eq:V3}.
\item In the universal case $E=\widetilde{\EE}$, the refinement $s_{\U(1)}$ descends to the stack $\Bun_{\U(1)^{\times 3}}(\EE)$. 
\end{enumerate}
\end{thm}

\bp
By inspection, the formulas~\eqref{eq:cubical} for the non-equivariant cubical structure have a unique equivariant extension given by the same formulas: when considered as a function on $\HH\times \C\times \C\times \C$, the formulas~\eqref{eq:cubical} give sections of $\Theta^3(\mathcal{O}(-0))$ on $\widetilde{\Euni}^\vee\times_\HH \widetilde{\Euni}^\vee\times_\HH \widetilde{\Euni}^\vee$. This gives the equivariant characteristic class~\eqref{eq:ellcomplete3} for~$V_3$. On inspection of the formulas, this equals the twisted equivariant Euler class of $V_3$, using the class from Proposition~\ref{prop:twistedequiv}. Finally, we observe that this cubical structure is invariant under the action of $\SL_2(\Z)$, and so descends to $\Euni^\vee \times_{\Mell} \Euni^\vee \times_{\Mell} \Euni^\vee$, and therefore is a global class for $\Ell^0_{\U(1)\times \U(1)\times \U(1)}(\pt)\otimes \Theta^3(\mathcal{O}(-0))$.
\ep

\begin{rmk} The uniqueness of a cubical structure for $\mc{O}(-0)$ on the elliptic curve produces a canonical $\MU\langle 6\rangle$-orientation of non-equivariant elliptic cohomology. However, there are possibly more cubical structures for $\mc{O}(-0)$ on the formal group, so this canonical class need not be unique. We find it striking that the cubical structure on $\mc{O}(-0)$ produces a \emph{unique} equivariant $\MU\langle 6\rangle$-orientation: the possible ambiguities on the formal group disappear in the equivariant refinement.
\end{rmk}

\appendix

\section{Background}

\subsection{Some Lie theory}\label{sec:technicalities}

A reference for the following results is~\cite{SegalRepRing}. 

\begin{lem} \label{lem:conj} Let $T<G$ be a maximal torus for a connected compact Lie group with normalizer $N(T)<G$. If $t_1, t_2 \in T$ are conjugate in~$G$, they are conjugate by an element of~$N(T)$. \end{lem}


Let $\mathfrak{t}$ be the Lie algebra of a maximal torus $T$ of a compact connected Lie group $G$, and $W=N(T)/T$ be the Weyl group. The following is proved in the same way as the above. 

\begin{cor}\label{cor:conj}
If $X_1,X_2\in \mf{t}$ are conjugate by the adjoint action of $G$ on $\mf{t}$, then they are conjugate by an element of $N(T)$. 
\end{cor}

\begin{prop}\label{prop:holoW} The ring of $W$-invariant holomorphic functions on $\mf{t}_{\mb{C}}$ is equivalent to the ring of $G$-invariant holomorphic functions on $\mf{g}_{\mb{C}}$. \end{prop}


\bp Any conjugation-invariant function on $\mf{g}_{\mb{C}}$ clearly restricts to a $W$-invariant function on $\mf{t}_{\mb{C}}$; the interesting direction is to extend a $W$-invariant function on $\mf{t}_{\mb{C}}$ to a $G$-invariant function on $\mf{g}_{\mb{C}}$. On the (Zariski) open sublocus $\mf{g}^{\rm rs}_{\mb{C}}$ of regular semisimple elements, any element by definition may be conjugated into $\mf{t}_{\mb{C}}$, so that a holomorphic function on $\mf{t}_\C$ can automatically be extended to a holomorphic function on~$\mf{g}^{\rm rs}_{\mb{C}}$. By Corollary~\ref{cor:conj}, the extension is conjugation invariant if the original function is $W$-invariant. It remains to extend further to $\mf{g}_{\mb{C}}$ (which would automatically continue to be conjugation-invariant). But we may approximate a holomorphic $W$-invariant function on $\mf{t}_{\mb{C}}$ by a $W$-invariant polynomial on $\mf{t}_{\mb{C}}$ and instead simply have to extend a polynomial from $\mf{g}^{\rm rs}_{\mb{C}}$ to all of $\mf{g}_\C$. By Algebraic Hartogs' Lemma, the polar locus is a closed subset of pure codimension one. However, the closures of all codimension one points of $\mf{g}_{\mb{C}} \setminus \mf{g}^{\rm rs}_{\mb{C}}$ contain $0$, where our polynomial is clearly well-defined, and so the polar locus must be empty and we have a polynomial extension, as desired.\ep

\begin{rmk} The same result holds, with the same proof, replacing holomorphic functions in Proposition~\ref{prop:holoW} with germs of holomorphic functions. \end{rmk}

%
%

\begin{prop} \label{prop:conjugategeneral} Let $G$ be a compact Lie group, not necessarily connected. Given $h \in G$ and $X, X' \in \mf{g}^h$ sufficiently small, the set of elements which conjugates $he^X$ to $he^{X'}$ is contained in $C(h)$, the centralizer of~$h$. \end{prop}

\bp Let $S = \{g \in G | ghe^Xg^{-1} = he^{X'}\}$; by construction, it is a coset of $C(he^X)$. By~\cite[Lemma~1.3]{BlockGetzler}, we may assume $X$ is sufficiently small such that $C(he^X) \subset C(h)$  (compare Lemma~\ref{lem:BlockGetzler} above). Hence either $S \subset C(h)$, as desired, or $S$ is entirely disjoint from $C(h)$. Choose a faithful representation $G \hookrightarrow \U(n)$ and assume for now the result for $\U(n)$. Then $S \subset C_{\U(n)}(h)$, where $C_{\U(n)}(h) \subset \U(n)$ is the subgroup of $\U(n)$ which centralizes $h$. But then $S \subset G \cap C_{U(n)}(h) = C_G(h)$, as desired. Hence, it suffices to show the result for $G = \U(n)$.

The statement is clearly invariant under conjugation, so we may assume $h$ is diagonal and of some block-form for a partition $n = n_1 + \cdots + n_k$, where $h$ has distinct eigenvalues $\lambda_1, \cdots, \lambda_k$ with each eigenvalue $\lambda_i$ occurring with multiplicity $n_i$. Then $C(h)$ is the corresponding group of block-diagonal matrices. Pick disjoint open intervals $U_i$ centered at the $\lambda_i$ and interpret ``sufficiently small'' to mean that the eigenvalues of the $i^{\rm th}$ block of $he^X, he^{X'}$ remain within $U_i$. Then one may show directly any element conjugating $he^X$ to $he^{X'}$ must be block-diagonal, i.e., lie in $C(h)$, as desired. \ep

\begin{cor} \label{prop:conjugate} For $G, h, X, X'$ as in the previous proposition, the set of elements of $G_0$ (the identity connected component) which conjugate $he^X$ to $he^{X'}$ is contained in $G_0^h$. \end{cor}

%

\begin{lem} \label{lem:centslicesprep} Let $G$ be a compact Lie group. For a fixed $g \in G$, consider the adjoint action $\Ad_g: \mf{g} \to \mf{g}$. Define $A_g = \Ad_g - \id$ so that $\ker A_g = \mathrm{Lie}(C(g))$. Then ${\rm image} A_g \cap \ker A_g = 0$. \end{lem}

\begin{proof} We wish to show $\ker A_g^2 = \ker A_g$, i.e., the generalized eigenspace of $\Ad_g$ with eigenvalue $1$ is in fact just a usual eigenspace. But this follows from $\Ad_g$ being self-adjoint with respect to the nondegenerate Killing form, so that all generalized eigenspaces of $\Ad_g$ are usual eigenspaces. \end{proof}

\begin{lem} \label{lem:centslices} Given $G$ as above and $g \in G$, for any element $X \in \mf{g}$ sufficiently small, there exists some small $Y \in \mf{g}^g$ such that $ge^X$ is conjugate to $ge^Y$. \end{lem}

\bp It suffices to prove the above infinitesimally, i.e., to show that on the tangent space $T_gG \simeq \mf{g}$ as identified with the Lie algebra by left-translation under $g^{-1}$, the orbit of $\mf{g}^g$ under the ($g$-twisted) adjoint action spans the full tangent space. But indeed, the centralizer $\mf{g}^g$ is exactly $\ker A_g$ as in the previous lemma, while the infinitesimal adjoint action under conjugacy spans ${\rm image}A_g$. The prior lemma plus a simple dimension count yields that $\mf{g} \simeq \ker A_g \oplus {\rm image}A_g$, i.e., the full tangent space is spanned by the centralizer and infinitesimal deformations under conjugacy, which is what we wished to show. \ep

We recall from Definition~\ref{defn:C2G} that $\mathcal{C}^2(G)\subset G\times G$ is the subsheaf of pairs of commuting elements in~$G$, and we use the notation $h=(h_1,h_2)\in \mathcal{C}^2(G)$ to denote an element of the set $h\in \mathcal{C}^2(G)(\pt)$, i.e., $(h_1,h_2)\in G\times G$ are just a pair of commuting elements in $G$. 

\begin{lem} \label{lem:ellcentslices} Given $G$ as above and $(h_1, h_2) \in \mathcal{C}^2(G)$, suppose $X_1, X_2 \in \mf{g}$ sufficiently small are such that $(h_1e^{tX_1}, h_2e^{tX_2}) \in \mathcal{C}^2(G)$ for $0 < t < 1$. Then there exists some small $Y_1, Y_2 \in \mf{t}_{\mf{g}^h}$ such that $he^X$ is conjugate to $he^Y$. \end{lem}

\bp The commutation hypothesis is equivalent to $X_1 \in \mf{g}^{h_2}, X_2 \in \mf{g}^{h_1}, [X_1, X_2] = 0$. Hence, we may first use the above Lemma~\ref{lem:centslices} with $G = C(h_2)$ to conjugate $X_1$ into $Y_1 \in \mf{g}^{h_1}$; as it is still also in $\mf{g}^{h_2}$, it is in $\mf{g}^h$; let us suppose $X_2$ has now been conjugated to some $X'_2$. We may then apply Lemma~\ref{lem:centslices} with $G = C(h_1, Y_1)$ to conjugate $X'_2$ further into $\mf{g}^{h_2}$ and hence also lie in $\mf{g}^h$; note that $Y_1$ stays fixed under this further conjugation. Finally, as $[X_1, X_2] = 0$, we also have that $Y_1, Y_2$ commute. We may hence simultaneously conjugate them from $\mf{g}^h$ into a Cartan (maximal abelian subalgebra) $\mf{t}_{\mf{g}^h}$ as they are certainly inside some Cartan, and all Cartans are conjugate. \ep

%
%

\begin{lem}\label{lem:lcrank} For $h\in \mathcal{C}^2(G)$, the function $h \mapsto \text{rank }G_0^h$ is locally constant. \end{lem}

\bp This follows from Lemma~\ref{lem:ellcentslices} and Lemma~\ref{lem:BlockGetzler}: it suffices to check local constancy as $h$ is deformed to $he^{\epsilon X}$ for $X \in \mf{t}_{\mf{g}^h}$ and $\epsilon$ sufficiently small, whereupon one may take $T_{G^h_0}$ to be a maximal torus for both $G_0^h$ and $G_0^{he^{\epsilon X}}$. \ep

\subsection{The Cartan model for equivariant cohomology}\label{sec:equivdeRham}

The equivariant cohomology of a manifold with $G$-action is defined by the Borel construction,
$$
\H_G(M):=\H(M\times_G EG),
$$
where above $\H(-)$ denotes ordinary cohomology with complex coefficients. By naturality, $\H_G(M)$ is a module over $\H_G(\pt)=\H(BG)$. The following standard facts will be useful; cohomology is taken with complex coefficients, $\H(-)=\H(-;\C)$.

\begin{lem}\label{lem:howuseful} For $G$ connected there is a natural isomorphism~$\H_G(M)\simeq \H_T(M)^W$ for any maximal torus $T<G$ with Weyl group $W=N(T)/T$. 
\end{lem}
\begin{lem} \label{lem:howuseful2}For $H<G$ a normal subgroup of finite index, there is a natural isomorphism $\H_G(M)\simeq \H_H(M)^{G/H}$. 
\end{lem}

The Cartan model for equivariant cohomology starts with the graded algebra $\Omega^{\bullet}_G(M):= \Sym(\mf{g}_{\mb{C}};\Omega^\bullet(M))^G$, where the polynomial generators in $\mathfrak{g}^\vee_\C \subset \Sym(\mf{g}_{\mb{C}}^\vee)$ have degree~2 and differential forms have their standard degree. We identify elements of $\Sym(\mf{g}_{\mb{C}};\Omega^\bullet(M))^G$ with $G$-invariant polynomial functions on $\mathfrak{g}_{\mb{C}}$ valued in $\Omega^\bullet(M)$. In this description, define a differential~$Q$ on $\Omega^{\bullet}_G(M)$ 
\beq
(Q\alpha)(X)=d(\alpha(X))-\iota_X\alpha(X),\qquad X\in \mathfrak{g}_\C\quad \alpha\in \Omega^{\bullet}_G(M)\label{eq:Cartandiff}
\eeq
extended complex-linearly, where $d$ is the ordinary de~Rham differential on forms, and $\iota_X$ denotes contraction with the vector field on~$M$ associated to~$X$ under the infinitesimal action of~$G$ on~$M$. One verifies that $Q^2=0$ on $G$-invariants using Cartan's magic formula. The chain complex $(\Omega^{\bullet,{\rm pol}}_G(M),Q)$ is the \emph{Cartan model} for equivariant cohomology, and we have an isomorphism
$$
\H((\Omega^{\bullet}_G(M)),Q)\simeq \H_G(M).
$$
We refer to~\cite{MeinrenkenCartan} for an excellent introduction to equivariant cohomology in the Cartan model. 



\subsection{Lie groupoids, sheaves, and smooth stacks}\label{appen:stacks}

Let ${\sf Mfld}$ denote the category of manifolds and smooth maps. A \emph{Lie groupoid}, denoted $\{G_1 \rightrightarrows G_0\}$, consists of a manifold of objects,~$G_0$, a manifold of morphisms,~$G_1$, source and target maps, $s,t\colon G_1\to G_0$, a unit map $G_0\to G_1$, and a composition map $c\colon G_1\times_{G_0}G_1\to G_1$. We further require that $s$ is a submersion so that the fibered product $G_1\times_{G_0}G_1$ exists in manifolds. These data are required to satisfy the usual axioms of a groupoid. 

\begin{ex}
Let a Lie group~$G$ act on a manifold~$M$. The \emph{action Lie groupoid}, denoted~$M\sq G$, has~$M$ as objects and $G\times M$ as morphisms. The source map $s\colon G\times M\to M$ is the projection, and the target map $t\colon G\times M\to M$ is the action map. The unit $M\to G\times M$ is the inclusion along the identity element $e\in G$. Composition is inherited from multiplication in $G$. 
\end{ex}

A \emph{presheaf} is a functor $F\colon {\sf Mfld}^\op \to {\sf Sets}$. A presheaf $F$ is a \emph{sheaf} if for all open covers $\{U_i\}$ of all manifolds $S$, the diagram
$$
F(S)\to \prod_i F(U_i)\rightrightarrows \prod_{i,j} F(U_i\bigcap U_j)
$$
is an equalizer. The set $F(S)$ are the \emph{$S$-points} of the (pre)sheaf~$F$. A (pre)sheaf is \emph{representable} when its values are determined by the set of maps to a fixed smooth manifold, $F(S)=\Map(S,N)$, $N\in {\sf Mfld}$. Note that a representable presheaf is a sheaf. When working with the functor of points, we will frequently use the same notation to denote a smooth manifold and its representable sheaf so that, e.g., $N(S)=\Map(S,N)$ is the $S$-points on~$N$. The following examples indicate the flavors of non-representable presheaves that appear in the body of the paper, namely sub-objects and (coarse) quotients. 


\begin{ex}\label{ex:subobject}
Given a smooth manifold $Z$, let $Y\subset Z$ be a subset (not necessarily a smooth submanifold). Define a presheaf whose $S$-points are maps $S\to Z$ with image in the subset $Y\subset Z$. It is easy to check that this presheaf is in fact a sheaf. In a mild abuse of notation, we usually denote the presheaf defined above by~$Y$. 
\end{ex}

\begin{ex}\label{ex:coarsequotient} Given a $G$-manifold $M$, define the coarse quotient presheaf $(M\cq G)^{\mathrm{pre}}$ as having $S$-points the set $M(S)/G(S)$; explicitly, these are $S$-points of $M$ subject to the equivalence relation that a pair of maps $f,f'\colon S\to M$ are equivalent if there is $g\colon S\to G$ such that $f'=g\cdot f$ using the $G$-action on $M$ on the right hand side. Define the coarse quotient sheaf, denoted $M\cq G$, as the sheafification of the presheaf $(M\cq G)^{\mathrm{pre}}$. \end{ex}

\begin{defn}\label{defn:genLiegrpd}
Define a \emph{generalized Lie groupoid} as a groupoid objects in presheaves on manifolds, denoted $\{ G_1 \rightrightarrows G_0 \}$. Explicitly, the data of a generalized Lie groupoid consists of presheaves $G_1,G_0$, and maps of presheaves called source, target, unit, and composition. These data are required to satisfy the properties of a functor ${\sf Mfld}^\op \to {\sf Grpd}$ from manifolds to groupoids given by~$S \mapsto \{ G_1(S) \rightrightarrows G_0(S) \}$. 
\end{defn}





\begin{defn}
A \emph{smooth stack} is a category fibered in groupoids over manifolds satisfying descent with respect to open covers. 
\end{defn}

For each manifold~$S$ a stack assigns a groupoid, and to each map $S\to S'$, a stack assigns a functor between groupoids. These data can be assembled into a weak 2-functor from manifolds to groupoids. A weak 2-functor from manifolds to groupoids that doesn't necessarily satisfy descent is called a \emph{prestack}. {\it Stackification} is the left adjoint to the forgetful functor from stacks to prestacks.

\begin{ex}
Any presheaf (of sets) on the site of smooth manifolds determines a prestack, and any sheaf determines a stack. Indeed, there is a faithful embedding of sheaves into stacks. In particular, smooth manifolds (regarded as representable sheaves) embed into smooth stacks. We often use the same notation, e.g., $N$, to denote a smooth manifold, its representable sheaf, and the associated smooth stack. 
\end{ex}

A reference for the relationship between Lie groupoids and stacks is~\cite[\S1]{BehrendXu}. We briefly review some of the highlights. 

\begin{ex} \label{ex:prestack}The $S$-points of a generalized Lie groupoid $G=\{G_1\rightrightarrows G_0\}$ define a prestack whose value on $S$ is the groupoid $\{G_1(S)\rightrightarrows G_0(S)\}$. All the stacks in this paper come from applying stackification to prestacks of this form. We use the notation~$[G_1\rightrightarrows G_0]$ or $[G_0/G_1]$ to denote the stackification of the prestack $\{G_1\rightrightarrows G_0\}$.  \end{ex}

\begin{ex}\label{ex:stackquotient} Given a $G$-manifold $M$, the quotient stack $[M\sq G]$ is the stack underlying the action Lie groupoid $M\sq G$. Explicitly, a map $S\to [M\sq G]$ is the data of a pair $(P, \sigma)$, where $P\to S$ is a principal $G$-bundle on $S$, and $\sigma\colon P\to M$ is a $G$-equivariant map. Isomorphisms between $S$-points $(P,\sigma)\Rightarrow (P',\sigma')$ are isomorphisms of principal bundles $P\to P'$ compatible with the $G$-equivariant maps to $M$.\end{ex}

\begin{rmk}
Note that there is always a map (in the category of smooth stacks) $[M\sq G] \to M\cq G$ from the stack quotient to the coarse quotient sheaf. This is an isomorphism (in the category of stacks) if and only if the $G$-action on $M$ is free so that the sheaf $M\cq G$ is representable. 
\end{rmk}

\begin{defn} A \emph{groupoid presentation} of a stack $\X$ is a Lie groupoid $\{G_1\rightrightarrows G_0\}$ whose underlying stack is equivalent to~$\X$, i.e., $\X\simeq [G_1\rightrightarrows G_0]$. When such a presentation exists, $\X$ is a \emph{differentiable stack}. \end{defn}




\begin{defn}\label{defn:atlas}
An \emph{atlas} for a stack $\X$ is a map $p\colon U\to \X$ whose source is a manifold~$U$ with the property that for any other map $q\colon V\to \X$ whose source is a manifold $V$, the 2-fibered product $U\times_\X V$ is representable by a smooth manifold, and the map $U\times_\X V\to V$ is a submersion of manifolds. 
\end{defn}

\begin{lem}
An atlas $U\to \X$ defines the groupoid presentation, $\{U\times_\X U\rightrightarrows U\}$ where all the structure maps in the groupoid are constructed from the universal property of the pullback. Hence a stack has a Lie groupoid presentation if and only if it admits an atlas. 
\end{lem}
%


Finally, we will construct holomorphic structures on stacks in terms of holomorphic atlases, defined as follows. 

\begin{defn}\label{defn:holomorphic}
A \emph{holomorphic atlas} is an atlas $U\to \X$ where $U$ and $U\times_\X U$ are given the structure of a complex manifold and all the structure maps in the groupoid $\{U\times_\X U\rightrightarrows U\}$ are holomorphic. Given a smooth stack $\X$, a choice of holomorphic atlas $U\to \X$ is a \emph{choice of holomorphic structure}, and $\X$ with this fixed choice is a \emph{complex analytic stack}. 
\end{defn}

\begin{ex}
Suppose that a discrete group $G$ acts on a complex manifold $Z$ preserving the complex structure. Then $Z\to [Z\sq G]$ is a holomorphic atlas for the quotient stack. 
\end{ex}

\bibliographystyle{amsalpha}
\bibliography{references}

\end{document}